\documentclass[11pt,a4paper]{article}
\usepackage[utf8]{inputenc}
\usepackage[T1]{fontenc}
\usepackage{amsmath}
\usepackage{amsthm}
\usepackage{amsfonts}
\usepackage{amssymb}
\usepackage{mathrsfs}
\usepackage{url}
\usepackage{mathdots}
\usepackage{geometry}
\usepackage{verbatim}
\usepackage{hyperref}
\usepackage{float}

\hypersetup{
    bookmarks=true,         
    unicode=false,          
    pdftoolbar=true,        
    pdfmenubar=true,        
    pdffitwindow=false,     
    pdfstartview={FitH},    
    pdftitle={Dimensions of spaces of level one automorphic forms for split classical groups using the trace formula},    
    pdfauthor={Olivier Taïbi},     
    colorlinks=true,       
    linkcolor=blue,          
    citecolor=green,        
    filecolor=green,      
    urlcolor=cyan}           

\newcommand{\Q}{\mathbb{Q}}
\newcommand{\A}{\mathbb{A}}
\newcommand{\R}{\mathbb{R}}
\providecommand{\C}{\mathbb{C}}
\renewcommand{\C}{\mathbb{C}}

\newcommand{\Qp}{\mathbb{Q}_p}

\newcommand{\Qbar}{\overline{\mathbb{Q}}}

\newcommand{\Fp}{\mathbb{F}_p}
\newcommand{\F}{\mathbb{F}}

\newcommand{\Zp}{\mathbb{Z}_p}
\newcommand{\Z}{\mathbb{Z}}
\newcommand{\GL}{\mathrm{GL}}
\newcommand{\SL}{\mathrm{SL}}
\newcommand{\SO}{\mathrm{SO}}
\newcommand{\End}{\mathrm{End}}
\newcommand{\Hom}{\mathrm{Hom}}
\newcommand{\Gal}{\mathrm{Gal}}
\newcommand{\Res}{\mathrm{Res}}
\newcommand{\Ind}{\mathrm{Ind}}

\newtheorem{theo}{Theorem}[subsection]

\newtheorem{lemm}[theo]{Lemma}

\newtheorem{assu}[theo]{Assumption}
\newtheorem{coro}[theo]{Corollary}

\newtheorem{prop}[theo]{Proposition}
\newtheorem{rema}[theo]{Remark}
\newtheorem{exam}[theo]{Example}

\numberwithin{equation}{subsection}

\begin{document}

\baselineskip=16pt

\author{Olivier Taïbi}
\title{Dimensions of spaces of level one automorphic forms for split classical groups using the trace formula}
\date{June 16, 2014}

\maketitle
\begin{abstract}
We consider the problem of explicitly computing dimensions of spaces of automorphic or modular forms in level one, for a split classical group $\mathbf{G}$ over $\mathbb{Q}$ such that $\mathbf{G}(\R)$ has discrete series.

Our main contribution is an algorithm calculating orbital integrals for the characteristic function of $\mathbf{G}(\mathbb{Z}_p)$ at torsion elements of $\mathbf{G}(\mathbb{Q}_p)$.
We apply it to compute the geometric side in Arthur's specialisation of his invariant trace formula involving stable discrete series pseudo-coefficients for $\mathbf{G}(\mathbb{R})$.
Therefore we explicitly compute the Euler-Poincaré characteristic of the level one discrete automorphic spectrum  of $\mathbf{G}$ with respect to a finite-dimensional representation of $\mathbf{G}(\mathbb{R})$.

For such a group $\mathbf{G}$, Arthur's endoscopic classification of the discrete spectrum allows to analyse precisely this Euler-Poincaré characteristic.
For example one can deduce the number of everywhere unramified automorphic representations $\pi$ of $\mathbf{G}$ such that $\pi_{\infty}$ is isomorphic to a given discrete series representation of $\mathbf{G}(\mathbb{R})$.
Dimension formulae for the spaces of vector-valued Siegel modular forms are easily derived.
\end{abstract}

\newpage
\tableofcontents
\newpage

\section{Introduction}

Let $\mathbf{G}$ be a Chevalley reductive group over $\Z$ admitting discrete series at the real place, i.e.\ one of $\mathbf{SO}_{2n+1}$, $\mathbf{Sp}_{2n}$ or $\mathbf{SO}_{4n}$ for $n \geq 1$.
We give an algorithm to compute the geometric side in Arthur's ``simple'' trace formula in \cite{ArthurL2} (see also \cite{GoKoMPh}) for $\mathbf{G}$ and the trivial Hecke operator in level one at the finite places, that is the characteristic function of $\mathbf{G}(\widehat{\Z})$.
There are essentially three steps to compute the geometric side of the trace formula:
\begin{enumerate}
\item for any prime $p$, compute the local orbital integrals of the characteristic function of $\mathbf{G}(\Zp)$ at torsion elements $\gamma_p$ in $\mathbf{G}(\Qp)$ (with respect to a Haar measure on the connected centraliser of $\gamma_p$),
\item for any semisimple elliptic and torsion conjugacy class $\gamma \in \mathbf{G}(\Q)$ with connected centraliser $\mathbf{I}$, use the Smith-Minkowski-Siegel mass formula to compute $\mathrm{Vol}(\mathbf{I}(\Q) \backslash \mathbf{I}(\A))$,
\item analyse the character of stable (averaged) discrete series on arbitrary maximal tori of $\mathbf{G}(\R)$ to express the parabolic terms using elliptic terms for groups of lower semisimple rank.
\end{enumerate}
We explain how to compute local orbital integrals for special orthogonal groups (resp.\ symplectic groups) in sections \ref{conjclassZpSection} and \ref{compOrbIntSection}, using quadratic and hermitian (resp.\ alternate and antihermitian) lattices over cyclotomic extensions of $\Zp$.
To compute the volumes appearing in local orbital integrals we rely on the local density formulae for such lattices given in \cite{GanYu}, \cite{ChoQuad} and \cite{ChoHermi}.
We use the formulation of \cite{GrossMot} for the local and global volumes (see section \ref{sectionVolumes}).
For the last step we follow \cite{GoKoMPh}, and we only add that for the trivial Hecke operator the general formula for the archimedean factor of each parabolic term simplifies significantly (Proposition \ref{charavDScpct}).
Long but straightforward calculations lead to explicit formulae for the parabolic terms (see section \ref{explicitpara}).

Thus for any irreducible algebraic representation $V_{\lambda}$ of $\mathbf{G}_{\C}$ characterised by its highest weight $\lambda$, we can compute the spectral side of the trace formula, which we now describe.
Let $K_{\infty}$ be a maximal compact subgroup of $\mathbf{G}(\R)$ and let $\mathfrak{g} = \C \otimes_{\R} \mathfrak{g}_0$ where $\mathfrak{g}_0 = \mathrm{Lie}(\mathbf{G}(\R))$.
For an irreducible $(\mathfrak{g}, K_{\infty})$-module $\pi_{\infty}$, consider the Euler-Poincaré characteristic
$$ \mathrm{EP}(\pi_{\infty} \otimes V_{\lambda}^*) = \sum_i (-1)^i\dim H^i \left((\mathfrak{g},K_{\infty}),\ \pi_{\infty} \otimes V_{\lambda}^* \right) $$
where $V_{\lambda}$ is seen as a representation of $\mathbf{G}(\R)$.
Let $\Pi_{\mathrm{disc}}(\mathbf{G})$ be the set of isomorphism classes of irreducible $(\mathfrak{g}, K_{\infty}) \times \mathbf{G}(\A_f)$-modules occurring in the discrete automorphic spectrum of $\mathbf{G}$.
For $\pi \in \Pi_{\mathrm{disc}}(\mathbf{G})$ denote by $m_{\pi} \in \Z_{\geq 1}$ the corresponding multiplicity.
Let $\Pi^{\mathrm{unr}}_{\mathrm{disc}}(\mathbf{G})$ be the set of $\pi \in \Pi_{\mathrm{disc}}(\mathbf{G})$ which are unramified at all the finite places of $\Q$.
For any dominant weight $\lambda$ the set of $\pi \in \Pi^{\mathrm{unr}}_{\mathrm{disc}}(\mathbf{G})$ such that $H^{\bullet}((\mathfrak{g}, K_{\infty}), \pi_{\infty} \otimes V_{\lambda}^*) \neq 0$ is finite.
The spectral side of Arthur's trace formula in \cite{ArthurL2} for our choice of function at the finite places is
\begin{equation} \label{introspecside} \sum_{\pi \in \Pi^{\mathrm{unr}}_{\mathrm{disc}}(\mathbf{G})} m_{\pi} \mathrm{EP}(\pi_{\infty} \otimes V_{\lambda}^*). \end{equation}
This integer is interesting but it is only an alternate sum.
To obtain subtler information, e.g.\ the sum of $m_{\pi}$ for $\pi_{\infty}$ isomorphic to a given $(\mathfrak{g}, K_{\infty})$-module, we use Arthur's endoscopic classification of the discrete automorphic spectrum for symplectic and special orthogonal groups \cite{Arthur}.
Arthur's work allows to parametrise the representations $\pi$ contributing to the spectral side \ref{introspecside} using self-dual automorphic representations for general linear groups.
Denote $W_{\R}$ the Weil group of $\R$ and $\epsilon_{\C / \R}$ the character of $W_{\R}$ having kernel $W_{\C} \simeq \C^{\times}$.
For $w \in \frac{1}{2}\Z$ define the bounded Langlands parameter $I_w : W_{\R} \rightarrow \mathrm{GL}_2(\C)$ as
$$ \mathrm{Ind}_{W_{\C}}^{W_{\R}} \left( z \mapsto (z/|z|)^{2w} \right) $$
so that $I_0 \simeq 1 \oplus \epsilon_{\C / \R}$.
The three families that we are led to consider are the following.
\begin{enumerate}
\item For $n \geq 1$ and $w_1, \dots, w_n \in \frac{1}{2}\Z \smallsetminus \Z$ such that $w_1 > \dots > w_n > 0$, define $S(w_1, \dots, w_n)$ as the set of self-dual automorphic cuspidal representations of $\mathbf{GL}_{2n} / \Q$ which are unramified at all the finite places and with Langlands parameter at the real place
$$ I_{w_1} \oplus \dots \oplus I_{w_n}. $$
Equivalently we could replace the last condition by ``with infinitesimal character having eigenvalues $\{ \pm w_1, \dots, \pm w_n \}$''.
Here $S$ stands for ``symplectic'', as the conjectural Langlands parameter of such a representation should be symplectic.
\item For $n \geq 1$ and integers $w_1 > \dots > w_n > 0$ define $O_o(w_1, \dots, w_n)$ as the set of self-dual automorphic cuspidal representations of $\mathbf{GL}_{2n+1} / \Q$ which are everywhere unramified and with Langlands parameter at the real place
$$ I_{w_1} \oplus \dots \oplus I_{w_n} \oplus \epsilon_{\C / \R}^n. $$
Equivalently we could replace the last condition by ``with infinitesimal character having eigenvalues $\{ \pm w_1, \dots, \pm w_n , 0\}$''.
Here $O_o$ stands for ``odd orthogonal''.
\item For $n \geq 1$ and integers $w_1 > \dots > w_{2n-1} > w_{2n} \geq 0$ define $O_e(w_1, \dots, w_{2n})$ as the set of self-dual automorphic cuspidal representations of $\mathbf{GL}_{4n} / \Q$ which are everywhere unramified and with Langlands parameter at the real place
$$ I_{w_1} \oplus \dots \oplus I_{w_{2n}}. $$
In this case also we could replace the last condition by ``with infinitesimal character having eigenvalues $\{ \pm w_1, \dots, \pm w_{2n} \}$'', even in the slightly singular case where $w_{2n}=0$.
Here $O_e$ stands for ``even orthogonal''.
\end{enumerate}
Following Arthur using these three families we can define, for any $\mathbf{G}$ and $\lambda$ as above, a set $\Psi(\mathbf{G})^{\mathrm{unr}, \lambda}$ of ``formal Arthur-Langlands parameters'' which parametrises the representations $\pi \in \Pi_{\mathrm{disc}}^{\mathrm{unr}}(\mathbf{G})$ contributing to \ref{introspecside}.
We stress that for a given $\mathbf{G}$ all three families take part in these formal parameters.
Among these formal parameters, one can distinguish a subset $\Psi(\mathbf{G})^{\mathrm{unr}, \lambda}_{\mathrm{sim}}$ of ``simple'' parameters, that is the tempered and non-endoscopic ones.
When $\mathbf{G} = \mathbf{SO}_{2n+1}$ (resp.\ $\mathbf{Sp}_{2n}$, resp.\ $\mathbf{SO}_{4n}$), this set is exactly $S(w_1, \dots, w_n)$ (resp.\ $O_o(w_1, \dots, w_n)$, resp.\ $O_o(w_1, \dots, w_{2n})$) where $(w_i)_i$ is determined by $\lambda$.
The contribution of any element of $\Psi(\mathbf{G})^{\mathrm{unr}, \lambda}_{\mathrm{sim}}$ to the spectral side \ref{introspecside} is a non-zero number depending only on $\mathbf{G}(\R)$.
Therefore it is natural to attempt to compute the cardinalities of the sets $S(\cdot)$, $O_o(\cdot)$ and $O_e(\cdot)$ inductively, the induction being on the dimension of $\mathbf{G}$.
More precisely we have to compute the contribution of $\Psi(\mathbf{G})^{\mathrm{unr}, \lambda} \smallsetminus \Psi(\mathbf{G})^{\mathrm{unr}, \lambda}_{\mathrm{sim}}$ to \ref{introspecside} to deduce the cardinality of $\Psi(\mathbf{G})^{\mathrm{unr}, \lambda}_{\mathrm{sim}}$.

When the highest weight $\lambda$ is regular, any element of $\Psi(\mathbf{G})^{\mathrm{unr}, \lambda}$ is tempered and consequently any $\pi \in \Pi_{\mathrm{disc}}^{\mathrm{unr}}(\mathbf{G})$ contributing to the spectral side is such that $\pi_{\infty}$ is a discrete series representation having same infinitesimal character as $V_{\lambda}$.
Thanks to the work of Shelstad on real endoscopy and using Arthur's multiplicity formula it is not difficult to compute the contribution of $\Psi(\mathbf{G})^{\mathrm{unr}, \lambda} \smallsetminus \Psi(\mathbf{G})^{\mathrm{unr}, \lambda}_{\mathrm{sim}}$ to the Euler-Poincaré characteristic on the spectral side in this case (see section \ref{sectionDS}).
The general case is more interesting because we have to consider non-tempered representations $\pi_{\infty}$.
Since Arthur's construction of non-tempered Arthur packets at the real place in \cite{Arthur} is rather abstract, we have to make an assumption (see Assumption \ref{assumAJweak}) in order to be able to compute explicitly the non-tempered contributions to the Euler-Poincaré characteristic.
This assumption is slightly weaker than the widely believed Assumption \ref{assumAJ}, which states that the relevant real non-tempered Arthur packets at the real place coincide with those constructed long ago by Adams and Johnson in \cite{AdJo}.

Thus we obtain an algorithm to compute the cardinalities of the sets $S(w_1, \dots, w_n)$, $O_o(w_1, \dots, w_n)$ and $O_e(w_1, \dots, w_{2n})$, under assumption \ref{assumAJweak} when $\lambda$ is singular.
For the computer the hard work consists in computing local orbital integrals.
Our current implementation, using Sage \cite{sage}, allows to compute them at least for $\mathrm{rank}(\mathbf{G}) \leq 6$.
See section \ref{tablestwGL} for some values.

Once these cardinalities are known we can \emph{count} the number of $\pi \in \Pi_{\mathrm{disc}}^{\mathrm{unr}}(\mathbf{G})$ such that $\pi_{\infty}$ is isomorphic to a given $(\mathfrak{g}, K_{\infty})$-module having same infinitesimal character as $V_{\lambda}$ for some highest weight $\lambda$.
A classical application is to compute dimensions of spaces of (vector-valued) Siegel cusp forms.
For a genus $n \geq 1$ and $m_1 \geq \dots \geq m_n \geq n+1$, let $r$ be the holomorphic (equivalently, algebraic) finite-dimensional representation of $\GL_n(\C)$ with highest weight $(m_1, \dots, m_n)$.
Let $\Gamma_n = \mathrm{Sp}_{2n}(\Z)$.
The dimension of the space $S_r(\Gamma_n)$ of level one vector-valued cuspidal Siegel modular forms of weight $r$ can then be computed using Arthur's endoscopic classification of the discrete spectrum for $\mathbf{Sp}_{2n}$.
We emphasise that this formula depends on Assumption \ref{assumAJ} when the $m_k$'s are not pairwise distinct, in particular when considering scalar-valued Siegel cusp forms, of weight $m_1 = \dots = m_n$.
Our current implementation yields a dimension formula for $\dim S_r(\Gamma_n)$ for any $n \leq 7$ and any $r$ as above, although for $n \geq 3$ it would be absurd to print this huge formula.
See the table in section \ref{scalarSiegel} for some values in the scalar case.
The case $n=1$ is well-known: $\bigoplus_{m \geq 0} M_m(\Gamma_1) = \C[E_4,E_6]$ where the Eisenstein series $E_4,E_6$ are algebraically independant over $\C$, and the dimension formula for $S_m(\Gamma_1)$ follows.
Igusa \cite{Igusa} determined the ring of scalar Siegel modular forms and its ideal of cusp forms when $n=2$, which again gives a dimension formula.
Tsushima \cite{TsuExp}, \cite{TsuPf} gave a formula for the dimension of $S_r(\Gamma_2)$ for almost all representations $r$ as above (that is for $m_1 > m_2 \geq 5$ or $m_1 = m_2 \geq 4$) using the Riemann-Roch-Hirzebruch formula along with a vanishing theorem.
It follows from Arthur's classification that Tsushima's formula holds for any $(m_1,m_2)$ such that $m_1>m_2\geq 3$.
In genus $n=3$ Tsuyumine \cite{Tsuyumine} determined the structure of the ring of scalar Siegel modular forms and its ideal of cusp forms.
Recently Bergström, Faber and van der Geer \cite{BFG} studied the cohomology of certain local systems on the moduli space $\mathcal{A}_3$ of principally polarised abelian threefolds, and conjectured a formula for the Euler-Poincaré characteristic of this cohomology (as a motive) in terms of Siegel modular forms.
They are able to derive a conjectural dimension formula for spaces of Siegel modular cusp forms in genus three.
Our computations corroborate their conjecture, although at the moment we have only compared values and not the formulae.

Of course the present work is not the first one to attempt to use the trace formula to obtain spectral information, and we have particularly benefited from the influence of \cite{GrossPollack} and \cite{ChRe}.
In \cite{GrossPollack} Gross and Pollack use a simpler version of the trace formula, with hypotheses at a finite set $S$ of places of $\Q$ containing the real place and at least one finite place.
This trace formula has only elliptic terms.
They use the Euler-Poincaré function defined by Kottwitz in \cite{KoTam} at the finite places in $S$.
These functions have the advantage that their orbital integrals were computed conceptually by Kottwitz.
At the other finite places, they compute the stable orbital integrals indirectly, using computations of Lansky and Pollack \cite{LanskyPollack} for inner forms which are compact at the real place.
They do so for the groups $\mathbf{SL}_2$, $\mathbf{Sp}_4$ and $\mathbf{G}_2$.
Without Arthur's endoscopic classification it was not possible to deduce the number of automorphic representations of a given type from the Euler-Poincaré characteristic on the spectral side, even for a regular highest weight $\lambda$.
The condition $\mathrm{card}(S) \geq 2$ forbids the study of \emph{level one} automorphic representations.
More recently, Chenevier and Renard \cite{ChRe} computed dimensions of spaces of level one \emph{algebraic} automorphic forms in the sense of \cite{GrossAlg}, for the inner forms of the groups $\mathbf{SO}_7$, $\mathbf{SO}_8$ and $\mathbf{SO}_9$ which are split at the finite places and compact at the real place.
They used Arthur's classification to deduce the cardinalities of the sets $S(w_1,w_2,w_3)$ and $S(w_1,w_2,w_3,w_4)$ and, using the conjectural dimension formula of \cite{BFG}, $O_e(w_1,w_2,w_3,w_4)$.
Unfortunately the symplectic groups do not have such inner forms, nor do the special orthogonal groups $\mathbf{SO}_n$ when $n \mod 8 \not\in \{-1,0,1\}$.
Thus our main contribution is thus the direct computation of local orbital integrals.

Finally I would like to heartily thank Gaëtan Chenevier for suggesting that I work on this problem, for his enthusiasm and for enlightening discussions.

\section{Notations and definitions}
Let us precise some notations.
Let $\A_f$ denote the finite adèles $\prod'_p \Qp$ and $\A = \R \times \A_f$.
We will use boldface letters to denote linear algebraic groups, for example $\mathbf{G}$.
For schemes we denote base change using simply a subscript, for example $\mathbf{G}_{\Qp}$ instead of $\mathbf{G} \times_{\Q} \Qp$ where $\mathbf{G}$ is defined over $\Q$.
For a reductive group $\mathbf{G}$ we abusively call ``Levi subgroup of $\mathbf{G}$'' any Levi subgroup of a parabolic subgroup of $\mathbf{G}$, i.e.\ the centraliser of a split torus.
Rings are unital.
If $R$ is a ring and $\Lambda$ a finite free $R$-module, $\mathrm{rk}_R (\Lambda)$ denotes its rank.
If $G$ is a finite abelian group $G^{\wedge}$ will denote its group of characters.

Let us define the reductive groups that we will use.
For $n \geq 1$, let $q_n$ be the quadratic form on $\Z^n$ defined by
$$ q_n(x) = \sum_{i=1}^{\lfloor (n+1)/2 \rfloor} x_ix_{n+1-i}. $$
Let $\mathbf{O}_n$ be the algebraic group over $\Z$ representing the functor
\begin{align*}
\text{Category of commutative rings} & \rightarrow \text{Category of groups} \\
A & \mapsto \left\{ g \in \GL_n(A)\ |\ q_n \circ g = q_n \right\}.
\end{align*}
For $n$ odd define $\mathbf{SO}_n$ as the kernel of $\det : \mathbf{O}_n \rightarrow \mu_2$.
For n even, $\det : \mathbf{O}_n \rightarrow \mu_2$ factors through the Dickson morphism $\mathrm{Di} : \mathbf{O}_n \rightarrow \Z/2\Z$ (constant group scheme over $\Z$) and the morphism $\Z/2\Z \rightarrow \mu_2$ ``mapping $1 \in \Z/2\Z$ to $-1 \in \mu_2$''.
In that case $\mathbf{SO}_n$ is defined as the kernel of $\mathrm{Di}$.
For any $n \geq 1$, $\mathbf{SO}_n \rightarrow \mathrm{Spec}(\Z)$ is reductive in the sense of \cite{SGA3-III}[Exposé XIX, Définition 2.7].
It is semisimple if $n \geq 3$.

For $n \geq 1$ the subgroup $\mathbf{Sp}_{2n}$ of $\mathbf{GL}_{2n}/ \Z$ defined as the stabiliser of the alternate form
$$ (x,y) \mapsto \sum_{i=1}^n x_i y_{2n+1-i} - x_{2n+1-i} y_i $$
is also semisimple over $\Z$ in the sense of \cite{SGA3-III}[Exposé XIX, Définition 2.7].

If $\mathbf{G}$ is one of $\mathbf{SO}_{2n+1}$ ($n \geq 1$), $\mathbf{Sp}_{2n}$ ($n \geq 1$) or $\mathbf{SO}_{2n}$ ($n \geq 2$), the diagonal matrices form a split maximal torus $\mathbf{T}$, and the upper-triangular matrices form a Borel subgroup $\mathbf{B}$.
We will simply denote by $\underline{t} = (t_1, \dots, t_n)$ the element of $\mathbf{T}(A)$ ($A$ a commutative ring) whose first $n$ diagonal entries are $t_1, \dots, t_n$.
For $i \in \{1, \dots, n\}$, let $e_i \in X^*(\mathbf{T})$ be the character $\underline{t} \mapsto t_i$.
The simple roots corresponding to $\mathbf{B}$ are
$$ \begin{cases}
e_1 - e_2, \dots, e_{n-1}-e_n, e_n & \text{if } \mathbf{G} = \mathbf{SO}_{2n+1}, \\
e_1 - e_2, \dots, e_{n-1}-e_n, 2e_n & \text{if } \mathbf{G} = \mathbf{Sp}_{2n}, \\
e_1 - e_2, \dots, e_{n-1}-e_n, e_{n-1} + e_n & \text{if } \mathbf{G} = \mathbf{SO}_{2n}. \\
\end{cases} $$
In the first two cases (resp.\ third case), the dominant weights in $X^*(\mathbf{T})$ are the $\underline{k} = \sum_{i=1}^n k_i e_i$ with $k_1 \geq \dots \geq k_n \geq 0$ (resp.\ $k_1 \geq \dots \geq k_{n-1} \geq |k_n|$).

\section{Computation of the geometric side of Arthur's trace formula}

Arthur's invariant trace formula \cite{ArthurITF} for a reductive group $\mathbf{G}/\Q$ simplifies and becomes more explicit when $\mathbf{G}(\R)$ has discrete series and a ``nice'' smooth compactly supported distribution $f_{\infty}(g_{\infty})dg_{\infty}$ is used at the real place, as shown in \cite{ArthurL2} (see also \cite{GoKoMPh} for a topological proof).
In section \ref{subsectionEllTerms} we recall the elliptic terms $T_{\mathrm{ell}}\left(f_{\infty}(g_{\infty}) dg_{\infty} \prod_p f_p(g_p)dg_p\right)$ on the geometric side of this trace formula, where $\prod_p f_p(g_p)dg_p$ is a smooth compactly supported distribution on $\mathbf{G}(\A_f)$.
Then (section \ref{subsectionCompEll}) we give an algorithm to compute these elliptic terms when $\mathbf{G}$ is a split classical group and for any prime $p$, $f_p(g_p)dg_p$ is the trivial element of the unramified Hecke algebra.
Finally (section \ref{subsectionCompPara}) we give explicit formulae for the parabolic terms using the elliptic terms for groups of lower semisimple rank.

\subsection{Elliptic terms}
\label{subsectionEllTerms}

\subsubsection{Euler-Poincaré measures and functions}
\label{EPinftySection}

Let $\mathbf{G}$ be a reductive group over $\R$.
Thanks to \cite{EPSerre}, we have a canonical signed Haar measure on $\mathbf{G}(\R)$, called the Euler-Poincaré measure.
It is non-zero if and only if $\mathbf{G}(\R)$ has discrete series, that is if and only if $\mathbf{G}$ has a maximal torus defined over $\R$ which is anisotropic.

So assume that $\mathbf{G}(\R)$ has discrete series.
Let $K$ be a maximal compact subgroup of $\mathbf{G}(\R)$, $\mathfrak{g}_0 = \mathrm{Lie}(\mathbf{G}(\R))$ and $\mathfrak{g} = \C \otimes_{\R} \mathfrak{g}_0$.
Let $V_{\lambda}$ be an irreducible algebraic representation of $\mathbf{G}_{\C}$, parametrised by its highest weight $\lambda$.
We can see $V_{\lambda}$ as an irreducible finite-dimensional representation of $\mathbf{G}(\R)$, or as an irreducible $(\mathfrak{g}, K)$-module.
If $\pi$ is a $(\mathfrak{g}, K)$-module of finite length, consider
$$ \mathrm{EP}( \pi, \lambda) := \sum_i (-1)^i \dim H^i\left( (\mathfrak{g}, K), \pi \otimes V_{\lambda}^* \right). $$
Clozel and Delorme \cite{CloDel}[Théorème 3] show that there is a smooth, compactly supported distribution $f_{\lambda}(g)dg$ on $\mathbf{G}(\R)$ such that for any $\pi$ as above,
$$ \mathrm{Tr} \left(\pi\left(f_{\lambda}(g)dg\right)\right) = \mathrm{EP}( \pi, \lambda). $$
If $\pi$ is irreducible and belongs to the L-packet $\Pi_{\mathrm{disc}}(\lambda)$ of discrete series having the same infinitesimal character as $V_{\lambda}$, this number is equal to $(-1)^{q(\mathbf{G}(\R))}$ where $2 q(\mathbf{G}(\R)) = \dim \mathbf{G}(\R) - \dim K$.
If $\pi$ is irreducible and tempered but does not belong to $\Pi_{\mathrm{disc}}(\lambda)$ it is zero.

These nice spectral properties of $f_{\lambda}$ allow Arthur to derive nice geometric properties, similarly to the $p$-adic case in \cite{KoTam}.
If $\gamma \in \mathbf{G}(\R)$, the orbital integral $O_{\gamma}(f_{\lambda}(g)dg)$ vanishes unless $\gamma$ is elliptic semisimple, in which case, letting $\mathbf{I}$ denote the connected centraliser of $\gamma$ in $\mathbf{G}$:
$$ O_{\gamma}(f_{\lambda}(g)dg) = \mathrm{Tr} \left( \gamma | V_{\lambda} \right) \mu_{\mathrm{EP}, \mathbf{I}(\R)}. $$
In fact \cite{ArthurL2}[Theorem 5.1] computes more generally the invariant distributions $I_{\mathbf{M}}(\gamma, f_{\lambda})$ occurring in the trace formula (here $\mathbf{M}$ is a Levi subgroup of $\mathbf{G}$), and the orbital integrals above are just the special case $\mathbf{M} = \mathbf{G}$.
These more general invariant distributions will be used in the parabolic terms.

\subsubsection{Orbital integrals for $p$-adic groups}
\label{OrbIntSection}

We recall more precisely the definition of orbital integrals for the $p$-adic groups.
Let $p$ be a prime and $\mathbf{G}$ a reductive group over $\Qp$.
Let $K$ be a compact open subgroup of $\mathbf{G}(\Qp)$, $\gamma \in \mathbf{G}(\Qp)$ a semisimple element, and $\mathbf{I}$ its connected centraliser in $\mathbf{G}$.
Lemma 19 of \cite{HC} implies that for any double coset $KcK$ in $\mathbf{G}(\Qp)$, the set $X$ of $[g] \in K \backslash \mathbf{G}(\Qp) / \mathbf{I}(\Qp)$ such that $g \gamma g^{-1} \in KcK$ is finite.
Let $\mu$ (resp.\ $\nu$) be a Haar measure on $\mathbf{G}(\Qp)$ (resp.\ $\mathbf{I}(\Qp)$).
Then the orbital integral at $\gamma$ of the characteristic function of $KcK$
\[ O_{\gamma}(\mathbf{1}_{KcK}, \mu, \nu) = \int_{\mathbf{G}(\Qp) / \mathbf{I}(\Qp)} \mathbf{1}_{KcK}\left(g \gamma g^{-1}\right) \frac{d \mu}{d \nu}(g)  \]
is equal to
$$ \sum_{[g] \in X} \frac{\mu(K)}{\nu\left(g^{-1}Kg \cap \mathbf{I}(\Qp)\right)}. $$
The Haar measure $O_{\gamma}(\mathbf{1}_{KcK}, \mu, \nu) \nu$ is canonical, i.e.\ it does not depend on the choice of $\nu$.
Thus $O_{\gamma}$ canonically maps the space of smooth compactly supported complex valued distributions on $\mathbf{\mathbf{G}}(\Qp)$ (i.e.\ linear combinations of distributions of the form $\mathbf{1}_{KcK}(g) d\mu(g)$) to the one-dimensional space of complex Haar measures on $\mathbf{I}(\Qp)$.
\begin{rema} \label{RemaMeasures}
Note that any automorphism of the algebraic group $\mathbf{I}$ preserves $\nu$, and thus if $\mathbf{I}$ and $\nu$ are fixed, for any algebraic group $\mathbf{I}'$ isomorphic to $\mathbf{I}$, there is a well-defined corresponding Haar measure on $\mathbf{I}'$.
\end{rema}

\subsubsection{Definition of the elliptic terms}

Let $\mathbf{G}$ be a reductive group over $\Q$ such that $\mathbf{G}(\R)$ has discrete series.
Let $\lambda$ be a highest weight for the group $\mathbf{G}_{\C}$.
Choose a Haar measure $dg_{\infty}$ on $\mathbf{G}(\R)$, and let $f_{\infty}$ be a smooth compactly supported function on $\mathbf{G}(\R)$ such that the distribution $f_{\infty, \lambda}(g_{\infty})dg_{\infty}$ computes the Euler-Poincaré characteristic with respect to $V_{\lambda}$ as in \ref{EPinftySection}.
Let $\prod_p f_p(g_p)dg_p$ be a smooth compactly supported distribution on $\mathbf{G}(\A_f)$.
For almost all primes $p$, $\mathbf{G}_{\Qp}$ is unramified, $f_p = \mathbf{1}_{K_p}$ and $\int_{K_p} dg_p = 1$ where $K_p$ is a hyperspecial maximal compact subgroup in $\mathbf{G}(\Qp)$.
Let $C$ be the set of semisimple conjugacy classes $\mathrm{cl}(\gamma)$ in $\mathbf{G}(\Q)$ such that $\gamma$ belongs to an anisotropic maximal torus in $\mathbf{G}(\R)$.
For $\mathrm{cl}(\gamma) \in C$, denote by $\mathbf{I}$ the connected centraliser of $\gamma$ in $\mathbf{G}$.
Given such a $\gamma$, for almost all primes $p$, $\mathbf{I}_{\Qp}$ is unramified and $O_{\gamma}(f_p(g_p)dg_p)$ is the Haar measure giving measure one to a hyperspecial maximal compact subgroup of $\mathbf{I}(\Qp)$ (see \cite[Corollary 7.3]{KottEllSing}).
Thus $\prod_p O_{\gamma}(f_p(g_p)dg_p)$ is a well-defined complex Haar measure on $\mathbf{I}(\A_f)$.
Let $f(g)dg = f_{\infty, \lambda}(g_{\infty}) dg_{\infty} \prod_p f_p(g_p) dg_p$.
The elliptic part of the geometric side of Arthur's trace formula is
\begin{equation} \label{formulaTell} T_{\mathrm{ell}}(f(g)dg) = \sum_{\mathrm{cl}(\gamma) \in C} \frac{\mathrm{Vol}(\mathbf{I}(\Q) \backslash \mathbf{I}(\A))}{\mathrm{card}\left( \mathrm{Cent}(\gamma, \mathbf{G}(\Q)) / \mathbf{I}(\Q) \right)}  \mathrm{Tr}(\gamma\,|\,V_{\lambda}) \end{equation}
where $\mathbf{I}(\R)$ is endowed with the Euler-Poincaré measure, $\mathbf{I}(\A_f)$ the complex Haar measure $\prod_p O_{\gamma}(f_p(g_p)dg_p)$ and $\mathbf{I}(\Q)$ the counting measure.
The set of $\mathrm{cl}(\gamma) \in C$ such that for any prime $p$, $\gamma$ is conjugate in $\mathbf{G}(\Qp)$ to an element belonging to the support of $f_p$ is finite, so that the sum has only a finite number of nonzero terms.

\subsection{Computation of the elliptic terms in the trace formula}
\label{subsectionCompEll}

Our first task is to explicitly compute $T_{\mathrm{ell}}(f(g) dg)$ when $\mathbf{G}$ is one of $\mathbf{SO}_{2n+1}$, $\mathbf{Sp}_{2n}$ or $\mathbf{SO}_{4n}$ and moreover for any prime $p$, $f_p = \mathbf{1}_{\mathbf{G}(\Zp)}$ and $\int_{\mathbf{G}(\Zp)} d g_p = 1$.
In this case any $\gamma \in \mathbf{G}(\Q)$ whose contribution to $T_{\mathrm{ell}}(f(g) dg)$ is nonzero is torsion ($\gamma^r=1$ for some integer $r>0$), since $\gamma$ is compact in $\mathbf{G}(\Q_v)$ for any place $v$.
Here ``compact'' means that the smallest closed subgroup of $\mathbf{G}(\Q_v)$ containing $\gamma$ is compact, and it is equivalent to the fact that the eigenvalues of $\gamma$ in any faithful algebraic representation of $\mathbf{G}_{\overline{\Q_v}}$ have norm one.

First we describe the semisimple conjugacy classes in $\mathbf{G}(\Q)$ and their centralisers, a necessary first step to compute the set $C$ and the groups $\mathbf{I}$.
Then we explain how to enumerate the conjugacy classes of torsion elements in the group $\mathbf{G}(\Zp)$.
To be precise we can compute a collection of subsets $(Y_s)_s$ of $\mathbf{G}(\Zp)$ such that
$$ \{ g \in \mathbf{G}(\Zp)\,|\, \exists r>0,\,g^r=1 \} = \bigsqcup_s \{ xyx^{-1}\,|\, y \in Y_s,\, x \in \mathbf{G}(\Zp) \}. $$
Note that this leaves the possibility that for a fixed $s$, there exist distinct $y,y' \in Y_s$ which are conjugated under $\mathbf{G}(\Zp)$.
Thus it seems that to compute local orbital integrals we should check for such cases and throw away redundant elements in each $Y_s$, and then compute the measures of the centralisers of $y$ in $\mathbf{G}(\Zp)$.
This would be a computational nightmare.
Instead we will show in section \ref{compOrbIntSection} that the fact that such orbital integrals are masses (as in ``mass formula'') implies that we only need to compute the cardinality of each $Y_c$.
Finally the Smith-Minkowski-Siegel mass formulae of \cite{GanYu} provide a means to compute the global volumes.

\subsubsection{Semisimple conjugacy classes in classical groups}
\label{conjclassfieldSection}

Let us describe the absolutely semisimple conjugacy classes in classical groups over a field, along with their centralisers.
It is certainly well-known, but we could not find a reference.
We explain in detail the case of quadratic forms (orthogonal groups).
The case of alternate forms (symplectic groups) is similar but simpler since characteristic $2$ is not ``special'' and symplectic automorphisms have determinant $1$.
The case of (anti-)hermitian forms (unitary groups) is even simpler but it will not be used hereafter.

Let $V$ be a vector space of finite dimension over a (commutative) field $K$, equipped with a regular (``ordinaire'' in the sense of \cite[Exposé XII]{SGA7.2}) quadratic form $q$.
Let $\gamma \in \mathrm{O}(q)$ be absolutely semisimple, i.e.\ $\gamma \in \End_K(V)$ preserves $q$ and the finite commutative $K$-algebra $K[\gamma]$ is étale.
Since $\gamma$ preserves $q$, the $K$-automorphism $\tau$ of $K[\gamma]$ sending $\gamma$ to $\gamma^{-1}$ is well-defined: if $\dim_K V$ is even or $2 \neq 0$ in $K$, $\tau$ is the restriction to $K[\gamma]$ of the antiautomorphism of $\End_K(V)$ mapping an endomorphism to its adjoint with respect to the bilinear form $B_q$ corresponding to $q$, defined by the formula $B_q(x,y):=q(x+y)-q(x)-q(y)$.

In characteristic $2$ and odd dimension, $(V, q)$ is the direct orthogonal sum of its $\gamma$-stable subspaces $V' = \ker (\gamma-1)$ and $V'' = \ker P(\gamma)$ where $(X-1)P(X) \in K[X] \setminus \{0\}$ is separable and annihilates $\gamma$.
If $V''$ were odd-dimensional, the kernel of $B_q|_{V'' \times V''}$ would be a $\gamma$-stable line $Kx$ with $q(x) \neq 0$, which imposes $\gamma(x) = x$, in contradiction with $P(1) \neq 0$.
Thus $K[\gamma] = K\left[ \gamma|_{V''} \right] \times K$ if $V'' \neq 0$, and $\tau$ is again well-defined.

Thanks to $\tau$ we have a natural decomposition as a finite product:
\[ (K[\gamma], \gamma) = \prod_i (A_i, \gamma_i) \]
where for any $i$, $A_i$ is a finite étale $K$-algebra generated by $\gamma_i$ such that $\gamma_i \mapsto \gamma_i^{-1}$ is a well-defined $K$-involution $\tau_i$ of $A_i$ and $F_i = \{ x \in A_i\,|\, \tau_i(x)=x \}$ is a field.
Moreover the minimal polynomials $P_i$ of $\gamma_i$ are pairwise coprime.
For any $i$, either:
\begin{itemize}
\item
$\gamma_i^2 = 1$ and $A_i = K$,
\item
$\gamma_i^2 \neq 1$ and $A_i$ is a separable quadratic extension of $F_i$, $\Gal(A_i/F_i) = \{1,\tau_i\},$
\item
$\gamma_i^2 \neq 1$, $A_i \simeq F_i \times F_i$ and $\tau_i$ swaps the two factors.
\end{itemize}
Let $I_{\mathrm{triv}}$, $I_{\mathrm{field}}$ and $I_{\mathrm{split}}$ be the corresponding sets of indices.
There is a corresponding orthogonal decomposition of $V$:
\[ V = \bigoplus_i V_i \]
where $V_i$ is a projective $A_i$-module of constant finite rank.


\begin{lemm}
For any $i$, there is a unique $\tau_i$-hermitian (if $\tau_i$ is trivial, this simply means quadratic) form $h_i : V_i \rightarrow F_i$ such that for any $v \in V_i$, $q(v) = \mathrm{Tr}_{F_i/K}\left( h_i(v) \right)$.
\end{lemm}
\begin{proof}
If $i \in I_{\mathrm{triv}}$ this is obvious, so we can assume that $\dim_{F_i} A_i = 2$.
Let us show that the $K$-linear map
\begin{eqnarray*}
T : \{\tau_i\text{-hermitian forms on } V_i\} & \longrightarrow & \{K\text{-quadratic forms on } V_i \text{ preserved by } \gamma_i\} \\
h_i & \mapsto & \left( v \mapsto \mathrm{Tr}_{F_i/K} h_i(v) \right)
\end{eqnarray*}
is injective.
If $h_i$ is a $\tau_i$-hermitian form on $V_i$, denote by $B_{h_i}$ the unique $\tau_i$-sesquilinear map $V_i \times V_i \rightarrow A_i$ such that for any $v,w \in V_i$, $h_i(v+w)-h_i(v)-h_i(w) = \mathrm{Tr}_{A_i/F_i} B_{h_i}(v,w)$, so that in particular $h_i(v) = B_{h_i}(v,v)$.
Moreover for any $v,w \in V_i$, $B_{T(h_i)}(v,w) = \mathrm{Tr}_{A_i/K} B_{h_i}(v,w)$.
If $h_i \in \ker T$, then $B_{T(h_i)} = 0$ and by non-degeneracy of $\mathrm{Tr}_{A_i/K}$ we have $B_{h_i} = 0$ and thus $h_i=0$.

To conclude we have to show that the two $K$-vector spaces above have the same dimension.
Let $d = \dim_K F_i$ and $n = \dim_{A_i} V_i$, then $\dim_K \{\tau_i\text{-hermitian forms on } V_i\} = d n^2$.
To compute the dimension of the vector space on the right hand side, we can tensor over $K$ with a finite separable extension $K'/K$ such that $\gamma_i$ is diagonalizable over $K'$.
Since $\gamma_i^2 \neq 1$ the eigenvalues of $1 \otimes \gamma_i$ on $K' \otimes_K V_i$ are $t_1, t_1^{-1}, \dots, t_d, t_d^{-1}$ where the $t_k^{\pm 1}$ are distinct and $\neq 1$.
Furthermore each eigenspace $U_k^+ := \ker (1 \otimes \gamma_i - t_k \otimes 1), U_k^- := \ker (1 \otimes \gamma_i - t_k^{-1} \otimes 1)$ has dimension $n$ over $K'$.
If $q'$ is a $K'$-quadratic form on $K' \otimes_K V_i$ preserved by $1 \otimes \gamma_i$, then:
\begin{itemize}
\item for any $k$, $q'|_{U_k^{\pm}} = 0$ since $t_k^2 \neq 1$,
\item for any $k \neq l$, $B_{q'}|_{U_k^{\pm} \times U_l^{\pm}} = 0$ since $t_k/t_l, t_k t_l \neq 1$.
\end{itemize}
Hence $q'$ is determined by the restrictions of $B_{q'}$ to $U_k^+ \times U_k^-$, and conversely any family of $K'$-bilinear forms $U_k^+ \times U_k^- \rightarrow K'$ ($k \in \{1, \dots, d\}$) give rise to a $K'$-quadratic form on $K' \otimes_K V_i$ preserved by $1 \otimes \gamma_i$, and we conclude that the dimension is again $dn^2$.
\end{proof}

The regularity of $q$ implies that of $h_i$ (when $\gamma_i^2 \neq 1$, regularity means non-degeneracy of $B_{h_i}$).
In the split case, $V_i$ can be more concretely described as a pair $(W_i, W_i')$ of vector spaces over $F_i$ having the same dimension, $h_i$ identifies $W_i'$ with the dual $W_i^*$ of $W_i$ over $F_i$, and thus the pair $(V_i, h_i)$ is isomorphic to $((W_i, W_i^*), (w,f) \mapsto f(w))$. 

If instead of $q$ we consider a non-degenerate alternate form $\langle \cdot, \cdot \rangle$, we have the same kind of decomposition for $(K[\gamma], \gamma)$.
Moreover the above lemma still holds if instead of considering hermitian forms $h_i$ we consider $\tau_i$-sesquilinear forms $B_i : V_i \times V_i \rightarrow A_i$ such that for any $v \in V_i$, $\mathrm{Tr}_{A_i/F_i}(B_i(v,v)) = 0$.

\begin{prop}
\label{conjandcentss}
Two absolutely semisimple elements $\gamma, \gamma'$ of $\mathrm{O}(V,q)$ are conjugate if and only if there is a bijection $\sigma$ between their respective sets of indices $I$ and $I'$ and compatible isomorphisms $\left(A_i, \gamma_i\right) \simeq \left(A'_{\sigma(i)}, \gamma'_{\sigma(i)}\right)$ and $\left(V_i, h_i\right) \simeq \left(V'_{\sigma(i)}, h'_{\sigma(i)}\right)$.
Moreover the algebraic group $\mathrm{Cent}(\gamma, \mathbf{O}(V, q))$ is naturally isomorphic to
$$ \prod_{i \in I_{\mathrm{triv}}} \mathbf{O}(V_i, h_i) \times \prod_{i \in I_{\mathrm{field}}} \Res_{F_i/K} \mathbf{U}(V_i, h_i) \times \prod_{i \in I_{\mathrm{split}}} \Res_{F_i/K} \mathbf{GL}(W_i). $$
\end{prop}

If $\dim_K V$ is odd $\mathbf{O}(V,q) = \mathbf{SO}(V,q) \times \mu_2$, so this proposition easily yields a description of absolutely semisimple conjugacy classes in $\SO(V,q) = \mathbf{SO}(V,q)(K)$ and their centralisers.
If $\dim_K V$ is even the proposition still holds if we replace $\mathbf{O}(V,q)$ by $\mathbf{SO}(V,q)$ and $\prod_{i \in I_{\mathrm{triv}}} \mathbf{O}(V_i, h_i)$ by $\mathbf{S}\left(\prod_{i \in I_{\mathrm{triv}}} \mathbf{O}(V_i, h_i) \right)$ and add the assumption $I_{\mathrm{triv}} \neq \emptyset$.
If $\dim_K V$ is even and $I_{\mathrm{triv}} = \emptyset$, the datum $(A_i, \gamma_i, V_i, h_i)_{i \in I}$ determines two conjugacy classes in $\SO(V,q)$.

In the symplectic case there is a similar proposition, but now the indices $i \in I_{\mathrm{triv}}$ yield symplectic groups.

Note that if $K$ is a local or global field in which $2 \neq 0$, the simple and explicit invariants in the local case and the theorem of Hasse-Minkowski (and its simpler analogue for hermitian forms, see \cite{JacobsonHerm}) in the global case allow to classify the semisimple conjugacy classes explicitely.
For example if $K = \Q$, given $M>0$ one can enumerate the semisimple conjugacy classes in $\SO(V,q)$ annihilated by a non-zero polynomial having integer coefficients bounded by $M$.

\subsubsection{Semisimple conjugacy classes in hyperspecial maximal compact subgroups}
\label{conjclassZpSection}

To compute orbital integrals in the simplest case of the unit in the unramified Hecke algebra of a split classical group over a $p$-adic field, it would be ideal to have a similar description of conjugacy classes and centralisers valid over $\Zp$.
It is straightforward to adapt the above description over any ring (or any base scheme).
However, it is not very useful as the conjugacy classes for which we would like to compute orbital integrals are not all ``semisimple over $\Zp$'', i.e.\ $\Zp[\gamma]$ is not always an étale $\Zp$-algebra.
Note that the ``semisimple over $\Zp$'' case is covered by \cite[Corollary 7.3]{KottEllSing} (with the natural choice of Haar measures, the orbital integral is equal to $1$).
Nevertheless using the tools of the previous section, we give in this section a method to exhaust the isomorphism classes of triples $(\Lambda, q, \gamma)$ where $\Lambda$ is a finite free $\Zp$-module, $q$ is a regular quadratic form on $\Lambda$ and $\gamma \in \mathrm{SO}(\Lambda, q)$.
The symplectic case is similar.
This means that we will be able to enumerate them, but a priori we will obtain some isomorphism classes several times.
In the next section we will nonetheless see that the results of this section can be used to compute the orbital integrals, without checking for isomorphisms.

Let $\Lambda$ be a free $\Zp$-module of finite rank endowed with a regular quadratic form $q$, and let $\gamma \in \mathrm{Aut}_{\Zp}(\Lambda)$ preserving $q$ and semisimple over $\Qp$.
We apply the notations and considerations of section \ref{conjclassfieldSection} to the isometry $\gamma$ of $\Qp \otimes_{\Zp} \Lambda$, to obtain quadratic or hermitian spaces $\left(\Qp \otimes_{\Zp} \Lambda \right)_i$.
Consider the lattices
$$ \Lambda_i := \Lambda \cap \left(\Qp \otimes_{\Zp} \Lambda \right)_i = \ker \left( P_i(\gamma)\ |\ \Lambda \right). $$
Let $N \geq 0$ be such that $p^N$ belongs to the ideal of $\Zp[X]$ generated by the $\prod_{j \neq i} P_j$ for all $i$.
Then $\Lambda/\left(\oplus_i \Lambda_i \right)$ is annihilated by $p^N$, so this group is finite.
Since $\Lambda_i$ is saturated in $\Lambda$ and $q$ is regular, for any $v \in \Lambda_i \smallsetminus p \Lambda_i$,
\begin{equation} \label{ModCond} \begin{cases} p^N \in B(v,\Lambda_i) & \mbox{if $p \geq 3$ or $\mathrm{rk}_{\Zp} \Lambda_i$ is even,} \\ p^N \in B(v,\Lambda_i) \mbox{ or } q(v) \in \Z_2^{\times} & \mbox{if $p = 2$ and $\mathrm{rk}_{\Zp} \Lambda_i$ is odd.} \end{cases} \end{equation}
The $\Zp[\gamma_i]$-module $\Lambda_i$ is endowed with a hermitian (quadratic if $\gamma_i^2=1$) form $h_i$ taking values in $F_i$.
The sesquilinear (bilinear if $\gamma_i^2=1$) form $B_i : \Lambda_i \times \Lambda_i \rightarrow A_i$ associated with $h_i$ has the property that for all $v,w \in \Lambda_i$,
$$ B(v,w) = \mathrm{Tr}_{A_i / \Qp} \left( B_i(v,w) \right). $$
From now on we assume for simplicity that $\Zp[\gamma_i]$ is normal (i.e.\ either it is the integer ring of an extension of $\Qp$, or the product of two copies of such an integer ring), as it will be the case in our global situation which imposes that the $\gamma_i$'s be roots of unity.
The structure of quadratic or hermitian modules over such rings is known: see \cite{OMeara} for the quadratic case, \cite{Jacobowitz} for the hermitian case.
The ``split'' case amounts to the comparison of two lattices in a common vector space (isomorphism classes of such pairs are parametrised by ``invariant factors'').
Choose a uniformiser $\varpi_i$ of $\Zp[\gamma_i]$ (by definition, in the split case $\varpi_i$ is a uniformiser of $\mathcal{O}_{F_i}$).
In all cases, there is a (non-canonical) orthogonal decomposition $\Lambda_i = \bigoplus_{r \in \Z} \Lambda_i^{(r)}$ such that $\varpi_i^{-r} B_i|_{\Lambda_i^{(r)} \times \Lambda_i^{(r)}}$ is integral and non-degenerate.
If $(\varpi_i^{d_i})$ is the different of $\Zp[\gamma_i]/\Zp$ and $(p)=(\varpi_i^{e_i})$, condition \ref{ModCond} imposes (but in general stays stronger than) the following:
\begin{equation} \label{ModCondHerm} \begin{cases} \Lambda_i^{(r)} = 0 \mbox{ unless } -d_i \leq r \leq -d_i+Ne_i & \mbox{if $p \geq 3$ or $\mathrm{rk}_{\Zp} \Lambda_i$ is even,} \\ \Lambda_i^{(r)} = 0 \mbox{ unless } 0 \leq r \leq \max(1,N) & \mbox{if $p = 2$ and $\mathrm{rk}_{\Zp} \Lambda_i$ is odd.} \end{cases} \end{equation}
Note that in the second case $\gamma_i^2=1$ and $h_i$ is a quadratic form over $\Z_2$.
These conditions provide an explicit version of the finiteness result in section \ref{OrbIntSection}, since for any $i$ and $r$ there is a finite number of possible isomorphism classes for $\Lambda_i^{(r)}$, and when the $\Lambda_i$'s are fixed, there is only a finite number of possible $\gamma$-stable $q$-regular $\Lambda$'s since
\[ \bigoplus_i \Lambda_i \subset \Lambda \subset p^{-\max(1,N)} \bigoplus_i \Lambda_i. \]

For efficiency it is useful to sharpen these conditions.
Denote by $o$ an orbit of $\Z/2\Z \times \Gal\left( \overline{\Fp} / \Fp \right)$ acting on $\overline{\Fp}^{\times}$, where the non-trivial element of $\Z/2\Z$ acts by $x \mapsto x^{-1}$.
Concretely, $o$ is an orbit in the set of primitive $m$-th roots of unity ($m$ coprime to $p$) under the subgroup $\langle p, -1 \rangle$ of $\left(\Z/ m\Z\right)^{\times}$.
Let $I_o$ be the set of indices $i$ such that $\gamma_i$ modulo some (at most two possibilities) maximal ideal of $\Zp[\gamma_i]$ belongs to $o$.
Then for $o \neq o'$, $\prod_{i \in I_o} P_i$ and $\prod_{i \in I_{o'}} P_i$ generate the unit ideal in $\Zp[X]$, thus $\Lambda = \bigoplus_o \Lambda_{I_o}$ where
\[ \Lambda_{I_o} = \mathrm{Sat}_{\Lambda}\left(\bigoplus_{i \in I_o} \Lambda_i \right) = \ker \left(\prod_{i \in I_o} P_i(\gamma)\ |\ \Lambda \right). \]
Here $\mathrm{Sat}_{\Lambda}(\Lambda')$, the \emph{saturation} of $\Lambda'$ in $\Lambda$, is defined as $\Lambda \cap \left( \Qp \Lambda' \right)$.
Our task is now to enumerate the $\gamma$-stable $q$-regular lattices containing $\bigoplus_{i \in I_o} \Lambda_i$ in which each $\Lambda_i$ is saturated.
For $i \in I_o$, there is a canonical (``Jordan-Chevalley over $\Zp$'') decomposition $\gamma_i = \alpha_i \beta_i$ where $\Phi_m(\alpha_i)=0$ ($m$ associated with $o$ as above) and
\[ \beta_i^{p^n} \xrightarrow[n \rightarrow + \infty]{} 1 . \]
Since we assumed that $\Zp[\gamma_i]=\Zp[\alpha_i][\beta_i]$ is normal, either $\beta_i \in \Zp[\alpha_i]$ or over each factor of $\Qp[\alpha_i]$, $\Qp[\gamma_i]$ is a non-trivial totally ramified field extension and $\beta_i-1$ is a uniformiser.
In any case, define $h_i' := \mathrm{Tr}_{F_i/\Qp[\alpha_i+\alpha_i^{-1}]} (h_i)$, a quadratic or hermitian (with respect to $\tau_i : \alpha_i \mapsto \alpha_i^{-1}$) form on the $\Zp[\alpha_i]$-module $\Lambda_i$.
On $\Lambda_{I_o}$, $\gamma = \alpha_{I_o} \beta_{I_o}$ as above, the restriction of $\alpha_{I_o}$ to $\Lambda_i$ ($i \in I_o$) is $\alpha_i$, and the minimal polynomial of $\alpha_i$ over $\Qp$ does not depend on $i \in I_o$.
Thus we can see the $\Lambda_i$, $i \in I_o$ as finite free quadratic or hermitian modules over the same ring $\Zp\left[\alpha_{I_o}\right]$, each of these modules being endowed with an automorphism $\beta_i$ satisfying $\beta_i^{p^n} \rightarrow 1$.
Moreover since $\Zp\left[\alpha_{I_o}\right]$ is an étale $\Zp$-algebra, the regularity of $q$ (restricted to $\Lambda_{I_o}$) is equivalent to the regularity of $h' = \oplus_i h'_i$ on $\Lambda_{I_o}$.
Knowing the $\Lambda_i$'s, finding the possible $\Lambda_{I_o}$'s amounts to finding the $\beta$-stable $h'$-regular lattices containing $\bigoplus_{i \in I_o} \Lambda_i$ in which each $\Lambda_i$ is saturated, where $\beta = \oplus_i \beta_i$.

Let us now specialise to the case where each $\gamma_i$ is a root of unity, i.e.\ $\beta_i^{p^n} = 1$ for some $n \geq 0$.
Denote by $\Phi_r$ the $r$-th cyclotomic polynomial.
\begin{lemm}
Let $m \geq 1$ be coprime to $p$.
In $\Zp[X]$, for any $k \geq 1$, $p$ belongs to the ideal generated by $\Phi_{p^km}(X)$ and $\Phi_m\left(X^{p^{k-1}}\right)$.
\end{lemm}
\begin{proof}
For $k=1$, since $\Phi_m(X^p) = \Phi_{pm}(X) \Phi_m(X)$, by derivating we obtain the following equality in the finite étale $\Zp$-algebra $\Zp[X]/\Phi_m(X)$:
\[ \Phi_{pm}(X)=p X^{p-1} \Phi_m'(X^p)/\Phi_m'(X) = p \times \text{unit}. \]
Hence there exists $U,V \in \Zp[X]$ such that $\Phi_{pm}(X)U(X) + \Phi_m(X)V(X) = p$.
For any $k \geq 1$ we have $\Phi_{p^km}(X) = \Phi_{pm}\left(X^{p^{k-1}} \right)$, and the general case follows.
\end{proof} 
Having chosen quadratic or hermitian lattices $\left(\Lambda_i\right)_{i \in I_o}$, there is a natural order in which to proceed to enumerate the possible $\Lambda_{I_o}$.
Let us focus on one orbit $o$.
To lighten notation name the indices $I_o = \{1, \ldots , s \}$ in such a way that for $1 \leq t \leq s$, $P_t | \Phi_{mp^{k_t}}$ where $0 \leq k_1 < \ldots < k_s$.
Having fixed $o$ we also drop the indices $I_o$ from our notations.
The lemma tells us that for any $1 \leq t < s$, $p$ annihilates
$$ \mathrm{Sat}_{\Lambda} \left( \Lambda_1 \oplus \ldots \oplus \Lambda_{t+1} \right) / \left( \mathrm{Sat}_{\Lambda} \left( \Lambda_1 \oplus \ldots \oplus \Lambda_t \right) \oplus \Lambda_{t+1} \right) $$
and thus we also have that $p^{s-t}$ annihilates
$$ \Lambda / \left( \mathrm{Sat}_{\Lambda} \left( \Lambda_1 \oplus \ldots \oplus \Lambda_t \right) \oplus \Lambda_{t+1} \oplus \dots \Lambda_s \right).$$
This will provide a sharper version of condition \ref{ModCond}.
Let $B'$ be the sesquilinear (bilinear if $\alpha^2=1$) form on $\Lambda$ associated with $h'$.
For any $i \in I_o$ there is an orthogonal decomposition with respect to $B'$: $ \Lambda_i = \bigoplus_r L_i^{(r)}$ where each $L_i^{(r)}$ is $p^r$-modular for $B'$, i.e.\ $p^{-r}B'|_{L_i^{(r)} \times L_i^{(r)}}$ takes values in $\Zp[\alpha]$ and is non-degenerate.
For $1 \leq t \leq s$ denote $M_t = \mathrm{Sat}_{\Lambda} \left( \Lambda_1 \oplus \ldots \oplus \Lambda_t \right)$, which can similarly be decomposed orthogonally with respect to $B'$: $ M_t = \bigoplus_r M_t^{(r)}$.
Note that $M_1 = \Lambda_1$.
Analogously to condition \ref{ModCond}, for $1 \leq t < s$ we have
\begin{equation}\label{ModCondCyclo}
L_{t+1}^{(r)} = M_t^{(r)} = 0 \quad\mbox{ unless } 0 \leq r \leq s-t.
\end{equation}
and if $s=1$ we simply have that the hermitian (or quadratic) module $(\Lambda_1, h')$ over $\Zp[\alpha]$ is regular.
We can deduce a sharper version of condition \ref{ModCondHerm}.
If $s>1$ then
\begin{align}\label{ModCondHermCyclo1}
\Lambda_1^{(r)} = 0 & \mbox{ unless } -d_1 \leq r \leq -d_1 + (s-1) e_1 \\\label{ModCondHermCyclo2}
\mbox{for } 1 < t \leq s,\ \Lambda_t^{(r)} = 0 & \mbox{ unless } -d_t \leq r \leq -d_t + (s-t+1) e_t.
\end{align}
while for $s=1$:
\begin{equation}\label{ModCondHermCyclo3} \begin{cases}
\Lambda_1^{(r)} = 0 \mbox{ if } r \neq -d_1 & \mbox{if $p\geq 3$ or $m>1$,} \\
\Lambda_1 \mbox{ is a regular quadratic $\Z_2$-module} & \mbox{if $p=2$ and $m=1$.}
\end{cases} \end{equation}

Let us recapitulate the algorithm thus obtained to enumerate \emph{non-uniquely} the isomorphism classes of triples $(\Lambda, q, \gamma)$ such that $(\Lambda,q)$ is regular and $\gamma$ is torsion.
Begin with a datum $(A_i, \gamma_i)_{i \in I}$, i.e.\ fix the characteristic polynomial of $\gamma$.
For any orbit $o$ for which $s = \mathrm{card}(I_o) > 1$:
\begin{enumerate}
\item
For any $i \in I_o$, enumerate the isomorphism classes of quadratic or hermitian $\Zp[\alpha_i]$-modules $\Lambda_i$ subject to conditions \ref{ModCondHermCyclo1} and \ref{ModCondHermCyclo2}, compute $B'$ on $\Lambda_i \times \Lambda_i$ and throw away those which do not satisfy condition \ref{ModCondCyclo}.
\item
For any such family $(\Lambda_i)_{i \in I_o}$, enumerate inductively the possible $\mathrm{Sat}_{\Lambda} \left( \Lambda_1 \oplus \ldots \oplus \Lambda_t \right)$.
At each step $t=1, \dots, s$, given a candidate $M_t$ for $\mathrm{Sat}_{\Lambda} \left( \Lambda_1 \oplus \ldots \oplus \Lambda_t \right)$, we have to enumerate the candidates $M_{t+1}$ for $\mathrm{Sat}_{\Lambda} \left( \Lambda_1 \oplus \ldots \oplus \Lambda_t \right)$, i.e.\ the $\beta$-stable lattices containing $M_t \oplus \Lambda_{t+1}$ such that
\begin{enumerate}
\item $h'$ is integral on $M_{t+1}$,
\item both $M_t$ and $\Lambda_{t+1}$ are saturated in $M_{t+1}$,
\item if $t<s-1$, $M_{t+1}$ satisfies condition \ref{ModCondCyclo},
\item if $t=s-1$, $M_{t+1}$ (a candidate for $\Lambda$) is regular for $h'$.
\end{enumerate}
\end{enumerate}
\begin{rema}
The first step can be refined, since already over $\Qp$ there are obstructions to the existence of a regular lattice.
These obstructions exist only when $h' = q$ is a quadratic form, i.e.\ $\alpha_{I_o}^2 = 1$, so let us make this assumption for a moment.
Consider its discriminant $D=\mathrm{disc}(q) \in \Qp^{\times} / \mathrm{squares}(\Qp^{\times})$.
If $\mathrm{rk}_{\Zp} \Lambda = 2n$ is even, then $\Qp[X]/(X^2-(-1)^nD)$ is unramified over $\Qp$.
If $\mathrm{rk}_{\Zp} \Lambda$ is odd, the valuation of $\mathrm{disc}(q)/2$ is even.
Moreover in any case, once we fix the discriminant, the Hasse-Witt invariant of $q$ is determined.
We do not go into more detail.
A subtler obstruction is given by the spinor norm of $\gamma$.
Assume that $N = \mathrm{rk}_{\Zp} \Lambda$ is at least $3$, and for simplicity assume also that $\det (\gamma) = 1$.
The regular lattice $(\Lambda, q)$ defines a reductive group $\mathbf{SO}(q)$ over $\Zp$.
The fppf exact sequence of groups over $\Zp$
$$ 1 \rightarrow \mathbf{\mu}_2 \rightarrow \mathbf{Spin}(q) \rightarrow \mathbf{SO}(q) \rightarrow 1$$
yields for any $\Zp$-algebra $R$ the spinor norm $\mathbf{SO}(q)(R) \rightarrow H^1_{\mathrm{fppf}}(R, \mathbf{\mu}_2)$ whose kernel is the image of $\mathbf{Spin}(q)(R)$.
Moreover if $\mathrm{Pic}(R)=1$ (which is the case if $R=\Qp$ or $\Zp$) we have $H^1_{\mathrm{fppf}}(R, \mathbf{\mu}_2) = R^{\times} / \mathrm{squares}(R^{\times})$.
Thus another obstruction is that the spinor norm of $\gamma$ must have even valuation.
We can compute the spinor norm of each $\gamma_i$ easily.
If $\gamma_i = -1$ its spinor norm is simply the discriminant of the quadratic form $h_i$.
If $i \not\in I_{\mathrm{triv}}$ a straightforward computation shows that the spinor norm of $\gamma_i$ is $N_{A_i / \Qp}(1+\gamma_i)^{\dim_{A_i} V_i}$.
Note that it does not depend on the isomorphism class of the hermitian form $h_i$.
\end{rema}

Let us elaborate on the second step of the algorithm.
For an orbit $o$ for which $s=1$, we simply have to enumerate the modules $\Lambda_1$ satisfying \ref{ModCondHermCyclo3} and such that the resulting quadratic form $q$ (equivalently, $h'$) is regular.

We have not given an optimal method for the case $s>1$.
A very crude one consists in enumerating all the free $\Fp[\alpha]$-submodules in $p^{-1} \Zp / \Zp \otimes_{\Zp} (M_t \oplus \Lambda_{t+1})$ and keeping only the relevant ones.
The following example illustrates that one can do much better in many cases.

\begin{exam} \label{examples2}
Consider the ``second simplest'' case $s=2$.
Assume for simplicity that $p>2$ or $m>1$.
Then condition \ref{ModCondCyclo} shows that for any pair $((\Lambda_1, h_1), (\Lambda_2, h_2))$ found at the first step of the algorithm, we have
$$ \Lambda_1 = L_1^{(0)} \oplus L_1^{(1)} \qquad \text{and} \qquad \Lambda_2 = L_2^{(0)} \oplus L_2^{(1)} $$
where each $L_i^{(r)}$ is $p^r$-modular.
Moreover for any $i \in \{1,2\}$ the topologically unipotent automorphism $\beta_i$ stabilises
$$ pL_i^{(0)} \oplus L_i^{(1)} = \{ v \in \Lambda_i\,|\, \forall w \in \Lambda_i,\,B_i'(v,w) \in p \Zp[\alpha] \} $$
and thus $\beta_i$ induces a unipotent automorphism $\overline{\beta_i}$ of $(V_i, \eta_i)$ where $V_i = L_i^{(1)}/pL_i^{(1)}$ and $\eta_i$ is a the non-degenerate quadratic or hermitian form $p^{-1} h'_i \mod p$ on $V_i$.
It is easy to check that any relevant $\Lambda \supset \Lambda_1 \oplus \Lambda_2$ is such that
$$ p\Lambda / (p\Lambda_1 \oplus p \Lambda_2 ) = \{ v_1 \oplus f(v_1)\,|\,v_1 \in V_1 \} $$
for a unique isomorphism $f : (V_1, \eta_1, \beta_1) \rightarrow (V_2, -\eta_2, \beta_2)$.
Conversely such an isomorphism yields a relevant $\Lambda$.
\end{exam}
For $p=2$ and $m=1$ there is a similar but a bit more complicated description of the relevant lattices $\Lambda \supset \Lambda_1 \oplus \Lambda_2$.
In that case each form $\eta_i$ is a ``quadratic form modulo $4$'', i.e.\ $x \mapsto \langle x,x \rangle \mod 4$ where $\langle \cdot, \cdot \rangle$ is a symmetric bilinear form on a free $\Z_2$-module $N$.
Note that $\langle x,x \rangle \mod 4$ only depends on the class of $x$ in $\F_2 \otimes N$.
A further complication comes into play when $\mathrm{rk}_{\Z_2} (\Lambda_1) + \mathrm{rk}_{\Z_2} (\Lambda_2)$ is odd, but we do not go into more detail.

In the case of an arbitrary $s>1$, the observation made in example \ref{examples2} still applies at the last step $t=s-1$, replacing $(\Lambda_1, \Lambda_2)$ with $(M_{s-1}, \Lambda_s)$.
We do not go into the details of our implementation of the previous steps ($t<s-1$).
We merely indicate that in general $pM_{t+1}/(M_t \oplus \Lambda_{t+1})$ is still described using an isomorphism $f$ between a $\beta$-stable subspace of $\bigoplus_{r \geq 1} M_t^{(r)} \mod p$ and a $\beta$-stable subspace of $\bigoplus_{r \geq 1} L_t^{(r)} \mod p$.

\begin{rema}
Regarding all the results of this section, the symplectic case is similar, replacing ``quadratic'' by ``symplectic'' and ``hermitian'' by ``antihermitian'', and even simpler because the prime $2$ is ``less exceptional''.
More precisely, the classification of hermitian modules for e.g.\ the quadratic extension $\Z_p[\zeta_{p^k}] / \Z_p[\zeta_{p^k}+\zeta_{p^k}^{-1}]$ is more involved for $p=2$ than for the other primes (see \cite{Jacobowitz}), but once we have enumerated the possible isomorphism classes of $\Lambda_i$'s, the enumeration of the relevant $\Lambda \supset \oplus_i \Lambda_i$ can be done uniformly in $p$.
\end{rema}

\subsubsection{Orbital integrals for the unit in the unramified Hecke algebra of a $p$-adic classical group}
\label{compOrbIntSection}

In this section we show that thanks to the fact that orbital integrals are formally sums of masses (where ``mass'' takes the same meaning as in ``mass formula'', or in overly fancy terms, the ``measure of a groupoid''), they can be computed by counting instead of enumerating and checking isomorphisms.
As before we focus on the case of special orthogonal groups, the case of symplectic groups being easier.

Let $\Lambda_0$ be a free $\Zp$-module of finite rank endowed with a regular quadratic form $q_0$ and consider the algebraic group $\mathbf{G} = \mathbf{SO}(\Lambda_0, q_0)$ which is reductive over $\Zp$.
Let $f = \mathbf{1}_{\mathbf{G}(\Zp)}$ be the characteristic function of $\mathbf{G}(\Zp)$ and fix the Haar measure on $\mathbf{G}(\Qp)$ such that $\int_{\mathbf{G}(\Zp)} dg = 1$.
Let $\gamma_0 \in \mathbf{G}(\Qp)$ be semisimple (for now we do not assume that it is torsion), and let $\mathbf{I}_0$ be its connected centraliser in $\mathbf{G}_{\Qp}$.
Fix a Haar measure $\nu$ on $\mathbf{I}_0(\Qp)$.
Consider the isomorphism classes of triples $(\Lambda, q, \gamma)$ such that
\begin{itemize}
\item
$\Lambda$ is a free $\Zp$-module of finite rank endowed with a regular quadratic form $q$,
\item
$\gamma \in \SO(\Lambda, q)$,
\item
there exists an isomorphism between $(\Qp \otimes_{\Zp} \Lambda, q, \gamma)$ and $(\Qp \otimes_{\Zp} \Lambda_0, q_0, \gamma_0)$.
\end{itemize}
We apply the previous section's notations and results to such $(\Lambda, q, \gamma)$.
The last condition can be expressed simply using the classical invariants of quadratic (over $\Qp$) or hermitian (over  $\Qp[\gamma_i]$) forms, as in Proposition \ref{conjandcentss}.
It implies that $\mathbf{I}_0$ and the connected centraliser $\mathbf{I}$ of $\gamma$ in $\mathbf{SO}(\Qp \otimes_{\Zp} \Lambda, q)$ are isomorphic, and by Remark \ref{RemaMeasures} we can see $\nu$ as a Haar measure on $\mathbf{I}(\Qp)$.
Then
$$ O_{\gamma_0}(f(g)dg) = \left(\sum_{(\Lambda, q, \gamma)} \nu\left( \mathbf{I}(\Qp) \cap \mathrm{SO}(\Lambda, q)\right)^{-1} \right) \nu $$
where the sum ranges over isomorphism classes as above.
Note that $\mathbf{I}(\Qp) \cap \mathrm{SO}(\Lambda, q)$ stabilises each $\Lambda_i$, so that it is a subgroup of $\prod_i \Gamma_i \subset \mathbf{I}(\Qp)$ where
$$ \Gamma_i = \begin{cases} \mathrm{SO}(\Lambda_i, h_i) & \mbox{if $i \in I_{\mathrm{triv}}$} \\ \mathrm{U}(\Lambda_i, h_i) & \mbox{if } i \in I_{\mathrm{field}} \cup I_{\mathrm{split}}. \end{cases} $$
In fact $\mathbf{I}(\Qp) \cap \mathrm{SO}(\Lambda, q)$ is the stabiliser of $\Lambda / \bigoplus_i \Lambda_i$ for the action of $\prod_i \Gamma_i$ on $(\Qp / \Zp) \otimes_{\Zp} \left( \bigoplus_i \Lambda_i \right)$.
Grouping the terms in the above sum according to the isomorphism classes of the quadratic or hermitian modules $\Lambda_i$, we obtain
\begin{equation} \label{formulaorbint} O_{\gamma_0}(f(g)dg) = \left( \sum_{\left(\Lambda_i, h_i\right)_{i \in I}} \frac{\mathrm{ext} \left( (\Lambda_i, h_i)_i \right)}{\nu\left(\prod_i \Gamma_i \right)} \right) \nu .\end{equation}
Now the sum ranges over the isomorphism classes of quadratic or hermitian lattices $(\Lambda_i, h_i)$ over $\Zp[\gamma_i]$, which become isomorphic to the corresponding datum for $(\Qp \otimes_{\Zp} \Lambda_0, q_0, \gamma_0)$ when $p$ is inverted, and
$$ \mathrm{ext} \left( (\Lambda_i, h_i)_i \right) := \mathrm{card} \left\{ q\text{-regular } (\oplus_i\gamma_i)\text{-stable }\Lambda \supset \bigoplus_i \Lambda_i \ |\ \forall i,\ \Lambda_i \text{ saturated in } \Lambda\right\}. $$
We will study the volumes appearing at the denominator below, but for the moment we consider these numerators.
Motivated by the global case, assume from now on that $\gamma_0$ is torsion as in the end of the previous section.
It is harmless to restrict our attention to a single orbit $o$, and assume $I=I_o$.
For the computation of orbital integrals, the benefit resulting from the transformation above is that instead of enumerating the possible $M_{t+1}$ knowing $M_t$ at the last step $t=s-1$, we only have to count them.
Let us discuss the various cases that can occur, beginning with the simplest ones.

The unramified case corresponds to $s=1$ and $A_1 = \Qp[\gamma_1] = \Qp[\alpha]$, and in that case there is a unique relevant isomorphism class $(\Lambda_1, h_1)$.
It is easy to check that we recover Kottwitz's result \cite{KottEllSing}[Corollary 7.3] that the orbital integral equals $1$ for the natural choice of Haar measures.

The case where $s=1$ but $\Qp[\gamma_1] / \Qp[\alpha]$ can be non-trivial (i.e.\ ramified) is not much harder: the algorithm given in the previous section identifies the relevant isomorphism classes $(\Lambda_1, h_1)$ appearing below the sum, and $\mathrm{ext}(\Lambda_1, h_1)=1$.
In this case we have reduced the problem of computing the orbital integral by that of computing the volume of the stabilisers of some lattices.
When $\mathbf{G} = \mathbf{Sp}_2 = \mathbf{SL}_2$ it is the worst that can happen.

The first interesting case is $s=2$.
Assume for simplicity that $p>2$ or $m>1$, and let us look back at example \ref{examples2}, using the same notations.
Then $\mathrm{ext}((\Lambda_1, h_1), (\Lambda_2, h_2))=0$ unless $(V_1, \eta_1, \beta_1) \simeq (V_2, -\eta_2, \beta_2)$, in which case $\mathrm{ext}((\Lambda_i, h_i)_i) = \mathrm{card}\left( \mathrm{Aut}(V_1, \eta_1, \beta_1)\right)$.
This group is the centraliser of a unipotent element in a classical group over a finite field.
Results of Wall \cite{Wall} give the invariants of such conjugacy classes as well as formulae for their centralisers.
In many cases (e.g.\ if $\mathrm{rk}_{\Zp}(\Lambda) < p^2-1$) the automorphism $\beta_1$ of $V_1$ is trivial, and thus we do not need the general results of Wall, but merely the simple cardinality formulae of finite classical groups.
For $\mathbf{G} = \mathbf{Sp}_4$ or $\mathbf{SO}_4$ we have $s \leq 2$ and $\beta_1|_{V_1} = 1$ at worst.

When $s>2$ the situation is of course more complicated, and it seems that we cannot avoid the enumeration of successive lattices $M_{t+1} \supset M_t \oplus \Lambda_{t+1}$ for $t<s-1$, although the last step $t=s-1$ is identical to the above case.
Note however that these ``very ramified'' cases are rare in low rank.
More precisely $\mathrm{rk}_{\Zp} \Lambda \geq p^{s-1}$, e.g.\ in rank less than $25$ it can happen that $s>2$ only for $p=2,3$.
Thus the ``worst cases'' have $p=2$.
This is fortunate because for fixed $k$ and $n$ the number of $k$-dimensional subspaces in an $n$-dimensional vector space over a finite field with $q$ elements increases dramatically with $q$.

\begin{rema}
In the case where $\mathbf{G}$ is an \emph{even} special orthogonal group, some of the semisimple conjugacy classes in $\mathbf{G}(\Qp)$ were parametrised only up to outer conjugation.
Since $\mathbf{G}(\Zp)$ is invariant by an outer automorphism of $\mathbf{G}$, for any $\gamma_0, \gamma_0' \in \mathbf{G}(\Qp)$ which are conjugate by an outer automorphism of $\mathbf{G}_{\Qp}$, the orbital integrals for $f(g_p)dg_p$ at $\gamma_0$ and $\gamma_0'$ are equal.
Of course the above formula for the orbital integral is valid for both.
\end{rema}

\subsubsection{Local densities and global volumes}
\label{sectionVolumes}

To complete the computation of adèlic orbital integrals we still have to evaluate the denominators in formula \ref{formulaorbint} and the global volumes.
Formulae for local densities and Smith-Minkowski-Siegel mass formulae are just what we need.
But we will use the point of view suggested by \cite{GrossMot} and used in \cite{GrossPollack}, i.e.\ fix canonical Haar measures to see local orbital integrals as numbers.
For this we need to work in a slightly more general setting than cyclotomic fields.

If $k$ is a number field or a $p$-adic field, denote by $\mathcal{O}_k$ its ring of integers.
If $k$ is a number field $\A_k = k \otimes_{\Q} \A$ will denote the adèles of $k$.

Let $k$ be a number field or a local field of characteristic zero, and let $K$ be a finite commutative étale $k$-algebra such that $\dim_k K \leq 2$, i.e.\ $K =k$ or $k \times k$ or $K$ is a quadratic field extension of $k$.
Let $\tau$ be such that $\mathrm{Aut}_k(K) = \{\mathrm{Id}_K, \tau \}$.
This determines $\tau$.
Let $V$ be a vector space over $K$ of dimension $r \geq 0$.
Let $\alpha \in \{ 1, -1\}$, and assume that $\alpha = 1$ if $\dim_k K = 2$.
Assume that $V$ is endowed with a non-degenerate $\tau$-sesquilinear form $\langle \cdot, \cdot \rangle$ such that for any $v_1,v_2 \in V$ we have $\langle v_2, v_1 \rangle = \alpha \tau\left( \langle v_1, v_2 \rangle \right)$.
Let $\mathbf{G} = \mathbf{Aut}(V, \langle \cdot, \cdot \rangle)^0$ be the connected reductive group over $k$ associated with this datum.
Then $\mathbf{G}$ is a special orthogonal ($K=k$ and $\alpha=1$), symplectic ($K=k$ and $\alpha=-1$), unitary ($K/k$ is a quadratic field extension and $\alpha=1$) or general linear ($K = k \times k$ and $\alpha=1$) group.

If $k$ is a number field, by Weil \cite{Weil} the Tamagawa number $\tau(\mathbf{G})$ equals $2$ (resp.\ $1$) in the orthogonal case if $r \geq 2$ and $V$ is not a hyperbolic plane (resp.\ if $r = 1$ or $V$ is a hyperbolic plane), $1$ in the symplectic case, $2$ in the unitary case if $r > 0$ and $1$ in the general linear case.

If $k$ is a $p$-adic field, consider a lattice $N$ in $V$, i.e.\ a finite free $\mathcal{O}_K$-module $N \subset V$ such that $V = K N$.
Denote $N^{\vee} = \{ v \in V\,|\,\forall w \in N,\, \langle v, w \rangle \in \mathcal{O}_K \}$.
If $\langle \cdot, \cdot \rangle|_{N \times N}$ takes values in $\mathcal{O}_K$ then $N^{\vee} \supset N$ and we can consider $[N^{\vee} : N]$, i.e.\ the cardinality of the finite abelian group $N^{\vee} / N$.
In general define $[N^{\vee} : N]$ as $[N^{\vee} : N^{\vee} \cap N ] / [N : N^{\vee} \cap N]$.
Recall also \cite{GanYu}[Definition 3.5] the \emph{density} $\beta_{N}$ associated with $(N, \langle \cdot, \cdot \rangle)$.

In \cite{GrossMot} Gross associates a motive $M$ of Artin-Tate type to any reductive group over a field.
For the groups $\mathbf{G}$ defined above, letting $n$ be the rank of $\mathbf{G}$, we have
$$ M = \begin{cases} \bigoplus_{x=1}^n \Q(1-2x) & \text{orthogonal case with } r \text{ odd and symplectic case,} \\ \chi\Q(1-n) \oplus \bigoplus_{x=1}^{n-1} \Q(1-2x) & \text{orthogonal case with } r>0 \text{ even,} \\ \bigoplus_{x=1}^n \chi^x \Q(1-x) & \text{unitary and general linear cases.} \end{cases} $$
In the orthogonal case with $r>0$ even let $(-1)^n D$ be the discriminant of $(V, \langle \cdot, \cdot \rangle)$ (i.e.\ the determinant of the Gram matrix), then $\chi$ is defined as the character $\mathrm{Gal}(k(\sqrt{D})/k) \rightarrow \{ \pm 1 \}$ which is non-trivial if $D$ is not a square in $k$.
In the general linear case $\chi$ is trivial, and in the unitary case $\chi$ is the non-trivial character of $\mathrm{Gal}(K/k)$.
For L-functions and $\epsilon$-factors we will use the same notations as \cite{GrossMot}.

If $k$ is a number field $D_k$ will denote the absolute value of its discriminant.
For $K=k$ or $K = k \times k$ denote $D_{K/k} = 1$, whereas for a quadratic field extension $K$ of $k$ we denote $D_{K/k} = |N_{K/\Q}(\mathfrak{D}_{K/k})|$ where $\mathfrak{D}_{K/k}$ is the different ideal of $K/k$ and the absolute value of the ideal $m\Z$ of $\Z$ is $m$ if $m \geq 1$.
There are obvious analogues over any $p$-adic field, and $D_k$ (resp.\ $D_{K/k}$) is the product of $D_{k_v}$ (resp.\ $D_{K_v/k_v}$ where $K_v = k_v \otimes_k K$) over the finite places $v$ of $k$.

For $(k, K, \alpha, V \langle \cdot, \cdot \rangle)$ (local or global) as above define as in \cite{GanYu}
$$ n(V) = \begin{cases} r + \alpha & \text{if } K=k,\\ r & \text{if } \dim_k K = 2 \end{cases} $$
and
$$ \mu = \begin{cases} 2^r & \text{in the orthogonal case with } r \text{ even,} \\ 2^{(r+1)/2} & \text{in the orthogonal case with } r \text{ odd,} \\ 1 & \text{in the symplectic, unitary and general linear cases.} \end{cases} $$
Finally, consider the case where $k = \R$ and $\mathbf{G}(\R)$ has discrete series, i.e.\ the Euler-Poincaré measure on $\mathbf{G}(\R)$ is non-zero, i.e.\ $\mathbf{G}$ has a maximal torus $\mathbf{T}$ which is anisotropic.
Recall Kottwitz's sign $e(\mathbf{G}) = (-1)^{q(\mathbf{G})}$ and the positive rational number $c(\mathbf{G})$ defined in \cite{GrossMot}[§8].
Explicitly,
$$ c(\mathbf{G}) = \begin{cases} 1 & \text{in the symplectic case,} \\ 2^n / \binom{n}{\lfloor a/2 \rfloor} & \text{in the orthogonal case with signature }(a,b), b \text{ even,} \\ 2^n / \binom{n}{a} & \text{in the unitary case with signature }(a,b). \end{cases} $$

The following theorem is a reformulation of the mass formula \cite{GanYu}[Theorem 10.20] in our special cases.
\begin{theo}
Let $k$ be a totally real number field and let $K$, $\alpha$, $(V, \langle \cdot, \cdot \rangle)$ and $\mathbf{G}$ be as above.
Let $M$ denote the Gross motive of $\mathbf{G}$.
Assume that for any real place $v$ of $k$, $\mathbf{G}(k_v)$ has discrete series.
Define a signed Haar measure $\nu = \prod_v \nu_v$ on $\mathbf{G}(\A_k)$ as follows.
For any real place $v$ of $k$, $\nu_v$ is the Euler-Poincaré measure on $\mathbf{G}(k_v)$.
For any finite place $v$ of $k$, $\nu_v$ is the canonical measure $L_v(M^{\vee}(1)) |\omega_{\mathbf{G}_{k_v}}|$ on $\mathbf{G}(k_v)$ (see \cite{GrossMot}[§4]).
In particular, for any finite place $v$ such that $\mathbf{G}_{k_v}$ is unramified, the measure of a hyperspecial compact subgroup of $\mathbf{G}(k_v)$ is one.
Then for any $\mathcal{O}_K$-lattice $N$ in $V$,
\begin{align*} \int_{\mathbf{G}(k) \backslash \mathbf{G}(\A_k)} \nu &= \tau(\mathbf{G}) \times L(M) \times \frac{D_k^{\dim \mathbf{G} / 2} D_{K/k}^{r(r+1)/4}}{\epsilon(M)} \times \prod_{v | \infty} \frac{(-1)^{q(\mathbf{G}_{k_v})}}{c(\mathbf{G}_{k_v})} \\
& \qquad \times \mu^{\dim_{\Q} k} \prod_{v \text{ finite}} \frac{[N_v^{\vee} : N_v]^{n(V)/2} \times \nu_v\left( \mathbf{G}(k_v) \cap \mathrm{GL}(N_v) \right)}{L_v(M^{\vee}(1)) \beta_{N_v}} \end{align*}
\end{theo}
\begin{proof}
To get this formula from \cite{GanYu}[Theorem 10.20], use the comparison of measure at real places \cite{GrossMot}[Proposition 7.6], the fact that $L_v(M^{\vee}(1)) \beta_{N_v} = 1$ for almost all finite places of $k$, and the functional equation $\Lambda(M) = \epsilon(M) \Lambda(M^{\vee}(1))$ (see \cite{GrossMot}[9.7]).
\end{proof}
Note that the choice of $\nu$ at the finite places does not play any role.
This choice was made to compare with the very simple formula \cite{GrossMot}[Theorem 9.9]:
\begin{equation} \label{Grossthm99} \int_{\mathbf{G}(k) \backslash \mathbf{G}(\A_k)} \nu\ =\ \tau(\mathbf{G}) \times L(M) \times \prod_{v | \infty} \frac{(-1)^{q(\mathbf{G}_{k_v})}}{c(\mathbf{G}_{k_v})}. \end{equation}
We obtain that under the hypotheses of the theorem,
\begin{equation} \label{compGYwithG} \prod_{v \text{ finite}} \nu_v\left( \mathbf{G}(k_v) \cap \mathrm{GL}(N_v) \right) = \frac{\epsilon(M) \mu^{-\dim_{\Q} k}}{D_k^{\dim \mathbf{G} / 2} D_{K/k}^{r(r+1)/4}} \prod_{v \text{ finite}} \frac{L_v(M^{\vee}(1)) \beta_{N_v}}{[N_v^{\vee} : N_v]^{n(V)/2}}. \end{equation}
We can compute explicitely
$$ \frac{\epsilon(M)}{D_k^{\dim \mathbf{G} / 2} D_{K/k}^{r(r+1)/4}} = \begin{cases} D_{K/k}^{-n/2} & \text{in the unitary case if } r=n \text{ is even,} \\ \left|N_{k/\Q}(\delta) \right|^{n-1/2} & \text{in the orthogonal case if } r \text{ is even,} \\ 1 & \text{otherwise,} \end{cases} $$
where in the second case $(-1)^nD$ is the discriminant of $\langle \cdot, \cdot \rangle$ and $\delta$ is the discriminant of $k(\sqrt{D})/k$.
As the proof of the following proposition shows, the factor $\mu^{-\dim_{\Q} k}$, which is nontrivial only in the orthogonal cases, is local at the dyadic places.
\begin{prop} \label{proplocalmeasure}
Let $p$ be a prime.
Let $k_0$ be a $p$-adic field and let $(K_0, \alpha, V_0, \langle \cdot, \cdot \rangle_0)$ and $\mathbf{G}_0$ be as above.
Let $\nu_0$ be the canonical Haar measure $L(M^{\vee}(1)) | \omega_{\mathbf{G}_0}|$ on $\mathbf{G}_0(k_0)$.
If $p=2$, $K_0=k_0$ and $\alpha=1$, let $x_0 = \mu^{-\dim_{\Q_2} k_0}$, otherwise let $x_0 = 1$.
Then for any $\mathcal{O}_{K_0}$-lattice $N_0$ in $V_0$,
\begin{align*} \nu_0\left( \mathbf{G}_0(k_0) \cap \mathrm{GL}(N_0) \right) & = L(M^{\vee}(1)) \times x_0 \times \beta_{N_0} \times [N_0^{\vee} : N_0]^{-n(V_0)/2} \\ & \quad \times \begin{cases} D_{K_0/k_0}^{-n/2} & \text{in the unitary case if } r=n \text{ is even,} \\ \left|N_{k_0/\Qp}(\delta_0) \right|^{n-1/2} & \text{in the orthogonal case if } r \text{ is even,} \\ 1 & \text{otherwise,} \end{cases} \end{align*}
where in the second case $(-1)^nD_0$ is the discriminant of $\langle \cdot, \cdot \rangle_0$ and $\delta_0$ is the discriminant of $k_0(\sqrt{D_0})/k_0$.
\end{prop}
\begin{proof}
We apologise for giving a global proof of this local statement.
We only give details for the hardest case of orthogonal groups.

When $p>2$ and the symmetric bilinear form $\langle \cdot, \cdot \rangle_0 |_{N_0 \times N_0}$ is integer-valued and non-degenerate, $\mathbf{G}_0$ is the generic fiber of a reductive group over $\mathcal{O}_{k_0}$ and the equality is obvious.
Note that this does not apply for $p=2$, even assuming further that the quadratic form $v \mapsto \langle v, v \rangle_0/2$ is integer-valued on $N_0$, because the local density $\beta_{N_0}$ is defined using the bilinear form $\langle \cdot, \cdot \rangle_0$, not the quadratic form $v \mapsto \langle v, v \rangle_0 / 2$.

Next consider the case $p=2$ and $N_0$ arbitrary.
By Krasner's lemma there exists a totally real number field $k$ and a quadratic vector space $(V, \langle \cdot, \cdot \rangle)$ which is positive definite at the real places of $k$ and such that $k$ has a unique dyadic place $v_0$ and $(k_0, V_0, \langle \cdot, \cdot \rangle_0) \simeq (k_{v_0}, k_{v_0} \otimes_k V, \langle \cdot, \cdot \rangle)$.
Let $S$ be the finite set of finite places $v \neq v_0$ of $k$ such that $(k_v \otimes_k V, \langle \cdot, \cdot \rangle)$ is ramified, i.e.\ does not admit an integer-valued non-degenerate $\mathcal{O}_{k_v}$-lattice.
For any $v \in S$ there is a finite extension $E^{(v)}$ of $k_v$ over which $(k_v \otimes_k V, \langle \cdot, \cdot \rangle)$ becomes unramified.
By Krasner's lemma again there exists a finite extension $k'$ of $k$ which is totally split over the real places of $k$ and over $v_0$ and such that for any $v \in S$, the $k_v$-algebra $k_v \otimes_k k'$ is isomorphic to a product of copies of $E^{(v)}$.
Let $S_0$ be the set of dyadic places of $k'$, i.e. the set of places of $k'$ above $v_0$.
There exists a lattice $N'$ in $k' \otimes_k V$ such that for any finite $v \not\in S_0$ the symmetric bilinear form $\langle \cdot, \cdot \rangle|_{N'_v \times N'_v}$ is integer-valued and non-degenerate, and for any $v \in S_0$ we have $\langle \cdot, \cdot \rangle_{N'_v \times N'_v} \simeq \langle \cdot, \cdot \rangle_0|_{N_0 \times N_0}$.
Applying formula \ref{compGYwithG} we obtain the desired equality to the power $\mathrm{card}(S_0)$, which is enough because all the terms are positive real numbers.
Having established the dyadic case, the general case can be established similarly.

The unitary case is similar but simpler, because the dyadic places are no longer exceptional and it is sufficient to take a quadratic extension $k'/k$ in the global argument.
The symplectic and general linear cases are even simpler.
\end{proof}
\begin{rema}
\begin{enumerate}
\item
In this formula, one can check case by case that the product of $[N_0^{\vee} : N_0]^{-n(V_0)/2}$ and the last term is always rational, as expected since all other terms are rational by definition.
\item
We did not consider the case where $\alpha=-1$ and $K/k$ is a quadratic field extension, i.e.\ the case of antihermitian forms, although this case is needed to compute orbital integrals for symplectic groups.
If $y \in K^{\times}$ is such that $\tau(y)=-y$, multiplication by $y$ induces a bijection between hermitian and antihermitian forms, and of course the automorphism groups are equal.
\item
There are other types of classical groups considered in \cite{GanYu} and which we left aside.
For a central simple algebra $K$ over $k$ with $\dim_k K = 4$ (i.e.\ $K=M_2(k)$ or $K$ is a quaternion algebra over $k$) they also consider hermitian (resp.\ antihermitian) forms over a $K$-vector space.
The resulting automorphism groups are inner forms of symplectic (resp.\ even orthogonal) groups.
Using the same method as in the proof of the proposition leads to a formula relating the local density $\beta_{N_0}$ to the canonical measure of $\mathrm{Aut}(N_0)$ in these cases as well.
\end{enumerate}
\end{rema}

We use the canonical measure defined by Gross (called $\nu_v$ above) when computing local orbital integrals.
In the previous section we explained how to compute the numerators in formula \ref{formulaorbint} for the local orbital integrals.
Proposition \ref{proplocalmeasure} reduces the computation of the denominators to that of local densities.
Using an elegant method of explicitly constructing smooth models, Gan and Yu \cite{GanYu} give a formula for $\beta_{N_0}$ for $p>2$ in general and for $p=2$ only in the case of symplectic and general linear groups and in the case of unitary groups if $K_0/k_0$ is unramified.
Using a similar method Cho \cite{ChoQuad} gives a formula in the case of special orthogonal groups when $p=2$ and $k_0 / \Q_2$ is unramified.
This is enough for our computations since we only need the case $k_0 = \Q_2$.
For $m \geq 1$ and $\zeta = \zeta_m$ the quadratic extension $\Q(\zeta)/\Q(\zeta+\zeta^{-1})$ is ramified over a dyadic place if and only if $m$ is a power of $2$.
In this case the different $\mathfrak{D}_{\Q_2(\zeta)/\Q_2(\zeta+\zeta^{-1})}$ is generated by a uniformiser of $\Q_2(\zeta+\zeta^{-1})$, which is the minimal ramification that one can expect from a ramified quadratic extension in residue characteristic $2$.
Cho \cite{ChoHermi}[Case 1] also proved an explicit formula for the local density in this case.
To be honest \cite{ChoHermi} only asserts it in the case where $k_0$ is unramified over $\Q_2$.
Nevertheless the proof in ``Case 1'' does not use this assumption.
This completes the algorithm to compute the local orbital integrals in all cyclotomic cases over $\Q$.
Note that the result is rational and the computations are exact (i.e.\ no floating point numbers are used).

Finally, the global volume is evaluated using Gross' formula \ref{Grossthm99}.
The value of $L(M)$ is known to be rational and computable by \cite{SiegelLfunc}.
However, we only need the values of $L(M)$ for $M$ which is a direct sum of Tate twists of \emph{cyclotomic} Artin motives (concretely, representations of $\mathrm{Gal}(E/F)$ where $E$ is contained in a cyclotomic extension of $\Q$).
Thus we only need the values of Dirichlet L-functions at non-negative integers, i.e.\ the values of generalised Bernoulli numbers (see e.g.\ \cite{Washington}).

\begin{rema}
Formally it is not necessary to use the results of \cite{GrossMot} to compute the factors $\mathrm{Vol}(\mathbf{I}(\Q) \backslash \mathbf{I}(\A))$ in formula \ref{formulaTell}, the mass formula in \cite{GanYu} along with the formulae for the local densities $\beta_{N_0}$ would suffice.
Apart from the fact that it is less confusing and more elegant to clearly separate local and global measures, using Gross' canonical measure, which is compatible between inner forms by definition, allows to compute $\kappa$-orbital integrals once we have computed orbital integrals.
The fundamental lemma gives a meaningful way to check the results of computations of orbital integrals.
More precisely we need the formulation of the fundamental lemma for semisimple singular elements \cite{KottEllSing}[Conjecture 5.5] which has been reduced to the semisimple regular case by \cite{KoTam}[§3] and \cite{LanSheGF}[Lemma 2.4.A].
For an unramified endoscopic group the fundamental lemma for the unit of the unramified Hecke algebra at regular semisimple elements is a consequence of the work of Hales, Waldspurger and Ngô.
The case of a ramified endoscopic group is \cite{KottEllSing}[Proposition 7.5]: the $\kappa$-orbital integral simply vanishes.
\end{rema}

\subsubsection{Short description of the global algorithm}
\label{SummaryAlgo}

Let $\mathbf{G}$ be one of $\mathbf{SO}_{2n+1}$ or $\mathbf{Sp}_{2n}$ or $\mathbf{SO}_{4n}$ over $\Z$, let $\prod_p f_p$ be the characteristic function of $\mathbf{G}(\widehat{\Z})$ and $\prod_p dg_p$ the Haar measure on $\mathbf{G}(\A_f)$ such that $\mathbf{G}(\widehat{\Z})$ has measure one.
Let $\lambda$ be a dominant weight for $\mathbf{G}_{\C}$ and let $f_{\infty, \lambda}(g_{\infty}) dg_{\infty}$ be the distribution on $\mathbf{G}(\R)$ defined in section \ref{EPinftySection}.
Denote $f(g)dg = f_{\infty, \lambda}(g_{\infty}) dg_{\infty} \prod_p f_p(g_p) dg_p$.
We give a short summary of the algorithm computing $T_{\mathrm{ell}}(f(g)dg)$ for a family of dominant weights $\lambda$, by outlining the main steps.
Realise $\mathbf{G}$ as $\mathbf{SO}(\Lambda, q)$ (resp.\ $\mathbf{Sp}(\Lambda, a)$) where $\Lambda$ is a finite free $\Z$-module endowed with a regular quadratic form $q$ (resp.\ nondegenerate alternate form $a$).
Denote $N = \mathrm{rank}_{\Z}(\Lambda)$.
\begin{enumerate}
\item Enumerate the possible characteristic polynomials in the standard representation of $\mathbf{G}$ for $\gamma \in C(\mathbf{G}(\Q))$.
That is, enumerate the polynomials $P \in \Q[X]$ unitary of degree $d$ such that all the roots of $P$ are roots of unity, and the multiplicity of $-1$ as root of $P$ is even.
\item For each such $P$, and for any prime number $p$, in $\Qp[X]$ write $P = \prod_i P_i$ as in section \ref{conjclassfieldSection}.
For any $i$, enumerate the finite set of isomorphism classes of quadratic or hermitian (resp.\ alternate or antihermitian) lattices $(\Lambda_i, h_i)$ as in section \ref{conjclassZpSection}.
For almost all primes $p$, the minimal polynomial $\mathrm{rad}(P) = P / \mathrm{gcd}(P,P')$ is separable modulo $p$, there is a unique isomorphism class $(\Lambda_i, h_i)$ to consider and $h_i$ is non-degenerate.
Thus we only need to consider a finite set of primes.
\item The combinations of these potential local data determine a finite set of conjugacy classes in $\mathbf{G}(\Q)$.
\item For any such conjugacy class over $\Q$, compute the local orbital integrals using section \ref{compOrbIntSection} and Proposition \ref{proplocalmeasure}.
Compute the global volumes using Gross' formula \ref{Grossthm99}.
\item Let $C'$ be the set of $\mathbf{G}(\Qbar)$-conjugacy classes in $C(\mathbf{G}(\Q))$.
For $c \in C'$ define the ``mass'' of $c$
$$m_c = \sum_{\mathrm{cl}(\gamma) \in c} \frac{\mathrm{Vol}(\mathbf{I}(\Q) \backslash \mathbf{I}(\A))}{\mathrm{card}(\mathrm{Cent}(\gamma, \mathbf{G}(\Q)) / \mathbf{I}(\Q)))} $$
so that
$$T_{\mathrm{ell}}(f(g)dg) = \sum_{c \in C'} m_c \mathrm{Tr}(c\,|\,V_{\lambda}).$$
Using Weyl's character formula, we can finally compute $T_{\mathrm{ell}}(f(g)dg)$ for the dominant weights $\lambda$ we are interested in.
Some conjugacy classes $c \in C'$ are singular, so that a refinement of Weyl's formula is needed: see \cite{CheClo}[Proposition 1.9] and \cite{ChRe}[Proposition 2.3].
\end{enumerate}

We give tables of the masses $m_c$ in section \ref{tablesMasses}, for the groups of rank $\leq 4$.
Our current implementation allows to compute these masses at least up to rank $6$ (and for $\mathbf{Sp}_{14}$ also), but starting with rank $5$ they no longer fit on a single page.

\begin{rema}
In the orthogonal case the group $\mathbf{G}$ is not simply connected and thus in $\mathbf{G}(\Q)$ there is a distinction between stable conjugacy and conjugacy in $\mathbf{G}(\Qbar)$.
However, if $\gamma,\gamma' \in C(\mathbf{G}(\Q))$ both contribute non-trivially to $T_{\mathrm{ell}}(f(g)dg)$ and are conjugated in $\mathbf{G}(\Qbar)$, then they are stably conjugate.
Indeed their spinor norms have even valuation at every finite prime, and are trivial at the archimedean place since they each belong to a compact connected torus, therefore their spinor norms are both trivial.
This implies that they lift to elements $\tilde{\gamma}, \tilde{\gamma}'$ in the spin group $\mathbf{G}_{\mathrm{sc}}(\Q)$, and moreover we can assume that $\tilde{\gamma}$ and $\tilde{\gamma}'$ are conjugated in $\mathbf{G}_{\mathrm{sc}}(\Qbar)$, which means that they are stably conjugate.

This observation allows to avoid unnecessary computations: if the spinor norm of $\gamma$ is not equal to $1$, the global orbital integral $O_{\gamma}(f(g)dg)$ vanishes.
\end{rema}

\subsection{Computation of the parabolic terms using elliptic terms for groups of lower semisimple rank}
\label{subsectionCompPara}

In the previous sections we gave an algorithm to compute the elliptic terms in Arthur's trace formula in \cite{ArthurL2}.
After recalling the complete geometric side of the trace formula, i.e.\ the parabolic terms, we explain how the archimedean contributions to these terms simplify in our situation where the functions $f_p$ at the finite places have support contained in a compact subgroup.
The result is that we can express the parabolic terms very explicitely (perhaps too explicitely) using elliptic terms for groups of lower semisimple rank in section \ref{explicitpara}.

\subsubsection{Parabolic terms}
\label{ParaTermsSection}

Let us recall the geometric side of the trace formula given in \cite{ArthurL2}[§6].
We will slightly change the formulation by using Euler-Poincaré measures on real groups instead of transferring Haar measures to compact inner forms.
The translation is straightforward using \cite{KoTam}[Theorem 1].
Let $\mathbf{G}$ be one of $\mathbf{SO}_{2n+1}$, $\mathbf{Sp}_{2n}$ or $\mathbf{SO}_{4n}$.
Of course the following notions and Arthur's trace formula apply to more general groups.

First we recall the definition of the constant term at the finite places.
Let $p$ be a finite prime, and denote $K = \mathbf{G}(\Zp)$.
Let $\mathbf{P} = \mathbf{M} \mathbf{N}$ be a parabolic subgroup of $\mathbf{G}$ having unipotent radical $\mathbf{N}$ admitting $\mathbf{M}$ as a Levi subgroup.
Since $K$ is a hyperspecial maximal compact subgroup of $\mathbf{G}(\Qp)$ it is ``good'': there is an Iwasawa decomposition $\mathbf{G}(\Qp) = K \mathbf{P}(\Qp)$.
When $p$ is not ambiguous write $\delta_{\mathbf{P}}(m) = |\det(m\,|\,\mathrm{Lie}(\mathbf{N}))|_p$.
In formulae we require the Haar measures on the unimodular groups $\mathbf{G}(\Qp)$, $\mathbf{M}(\Qp)$ and $\mathbf{N}(\Qp)$ to be compatible in the sense that for any continuous $h : \mathbf{G}(\Qp) \rightarrow \C$ having compact support,
$$ \int_{\mathbf{G}(\Qp)}\!\!\!\! h(g) dg = \int_{K \times \mathbf{N}(\Qp) \times \mathbf{M}(\Qp)}\!\!\!\!\!\!\!\!\!\!\!\! h(knm)\,dk\,dn\,dm = \int_{K \times \mathbf{N}(\Qp) \times \mathbf{M}(\Qp)}\!\!\!\!\!\!\!\!\!\!\!\! h(kmn) \delta_{\mathbf{P}}(m) \,dk\,dn\,dm.$$
If $f_p(g)dg$ is a smooth compactly supported distribution on $\mathbf{G}(\Qp)$, the formula
$$ f_{p,\mathbf{M}}(m) = \delta_{\mathbf{P}}(m)^{1/2} \int_K \int_{\mathbf{N}(\Qp)} f_p(kmnk^{-1}) dn dk $$
defines a smooth compactly supported distribution $f_{p,\mathbf{M}}(m)dm$ on $\mathbf{M}(\Qp)$.
Although it seems to depend on the choice of $\mathbf{N}$ and the good compact subgroup $K$, the orbital integrals of $f_{p, \mathbf{M}}(m)dm$ at semisimple $\mathbf{G}$-regular elements of $\mathbf{M}(\Qp)$ only depend on $f_p$ (see \cite{vanDijk}[Lemma 9]).
The case of arbitrary semisimple elements follows using \cite{KazhdanCusp}[Theorem 0].
When $f_p$ is the characteristic function $\mathbf{1}_{\mathbf{G}(\Zp)}$ of $\mathbf{G}(\Zp)$ (and $\mathrm{vol}(\mathbf{G}(\Zp))=1$), the fact that $\mathbf{T}_0$ is defined over $\Zp$ and the choice $K = \mathbf{G}(\Zp)$ imply that for any choice of $\mathbf{N}$, $ f_{p, \mathbf{M}} = \mathbf{1}_{\mathbf{M}(\Zp)}$ (if $\mathrm{vol}(\mathbf{M}(\Zp))=1$).

We can now define the factors appearing on the geometric side of the trace formula.
As for elliptic terms, consider a smooth compactly supported distribution $\prod_p f_p(g_p)dg_p$ on $\mathbf{G}(\A_f)$.
Fix a split maximal torus $\mathbf{T}_0$ of $\mathbf{G}$ (over $\Z$).
The geometric side is a sum over Levi subgroups $\mathbf{M}$ containing $\mathbf{T}_0$, they are also defined over $\Z$.
For such $\mathbf{M}$, denote by $\mathbf{A}_{\mathbf{M}}$ the connected center of $\mathbf{M}$ and let $C(\mathbf{M}(\Q))$ be the set of semisimple conjugacy classes of elements $\gamma \in \mathbf{M}(\Q)$ which belong to a maximal torus of $\mathbf{M}_{\R}$ which is anisotropic modulo $(\mathbf{A}_{\mathbf{M}})_{\R} = \mathbf{A}_{\mathbf{M}_{\R}}$.
If $\gamma$ is (a representative of) an element of $C(\mathbf{M}(\Q))$, let $\mathbf{I}$ denote the connected centraliser of $\gamma$ in $\mathbf{M}$.
Define $\iota^{\mathbf{M}}(\gamma) = |\mathrm{Cent}(\gamma, \mathbf{M}(\Q))/\mathbf{I}(\Q)|$.
For any finite prime $p$, to $f_p(g_p) dg_p$ we associate the complex Haar measure $O_{\gamma}(f_{p, \mathbf{M}})$ on $\mathbf{I}(\Qp)$.
For $p$ outside a finite set (containing the primes at which $\mathbf{I}$ is ramified), the measure of a hyperspecial maximal compact subgroup of $\mathbf{I}(\Qp)$ is $1$.
Define a complex Haar measure on $\mathbf{I}(\mathbb{A})/\mathbf{A}_{\mathbf{M}}(\mathbb{A})$ as follows:
\begin{itemize}
\item
Give $\mathbf{I}(\R)/\mathbf{A}_{\mathbf{M}}(\R)$ its Euler-Poincaré measure.
It is nonzero by our assumption on $\gamma$.
\item
Give $\mathbf{A}_{\mathbf{M}}(\Qp)$ its Haar measure such that its maximal compact subgroup (in the case at hand $\mathbf{A}_{\mathbf{M}}(\Zp)$) has measure $1$, and endow $\mathbf{I}(\Qp)/\mathbf{A}_{\mathbf{M}}(\Qp)$ with the quotient measure.
\end{itemize}

Now fix a dominant weight $\lambda$ for $\mathbf{G}$ and denote $\tau = \lambda + \rho$ (where $2 \rho$ is the sum of the positive roots) the associated infinitesimal character.
For $f(g)dg=f_{\infty,\lambda}(g_{\infty}) dg_{\infty} \prod_p f_p(g_p)dg_p$, the last ingredient occurring in $T_{\mathrm{geom}}(f(g)dg)$ is the continuous function $\gamma \mapsto \Phi_{\mathrm{M}}(\gamma, \tau)$ defined for semisimple $\gamma \in \mathbf{M}(\R)$ which belong to a maximal torus of $\mathbf{M}_{\R}$ which is anisotropic modulo $(\mathbf{A}_{\mathbf{M}})_{\R}$.
This function will be defined in terms of characters of discrete series and studied at compact elements $\gamma$ in section \ref{SectionCharDS}.
If $\gamma$ does not satisfy these properties define $\Phi_{\mathrm{M}}(\gamma, \tau)=0$.

The geometric side $T_{\mathrm{geom}}(f(g)dg)$ of the trace formula is
\begin{equation} \label{formulaTgeom}
\sum_{\mathbf{M} \supset \mathbf{T}_0} \left(\frac{-1}{2}\right)^{\dim \mathbf{A}_{\mathbf{M}}} \frac{|W(\mathbf{T}_0, \mathbf{M})|}{|W(\mathbf{T}_0, \mathbf{G})|} \sum_{\gamma \in C(\mathbf{M}(\mathbb{Q}))} \frac{\mathrm{vol}\left(\mathbf{I}(\Q) \backslash \mathbf{I}(\A) / \mathbf{A}_{\mathbf{M}}(\A) \right)}{\mathrm{card}\left(\mathrm{Cent}(\gamma, \mathbf{M}(\Q))/\mathbf{I}(\Q) \right)} \Phi_{\mathbf{M}}(\gamma, \tau).
\end{equation}
After the definition of the function $\Phi_{\mathbf{M}}$ it will be clear that the term corresponding to $\mathbf{M} = \mathbf{G}$ is $T_{\mathrm{ell}}(f(g)dg)$.

\subsubsection{Sums of averaged discrete series constants}

Harish-Chandra gave a formula for the character of discrete series representations of a real reductive group at regular elements of any maximal torus.
This formula is similar to Weyl's character formula but it also includes certain integers which can be computed inductively.
In the case of averaged discrete series this induction is particularly simple.
We recall the characterisation of these integers given in \cite{GoKoMPh}[§3] and compute their sum and alternate sum.
When the support of $\prod_p f_p(g_p)dg_p$ is contained in a compact subgroup of $\mathbf{G}(\A_f)$, in the trace formula only these alternate sums need to be computed, not the individual constants.

Let $X$ be a real finite-dimensional vector space and $R$ a reduced root system in $X^*$.
Assume that $-\mathrm{Id} \in W(R)$, i.e. any irreducible component of $R$ is of type $A_1$, $B_n$ ($n \geq 2$), $C_n$ ($n \geq 3$), $D_{2n}$ ($n \geq 2$), $E_7$, $E_8$, $F_4$ or $G_2$.
If $R_1$ is a subsystem of $R$ having the same property, letting $R_2$ be the subsystem of $R$ consisting of roots orthogonal to all the roots in $R_1$, $-\mathrm{Id}_{\R R_2} \in W(R_2)$ by \cite{BouLie456}[ch. V, §3, Proposition 2], and $\mathrm{rank}(R) = \mathrm{rank}(R_1)+\mathrm{rank}(R_2)$.
In particular for $\alpha \in R$, $R_{\alpha} := \{ \beta \in R\ |\ \alpha(\beta^{\vee})=0 \}$ is a root system in $Y^*$ where $Y = \ker \alpha$.

Recall that $X_{\mathrm{reg}} := \{x \in X\ |\ \forall \alpha \in R,\ \alpha(x) \neq 0 \}$, and define $X^*_{\mathrm{reg}}$ similarly with respect to $R^{\vee}$.
For $x \in X_{\mathrm{reg}}$ we denote by $\Delta_x$ the basis of simple roots of $R$ associated with the chamber containing $x$.
There is a unique collection of functions $\bar{c}_R : X_{\mathrm{reg}} \times X^*_{\mathrm{reg}} \rightarrow \Z$ for root systems $R$ as above such that:
\begin{enumerate}
\item $\bar{c}_{\emptyset}(0,0)=1$,
\item for all $(x,\lambda) \in X_{\mathrm{reg}} \times X^*_{\mathrm{reg}}$ such that $\lambda(x)>0$, $\bar{c}_R(x, \lambda) = 0$,
\item for all $(x,\lambda) \in X_{\mathrm{reg}} \times X^*_{\mathrm{reg}}$ and $\alpha \in \Delta_x$, $\bar{c}_R(x, \lambda) + \bar{c}_R(s_{\alpha}(x), \lambda) = 2 \bar{c}_{R_{\alpha}}(y, \lambda|_Y)$ where $Y = \ker \alpha$ and $y = (x+s_{\alpha}(x))/2$.
\end{enumerate}
In the third property note that for any $\beta \in R \smallsetminus \{ \pm \alpha \}$ such that $\beta(x)>0$, $\beta(y)>0$: writing $\beta = \sum_{\gamma \in \Delta_x} n_{\gamma} \gamma$ with $n_{\gamma} \geq 0$, we have
\begin{equation} \label{betaypos} \beta(y) = \beta(x) - \frac{\alpha(x) \beta(\alpha^{\vee})}{2} = \sum_{\gamma \in \Delta_x \smallsetminus \{ \alpha \}} n_{\gamma} \left(\gamma(x) - \frac{\gamma(\alpha^{\vee}) \alpha(x)}{2} \right) > 0. \end{equation}
In the second property we could replace ``$\lambda(x)>0$'' by the stronger condition that $R \neq \emptyset$ and $x$ and $\lambda$ define the same order: $\{\alpha \in R\ |\ \alpha(x)>0 \} = \{\alpha \in R\ |\ \lambda(\alpha^{\vee})>0 \}$.
By induction $\bar{c}_R$ is locally constant, and $W(R)$-invariant for the diagonal action of $W(R)$ on $X_{\mathrm{reg}} \times X^*_{\mathrm{reg}}$.

The existence of these functions follows from Harish-Chandra's formulae and the existence of discrete series for the split semisimple groups over $\R$ having a root system as above.
However, \cite{GoKoMPh} give a direct construction.

Let $x_0 \in X_{\mathrm{reg}}$ and $\lambda_0 \in X^*_{\mathrm{reg}}$ define the same order.
For $w \in W(R)$ define $d(w)=\bar{c}_R(x_0, w(\lambda_0))=\bar{c}_R(w^{-1}(x_0), \lambda_0)$.

\begin{prop} \label{propaltsumcbar}
Let $R$ be a root system as above, and denote by $q(R)$ the integer $\left( |R|/2 + \mathrm{rank}(R) \right)/2$.
Then
$$ \sum_{w \in W(R)} d(w) = |W(R)| \text{  and  } \sum_{w \in W(R)} \epsilon(w) d(w) = (-1)^{q(R)} |W(R)|.$$
\end{prop}
\begin{proof}
The two formulae are equivalent by \cite{GoKoMPh}[Theorem 3.2] so let us prove the first one by induction on the rank of $R$.
The case of $R = \emptyset$ is trivial.
Assume that $R$ is not empty and that the formula holds in lower rank.
Denote $W= W(R)$.
For $\alpha \in R$ let $\mathcal{C}_{\alpha} = \{x \in Wx_0\ |\ \alpha \in \Delta_x \}$ and $\mathcal{D}_{\alpha}$ the orthogonal projection of $\mathcal{C}_{\alpha}$ on $Y = \ker \alpha$.
Geometrically, $\mathcal{C}_{\alpha}$ represents the chambers adjacent to the wall $Y$ on the side determined by $\alpha$.
For $x \in \mathcal{C}_{\alpha}$, by a computation similar to \ref{betaypos}, orthogonal projection on $Y$ maps the chamber containing $x$ onto a connected component of $Y \smallsetminus \bigcup_{\beta \in R \small\setminus \{ \pm \alpha \}} \ker \beta$, i.e. a chamber in $Y$ \emph{relative to $R$}.
Thus the projection $\mathcal{C}_{\alpha} \rightarrow \mathcal{D}_{\alpha}$ is bijective and in any $R_{\alpha}$-chamber of $Y$ there is the same number $|\mathcal{D}_{\alpha}|/|W(R_{\alpha})|$ of elements in $\mathcal{D}_{\alpha}$.
\begin{align*}
\mathrm{rank}(R) \sum_{w \in W} d(w) &= \sum_{x \in Wx_0}\ \sum_{\alpha \in \Delta_x} \bar{c}_R(x, \lambda_0) \\
& = \frac{1}{2}\sum_{\alpha \in R} \ \sum_{x \in \mathcal{C}_{\alpha}} \bar{c}_R(x, \lambda_0) + \bar{c}_R(s_{\alpha}(x), \lambda_0) \\
& = \sum_{\alpha \in R}\ \sum_{y \in \mathcal{D}_{\alpha}} \bar{c}_{R_{\alpha}}(y, \lambda_0|_Y) \\
& = \sum_{\alpha \in R}\ |\mathcal{D}_{\alpha}| = \sum_{x \in Wx_0}\ |\Delta_x| = \mathrm{rank}(R) |W|.
\end{align*}
At the second line we used the permutation $\alpha \mapsto - \alpha$ of $R$ and the fact that $x \in \mathcal{C}_{\alpha} \Leftrightarrow s_{\alpha}(x) \in \mathcal{C}_{-\alpha}$.
\end{proof}

\subsubsection{Character of averaged discrete series on non-compact tori}
\label{SectionCharDS}

In this section we consider a reductive group $\mathbf{G}$ over $\R$ which has discrete series.
To simplify notations we assume that $\mathbf{G}$ is semisimple, as it is the case for the symplectic and special orthogonal groups.
Fix a dominant weight $\lambda$ for $\mathbf{G}_{\C}$, and let $\tau = \lambda + \rho$ where $2\rho$ is the sum of the positive roots.
Let $\mathbf{M}$ be a Levi subgroup of $\mathbf{G}$ and denote by $\mathbf{A}_{\mathbf{M}}$ the biggest split central torus in $\mathbf{M}$.
If $\gamma \in \mathbf{M}(\R)$ is semisimple, $\mathbf{G}$-regular and belongs to a maximal torus anisotropic modulo $\mathbf{A}_{\mathbf{M}}$, define
$$ \Phi_{\mathbf{M}}(\gamma, \tau) := (-1)^{q\left( \mathbf{G}(\R) \right)} \left|D_{\mathbf{M}}^{\mathbf{G}}(\gamma)\right|^{1/2} \sum_{\pi_{\infty} \in \Pi_{\mathrm{disc}}(\tau)} \Theta_{\pi_{\infty}}(\gamma)$$
where $D_{\mathbf{M}}^{\mathbf{G}}(\gamma) = \det \left( \mathrm{Id} - \mathrm{Ad}(\gamma)\,|\,\mathfrak{g} / \mathfrak{m} \right)$.
Note that for $\gamma \in \mathbf{G}(\R)$ semisimple elliptic regular, $\Phi_{\mathbf{G}}(\gamma, \tau) \mu_{\mathrm{EP}, \mathbf{I}(\R)} = \mathrm{Tr} \left( \gamma | V_{\lambda} \right) \mu_{\mathrm{EP}, \mathbf{I}(\R)} = O_{\gamma}\left( f_{\lambda}(g)dg \right)$ where $f_{\lambda}(g)dg$ is the smooth compactly supported distribution of section \ref{EPinftySection}.

When $\mathbf{M} \times_{\Q} \R$ admits a maximal torus $\mathbf{T}$ anisotropic modulo $\mathbf{A}_{\mathbf{M}} \times_{\Q} \R$, Arthur shows that $\Phi_{\mathbf{M}}(\cdot, \tau)$ extends continuously to $\mathbf{T}(\R)$ (beware that the statement \cite{ArthurL2}[(4.7)] is erroneous: in general $\Phi_{\mathbf{M}}(\gamma, \tau)$ is not identically zero outside the connected components that intersect the center of $\mathbf{G}$).
Following \cite{GoKoMPh}[§4], to which we refer for details, let us write the restriction of $\Phi_{\mathbf{M}}(\cdot, \tau)$ to any connected component of $\mathbf{T}(\R)_{\mathbf{G}-\mathrm{reg}}$ as a linear combination of traces in algebraic representations of $\mathbf{M}$.

Let $R$ be the set of roots of $\mathbf{T}$ on $\mathbf{G}$ (over $\C$).
Let $R_{\mathbf{M}}$ be the set of roots of $\mathbf{T}$ on $\mathbf{M}$.
Let $\gamma \in \mathbf{T}(\R)$ be $\mathbf{G}$-regular, and let $\Gamma$ be the connected component of $\gamma$ in $\mathbf{T}(\R)$.
Let $R_{\Gamma}$ be the set of real roots $\alpha \in R$ such that $\alpha(\gamma)>0$.
As the notation suggests, it only depends on $\Gamma$.
Moreover $R_{\Gamma}$ and $R_{\mathbf{M}}$ are orthogonal sub-root systems of $R$: the coroots of $R_{\mathbf{M}}$ factor through $\mathbf{T} \cap \mathbf{M}_{\mathrm{der}}$ which is anisotropic, while the roots in $R_{\Gamma}$ factor through the biggest split quotient of $\mathbf{T}$.
Finally $\Phi_{\mathbf{M}}(\gamma, \tau)=0$ unless $\gamma$ belongs to the image of $\mathbf{G}_{\mathrm{sc}}(\R)$, and in that case the Weyl group $W(R_{\Gamma})$ of $R_{\Gamma}$ contains $-\mathrm{Id}$ and $\mathrm{rk}(R_{\Gamma}) = \dim \mathbf{A}_{\mathbf{M}}$.
In the following we assume that $\gamma \in \mathrm{Im}(\mathbf{G}_{\mathrm{sc}}(\R) \rightarrow \mathbf{G}(\R))$.

Since $\gamma$ is $\mathbf{G}$-regular, it defines a set of positive roots $R^+_{\gamma} = \left\{ \alpha \in R_{\gamma}\ |\ \alpha(\gamma)>1\right\}$ in $R_{\Gamma}$.
Choose a parabolic subgroup $\mathbf{P} = \mathbf{M} \mathbf{N}$ with unipotent radical $\mathbf{N}$ such that $R_{\gamma}^+$ is included in the set of roots of $\mathbf{T}$ on $\mathbf{N}$.
In general this choice is not unique.
Choose any set of positive roots $R^+_{\mathbf{M}}$ for $R_{\mathbf{M}}$.
There is a unique Borel subgroup $\mathbf{B} \subset \mathbf{P}$ of $\mathbf{G}$ containing $\mathbf{T}$ such that the set of roots of $\mathbf{T}$ on $\mathbf{B} \cap \mathbf{M}$ is $R_{\mathbf{M}}^+$.
Let $R^+$ be the set of positive roots in $R$ corresponding to $\mathbf{B}$.

There is a unique $x_{\gamma} \in (\R R_{\Gamma})^* = \R \otimes_{\Z} X_*(\mathbf{A}_{\mathbf{M}})$ such that for any $\alpha \in R_{\Gamma}$, $\alpha(x_{\gamma}) = \alpha(\gamma)$.
Then $x_{\gamma}$ is $R_{\Gamma}$-regular and the chamber in which $x_{\gamma}$ lies only depends on the connected component of $\gamma$ in $\mathbf{T}(\R)_{\mathbf{G}-\mathrm{reg}}$.
Denote by $\mathrm{pr}$ the orthogonal projection $\R \otimes_{\Z} X^*(\mathbf{T}) \rightarrow \R R_{\Gamma}$.
When we identify $\R R_{\Gamma}$ with $\R \otimes_{\Z} X^*(\mathbf{A}_{\mathbf{M}})$, $\mathrm{pr}$ is simply ``restriction to $\mathbf{A}_{\mathbf{M}}$''.
By \cite{GoKoMPh}[proof of Lemma 4.1 and end of §4] we have
$$ \Phi_{\mathbf{M}}(\gamma, \tau) =  \frac{\delta_{\mathbf{P}}(\gamma)^{1/2}}{\prod_{\alpha \in R^+_{\mathbf{M}}} \left( 1 - \alpha(\gamma)^{-1}\right)} \sum_{w \in W(R)} \epsilon(w) \bar{c}_{R_{\Gamma}}(x_{\gamma}, \mathrm{pr}(w(\tau_{\mathbf{B}}))) \left[ w(\tau_{\mathbf{B}})-\rho_{\mathbf{B}}\right](\gamma)$$
where
$$ \delta_{\mathbf{P}}(\gamma) = \left| \det\left( \gamma\,|\,\mathrm{Lie}(\mathbf{N}) \right) \right| = \prod_{\alpha \in R^+ - R^+_{\mathbf{M}}} \left|\alpha(\gamma) \right|. $$
Since $\rho_{\mathbf{B}}-\rho_{\mathbf{B} \cap \mathbf{M}}$ is invariant under $W(R_{\mathbf{M}})$, in the above sum we can combine terms in the same orbit under $W(R_{\mathbf{M}})$ to identify Weyl's character formula for algebraic representations of $\mathbf{M}$.
Let $E = \left\{ w \in W(R)\ |\ \forall \alpha \in R^+_{\gamma} \cup R^+_{\mathbf{M}},\ w^{-1}(\alpha) \in R^+ \right\}$, a set of representatives for the action of $W(R_{\Gamma}) \times W(R_{\mathbf{M}})$ on the left of $W(R)$.
Denoting $V_{\mathbf{M},\lambda'}$ the algebraic representation of $\mathbf{M}$ with highest weight $\lambda'$, we obtain
$$ \Phi_{\mathbf{M}}(\gamma, \tau) = \delta_{\mathbf{P}}(\gamma)^{1/2} \sum_{w_0 \in E} \sum_{w_1 \in W\left(R_{\Gamma} \right)} \epsilon(w_1w_0) d(w_1) \mathrm{Tr}\left(\gamma | V_{\mathbf{M},   w_1w_0(\tau_{\mathbf{B}})-\rho_{\mathbf{B}}} \right) $$

Furthermore $w_1w_0(\tau_{\mathbf{B}})-w_0(\tau_{\mathbf{B}}) \in \Z R_{\Gamma}$ is invariant under $W(R_{\mathbf{M}})$, hence in the above sum
$$ \mathrm{Tr}\left(\gamma | V_{\mathbf{M},   w_1w_0(\tau_{\mathbf{B}})-\rho_{\mathbf{B}}} \right) = \left[ w_1w_0(\tau_{\mathbf{B}})-w_0(\tau_{\mathbf{B}})\right](\gamma) \times \mathrm{Tr}\left(\gamma | V_{\mathbf{M},   w_0(\tau_{\mathbf{B}})-\rho_{\mathbf{B}}} \right) $$
and $\left[ w_1w_0(\tau_{\mathbf{B}})-w_0(\tau_{\mathbf{B}})\right](\gamma)$ is a positive real number, which does not really depend on $\gamma$ but only on the coset $(\mathbf{T} \cap \mathbf{M}_{\mathrm{der}})(\R) \gamma$ (equivalently, on $x_{\gamma}$).
Finally we obtain
\begin{align*}
\Phi_{\mathbf{M}}(\gamma, \tau) & = \delta_{\mathbf{P}}(\gamma)^{1/2} \sum_{w_0 \in E} \epsilon(w_0) \Biggl[ \sum_{w_1 \in W\left(R_{\Gamma} \right)} \epsilon(w_1) d(w_1) \left[ w_1w_0(\tau_{\mathbf{B}})-w_0(\tau_{\mathbf{B}})\right](\gamma) \\
& \qquad \times \mathrm{Tr}\left(\gamma | V_{\mathbf{M},   w_0(\tau_{\mathbf{B}})-\rho_{\mathbf{B}}} \right) \Biggr].
\end{align*}
This formula is valid for $\gamma$ in the closure (in $\mathbf{T}(\R)$) of a connected component of $\mathbf{T}(\R)_{\mathbf{G}-\mathrm{reg}}$.
\begin{prop} \label{charavDScpct}
If $\gamma$ is compact, i.e.\ the smallest closed subgroup of $\mathbf{G}(\R)$ containing $\gamma$ is compact, then we have
$$ \Phi_{\mathbf{M}}(\gamma, \tau)  = (-1)^{q(R_{\Gamma})}|W(R_{\Gamma})| \sum_{w_0 \in E} \epsilon(w_0) \mathrm{Tr}\left(\gamma | V_{\mathbf{M},   w_0(\tau_{\mathbf{B}})-\rho_{\mathbf{B}}} \right) .$$
\end{prop}
\begin{proof}
This formula follows from $\left[ w_1w_0(\tau_{\mathbf{B}})-w_0(\tau_{\mathbf{B}})\right](\gamma) = 1$ and Proposition \ref{propaltsumcbar}.
\end{proof}

\subsubsection{Explicit formulae for the parabolic terms}
\label{explicitpara}

Let $\mathbf{G}$ be one of $\mathbf{SO}_{2n+1}$ or $\mathbf{Sp}_{2n}$ or $\mathbf{SO}_{4n}$ over $\Z$, let $\prod_p f_p$ be the characteristic function of $\mathbf{G}(\widehat{\Z})$ and $\prod_p dg_p$ the Haar measure on $\mathbf{G}(\A_f)$ such that $\mathbf{G}(\widehat{\Z})$ has measure one.
Let $\lambda$ be a dominant weight for $\mathbf{G}_{\C}$ and let $f_{\infty, \lambda}(g_{\infty}) dg_{\infty}$ be the distribution on $\mathbf{G}(\R)$ defined in section \ref{EPinftySection}.
Denote $f(g)dg = f_{\infty, \lambda}(g_{\infty}) dg_{\infty} \prod_p f_p(g_p) dg_p$.
Using Proposition \ref{charavDScpct} and tedious computations, we obtain explicit formulae for the geometric side $T_{\mathrm{geom}}(f(g)dg)$ of Arthur's trace formula defined in section \ref{ParaTermsSection}.
For a dominant weight $\lambda = k_1 e_1 + \dots + k_n e_n$ it will be convenient to write $T_{\mathrm{geom}}(\mathbf{G}, \underline{k})$ for $T_{\mathrm{geom}}(f(g)dg)$ to precise the group $\mathbf{G}$, and similarly for $T_{\mathrm{ell}}$.
If $\mathbf{G}$ is trivial ($\mathbf{SO_0}$ or $\mathbf{SO}_1$ or $\mathbf{Sp}_0$) then $T_{\mathrm{ell}}$ is of course simply equal to $1$.

Any Levi subgroup $\mathbf{M}$ of $\mathbf{G}$ is isomorphic to $\prod_i \mathbf{GL}_{n_i} \times \mathbf{G}'$ where $\mathbf{G}'$ is of the same type as $\mathbf{G}$.
Note that $\mathbf{M}(\R)$ has essentially discrete series (i.e.\ $\Phi_{\mathbf{M}}(\cdot, \cdot)$ is not identically zero) if and only if for all $i, n_i \leq 2$ and in case $\mathbf{G}$ is even orthogonal, $\mathbf{G}'$ has even rank.
Thus the Levi subgroups $M$ whose contribution to $T_{\mathrm{geom}}$ (that is formula \ref{formulaTgeom}) is nonzero are isomorphic to $\mathbf{GL}_1^a \times \mathbf{GL}_2^c \times \mathbf{G}'$ for some integers $a,c$.

Since $\mathbf{PGL}_2 \simeq \mathbf{SO}_3$, for $k \in \Z_{\geq 0}$ we denote $T_{\mathrm{ell}}(\mathbf{PGL}_2, k) = T_{\mathrm{ell}}(\mathbf{SO}_3, k)$.
For non-negative $k \in 1/2 \Z \smallsetminus \Z$ it is convenient to define $T_{\mathrm{ell}}(\mathbf{PGL}_2, k) = 0$, so that for any $k \in \Z_{\geq 0}$ we have $T_{\mathrm{ell}}(\mathbf{PGL}_2, k/2) = T_{\mathrm{ell}}(\mathbf{Sp}_2, k)/2$.

For $a,c,d \in \Z_{\geq 0}$, let $\Xi_{a,c,d}$ be the set of $\sigma$ in the symmetric group $S_{a+2c+d}$ such that
\begin{itemize}
\item $\sigma(1) < \dots < \sigma(a)$,
\item $\sigma(a+1)<\sigma(a+3)<\dots < \sigma(a+2c-1)$,
\item for any $1 \leq i \leq c$, $\sigma(a+2i-1) < \sigma(a+2i)$,
\item $\sigma(a+2c+1) < \dots < \sigma(n)$.
\end{itemize}

For $a \geq 0$ and $x \in \{0, \dots, a\}$, define
$$\eta^{(B)}(a,x) = \frac{(-1)^{a(a-1)/2}}{2^a} \sum_{b=0}^{\lfloor a/2 \rfloor} (-1)^b \sum_{r=0}^{2b} \binom{x}{r} \binom{a-x}{2b-r} (-1)^r.$$
It is easy to check that
$$\eta^{(B)}(a,x) = \frac{(-1)^{a(a-1)/2}}{2^{a+1}} \mathrm{Tr}_{\Q(\sqrt{-1})/\Q} \left( (1+\sqrt{-1})^{a-x}(1-\sqrt{-1})^x \right) \in \frac{1}{2^{\lfloor (a+1)/2 \rfloor}} \Z.$$
For $n \geq a$, $\sigma \in S_n$ and $\underline{k}=(k_1, \dots, k_n) \in \Z^n$, let
$$ \eta^{(B)}(a, \underline{k}, \sigma) = \eta^{(B)}\left(a, \mathrm{card}\{ i \in \{1, \dots, a\}\ |\ k_{\sigma(i)}+\sigma(i)+i = 1 \pmod{2} \}\right). $$
\begin{theo}[Parabolic terms for $\mathbf{G} = \mathbf{SO}_{2n+1}$]
Let $a,c,d \in \Z_{\geq 0}$ not all zero and $n = a+2c+d$.
The sum of the contributions to $T_{\mathrm{geom}}(\mathbf{SO}_{2n+1}, \underline{k})$ in formula \ref{formulaTgeom} of the Levi subgroups $\mathbf{M}$ in the orbit of $\mathbf{GL}_1^a \times \mathbf{GL}_2^c \times \mathbf{SO}_{2d+1}$ under the Weyl group $W(\mathbf{T}_0, \mathbf{G})$  is
$$\begin{array}{l}
\displaystyle\sum\limits_{\sigma \in \Xi_{a,c,d}} \eta^{(B)}(a,\underline{k},\sigma) \\
\quad \times \prod\limits_{i=1}^c \biggl[ T_{\mathrm{ell}}\left(\mathbf{PGL}_2,(k_{\sigma(a+2i-1)}-k_{\sigma(a+2i)}+\sigma(a+2i)-\sigma(a+2i-1)-1)/2\right) \\
\qquad - T_{\mathrm{ell}}(\mathbf{PGL}_2,(k_{\sigma(a+2i-1)}+k_{\sigma(a+2i)}-\sigma(a+2i)-\sigma(a+2i-1)+2n)/2) \Bigr] \\
\quad \times T_{\mathrm{ell}}(\mathbf{SO}_{2d+1}, (k_{\sigma(n-d+1)}+n-d+1-\sigma(n-d+1), \dots, k_{\sigma(n)}+n-\sigma(n))).
\end{array}$$
\end{theo}

We have a similar formula for the symplectic group.
For $a \geq 0$ and $x \in \{0, \dots, a\}$, define
$$\eta^{(C)}(a,x) = \frac{(-1)^{a(a-1)/2}}{2^a} \sum_{b=0}^a (-1)^{b(a-b)} \sum_{r=0}^b \binom{x}{r} \binom{a-x}{b-r} (-1)^r.$$
Then we have
$$ \eta^{(C)}(a,x) = \begin{cases} (-1)^{a/2} & \text{if } a \text{ is even and } x=a, \\ (-1)^{(a-1)/2} & \text{if } a \text{ is odd and } x=0, \\ 0 & \text{otherwise.} \end{cases} $$
For $n \geq a$, $\sigma \in S_n$ and $\underline{k}=(k_1, \dots, k_n) \in \Z^n$, let
$$ \eta^{(C)}(a, \underline{k}, \sigma) = \eta^{(C)}\left(a, \mathrm{card}\{ i \in \{1, \dots, a\}\ |\ k_{\sigma(i)}+\sigma(i)+i = 1 \pmod{2} \}\right). $$

\begin{theo}[Parabolic terms for $\mathbf{G} = \mathbf{Sp}_{2n}$]
Let $a,c,d \in \Z_{\geq 0}$ not all zero and $n = a+2c+d$.
The sum of the contributions to $T_{\mathrm{geom}}(\mathbf{Sp}_{2n}, \underline{k})$ in formula \ref{formulaTgeom} of the Levi subgroups $\mathbf{M}$ in the orbit of $\mathbf{GL}_1^a \times \mathbf{GL}_2^c \times \mathbf{Sp}_{2d}$ under the Weyl group $W(\mathbf{T}_0, \mathbf{G})$  is
$$\begin{array}{l}
\displaystyle\sum\limits_{\sigma \in \Xi_{a,c,d}} \eta^{(C)}(a, \underline{k},\sigma) \\
\quad \times \prod\limits_{i=1}^c \biggl[ T_{\mathrm{ell}}\left(\mathbf{PGL}_2,(k_{\sigma(a+2i-1)}-k_{\sigma(a+2i)}+\sigma(a+2i)-\sigma(a+2i-1)-1)/2\right) \\
\qquad - T_{\mathrm{ell}}(\mathbf{PGL}_2,(k_{\sigma(a+2i-1)}+k_{\sigma(a+2i)}-\sigma(a+2i)-\sigma(a+2i-1)+2n+1)/2) \Bigr] \\
\quad \times T_{\mathrm{ell}}(\mathbf{Sp}_{2d}, (k_{\sigma(n-d+1)}+n-d+1-\sigma(n-d+1), \dots, k_{\sigma(n)}+n-\sigma(n)))
\end{array}$$
\end{theo}

For $a \geq 0$ and $x \in \{0, \dots, 2a\}$, define
$$\eta^{(D)}(a,x) = \frac{1}{2^{2a}} \sum_{b=0}^a \sum_{r=0}^{2b} \binom{x}{r} \binom{2a-x}{2b-r} (-1)^r.$$
We have
$$ \eta^{(D)}(a,x) = \begin{cases} 1 & \text{if } a=0, \\ 1/2 & \text{if } a>0 \text{ and } x(2a-x)=0, \\ 0 & \text{otherwise.} \end{cases} $$
For $n \geq a$, $\sigma \in S_{2n}$ and $\underline{k}=(k_1, \dots, k_{2n}) \in \Z^{2n}$, let
$$ \eta^{(D)}(a, \underline{k}, \sigma) = \eta^{(D)}\left(a, \mathrm{card}\{ i \in \{1, \dots, 2a\}\ |\ k_{\sigma(i)}+\sigma(i)+i = 1 \pmod{2} \}\right). $$

For the group $\mathbf{SO}_{4n}$, we need only consider dominant weights $\underline{k}$ with $k_{2n} \geq 0$ (i.e.\ the same inequalities as for the other two infinite families) since the end result is invariant under the outer automorphism of $\mathbf{SO}_{4n}$, that is $T_{\mathrm{geom}}(\mathbf{SO}_{4n}, (k_1, \dots, k_{2n-1}, -k_{2n})) = T_{\mathrm{geom}}(\mathbf{SO}_{4n}, (k_1, \dots, k_{2n-1}, k_{2n}))$.
\begin{theo}[Parabolic terms for $\mathbf{G} = \mathbf{SO}_{4n}$]
Let $a,c,d \in \Z_{\geq 0}$ not all zero and $n=a+c+d$.
The sum of the contributions to $T_{\mathrm{geom}}(\mathbf{SO}_{4n}, \underline{k})$ in formula \ref{formulaTgeom} of the Levi subgroups $\mathbf{M}$ in the orbit of $\mathbf{GL}_1^{2a} \times \mathbf{GL}_2^c \times \mathbf{SO}_{4d}$ under the Weyl group $W(\mathbf{T}_0, \mathbf{G})$  is
$$\begin{array}{l}
\displaystyle\sum\limits_{\sigma \in \Xi_{2a,c,2d}} \eta^{(D)}(a, \underline{k}, \sigma) \\
\quad \times \prod\limits_{i=1}^c \biggl[ T_{\mathrm{ell}}\left(\mathbf{PGL}_2,(k_{\sigma(2a+2i-1)}-k_{\sigma(2a+2i)}+\sigma(2a+2i)-\sigma(2a+2i-1)-1)/2\right) \\
\qquad + T_{\mathrm{ell}}(\mathbf{PGL}_2,(k_{\sigma(2a+2i-1)}+k_{\sigma(2a+2i)}-\sigma(2a+2i)-\sigma(2a+2i-1)+4n-1)/2) \Bigr] \\
\quad \times T_{\mathrm{ell}}(\mathbf{SO}_{4d}, (k_{\sigma(2n-2d+1)}+2n-2d+1-\sigma(2n-2d+1), \dots, k_{\sigma(2n)}+2n-\sigma(2n))).
\end{array}$$
\end{theo}

\section{Endoscopic decomposition of the spectral side}

\subsection{The spectral side of the trace formula}

The previous sections give an algorithm to compute the geometric side of Arthur's trace formula in \cite{ArthurL2}.
Let us recall the spectral side of this version of the trace formula.
As before $\mathbf{G}$ denotes one of the reductive groups $\mathbf{SO}_{2n+1}$, $\mathbf{Sp}_{2n}$ or $\mathbf{SO}_{4n}$ over $\Z$.
Let $K_{\infty}$ be a maximal compact subgroup of $\mathbf{G}(\R)$ and denote $\mathfrak{g} = \C \otimes_{\R} \mathrm{Lie}(\mathbf{G}(\R))$.
Let $\mathcal{A}_{\mathrm{disc}}(\mathbf{G}(\Q) \backslash \mathbf{G}(\A))$ be the space of $K_{\infty} \times \mathbf{G}(\widehat{\Z})$-finite and $Z(U(\mathfrak{g}))$-finite functions in the discrete spectrum $L^2_{\mathrm{disc}}(\mathbf{G}(\Q) \backslash \mathbf{G}(\A))$.
It is also the space of automorphic forms in the sense of \cite{BoJaCorv} which are square-integrable.
There is an orthogonal decomposition
$$ \mathcal{A}_{\mathrm{disc}}(\mathbf{G}(\Q) \backslash \mathbf{G}(\A)) = \bigoplus_{\pi \in \Pi_{\mathrm{disc}}(\mathbf{G})} m_{\pi} \pi $$
where $\Pi_{\mathrm{disc}}(\mathbf{G})$ is a countable set of distinct isomorphism classes of unitary $(\mathfrak{g}, K_{\infty}) \times \mathbf{G}(\A_f)$-modules and $m_{\pi} \in \Z_{\geq 1}$.
Denote $\Pi_{\mathrm{disc}}^{\mathrm{unr}}(\mathbf{G}) \subset \Pi_{\mathrm{disc}}(\mathbf{G})$ the set of $\pi$ such that for any prime number $p$ the representation $\pi_p$ is unramified, i.e.\ $\pi_p^{\mathbf{G}(\Zp)} \neq 0$.

Let $\lambda$ be a dominant weight for $\mathbf{G}_{\C}$, and denote $V_{\lambda}$ the corresponding algebraic representation of $\mathbf{G}(\C)$, which by restriction to $\mathbf{G}(\R)$ we see as a $(\mathfrak{g}, K_{\infty})$-module.
If $X$ is an admissible $(\mathfrak{g}, K_{\infty})$-module, define its Euler-Poincaré characteristic with respect to $V_{\lambda}$
$$ \mathrm{EP}(X \otimes V_{\lambda}^*) = \sum_{i \geq 0} (-1)^i \dim H^i((\mathfrak{g}, K_{\infty}), X \otimes V_{\lambda}^*). $$
We refer to \cite{BoWa} for definitions and essential properties of $(\mathfrak{g}, K_{\infty})$-cohomology.
By \cite{BoWa}[Chapter I, Corollary 4.2] for any irreducible $(\mathfrak{g}, K_{\infty})$-module $X$, we have that $H^{\bullet}((\mathfrak{g}, K_{\infty}), X \otimes V_{\lambda}^*)=0$ unless $X$ has the same infinitesimal character as $V_{\lambda}$.

For our particular choice of function on $\mathbf{G}(\A_f)$ the spectral side of Arthur's trace formula in \cite{ArthurL2} is
\begin{equation} \label{spectralside} \sum_{\pi \in \Pi_{\mathrm{disc}}^{\mathrm{unr}}(\mathbf{G})} m_{\pi} \mathrm{EP}(\pi_{\infty} \otimes V_{\lambda}^*). \end{equation}
By \cite{HCAuto}[Theorem 1] there is only a finite number of nonzero terms.
Vogan and Zuckerman \cite{VoZu} (see also \cite{BoWa}[Chapter VI, §5]) have classified the irreducible unitary $(\mathfrak{g}, K_{\infty})$-modules having cohomology with respect to $V_{\lambda}$, and computed this cohomology.
However, the integer \ref{spectralside} alone is not enough to recover the number $m(X)$ of $\pi \in \Pi_{\mathrm{disc}}^{\mathrm{unr}}(\mathbf{G})$ such that $\pi_{\infty}$ is isomorphic to a given irreducible unitary $(\mathfrak{g}, K_{\infty})$-module $X$ having the same infinitesimal character as $V_{\lambda}$.

Arthur's endoscopic classification of the discrete automorphic spectrum of $\mathbf{G}$ \cite{Arthur} allows to express $m(X)$ using numbers of certain \emph{self-dual} cuspidal automorphic representations of general linear groups.
Conversely these numbers can be obtained from the Euler-Poincaré characteristic \ref{spectralside} for various groups $\mathbf{G}$ and weights $\lambda$.
For explicit computations we will have to make Assumption \ref{assumAJweak} that relates the rather abstract Arthur packets at the real place with the ones previously defined by Adams and Johnson in \cite{AdJo}.

Note that it will not be necessary to use \cite{VoZu} since the Euler-Poincaré characteristic is a much simpler invariant than the whole cohomology.

\subsubsection{Arthur's endoscopic classification}

Let us review  how Arthur's very general results in \cite{Arthur} specialise in our particular situation: level one and regular infinitesimal character.
We are brief since this was done in \cite{ChRe}[§3], though with a slightly different formulation.
We refer to \cite{BorelCorvallis} for the definition of L-groups.
For $\mathbf{G}$ a reductive group over $F$ we will denote $\widehat{\mathbf{G}}$ the connected component of the neutral element in ${}^L \mathbf{G}$ (which Borel denotes ${}^L \mathbf{G}^0$).

Let $F$ be a local field of characteristic zero.
The Weil-Deligne group of $F$ is denoted by $W_F'$: if $F$ is archimedean $W_F' = W_F$, whereas in the $p$-adic case $W_F' = W_F \times \mathrm{SU}(2)$.
Consider a quasisplit special orthogonal or symplectic group $\mathbf{G}$ over $F$.
Let $\psi : W_F' \times \mathrm{SL}_2(\C) \rightarrow {}^L \mathbf{G}$ be a local Arthur parameter, i.e.\ $\psi|_{W_F'}$ is a continuous semisimple splitting of ${}^L \mathbf{G} \rightarrow W_F'$ with bounded image and $\psi|_{\mathrm{SL}_2(\C)}$ is algebraic.
If $\psi|_{\mathrm{SL}_2(\C)}$ is trivial then $\psi$ is a tempered Langlands parameter.
The general case is considered for global purposes, which we will discuss later.
Consider the group $C_{\psi} = \mathrm{Cent}(\psi, \widehat{\mathbf{G}})$ and the finite group
$$S_{\psi} = C_{\psi} / C_{\psi}^0 Z(\widehat{\mathbf{G}})^{\mathrm{Gal}(\overline{F}/F)}.$$
For the groups $\mathbf{G}$ considered here the group $S_{\psi}$ is isomorphic to a product of copies of $\{ \pm 1 \}$.
Arthur \cite{Arthur}[Theorem 1.5.1] associates with $\psi$ a finite multiset $\Pi_{\psi}$ of irreducible unitary representations of $\mathbf{G}(F)$, along with a character $\langle \cdot , \pi \rangle$ of $S_{\psi}$ for any $\pi \in \Pi_{\psi}$.
In the even orthogonal case this is not exactly true: instead of actual representations, $\Pi_{\psi}$ is comprised of orbits of the group $\mathrm{Out}(\mathbf{G}) \simeq \Z/2\Z$ of outer automorphisms of $\mathbf{G}$ on the set of isomorphism classes of irreducible representations of $\mathbf{G}(F)$.
These orbits can be described as modules over the $\mathrm{Out}(\mathbf{G})$-invariants of the Hecke algebra $\mathcal{H}(\mathbf{G}(F))$ of $\mathbf{G}(F)$, which we denote $\mathcal{H}'(\mathbf{G}(F))$.
Here we have fixed a splitting $\mathrm{Out}(\mathbf{G}) \rightarrow \mathrm{Aut}(\mathbf{G})$ defined over $F$.
Note that if $F$ is $p$-adic, $\mathbf{G}$ is unramified and $K$ is a hyperspecial subgroup of $\mathbf{G}(F)$ we can choose a splitting $\mathrm{Out}(\mathbf{G}) \rightarrow \mathrm{Aut}(\mathbf{G})$ that preserves K.
If $F$ is archimedean and $K$ is a maximal compact subgroup of $\mathbf{G}(F)$, we can also choose a splitting that preserves $K$, and $\mathcal{H}'(\mathbf{G}(F))$ is the algebra of left and right $K$-finite $\mathrm{Out}(\mathbf{G})$-invariant distributions on $\mathbf{G}(F)$ with support in K.
Note that the choice of splitting does not matter when one considers invariant objects, such as orbital integrals or traces in representations.

Denote $\mathrm{Std} : {}^L \mathbf{G} \rightarrow \mathrm{GL}_N(\C)$ the standard representation, where
$$ N = \begin{cases}
2n & \text{if }\mathbf{G}_{\bar{F}} \simeq \left(\mathbf{SO}_{2n+1}\right)_{\bar{F}}\text{, i.e.\ } \widehat{\mathbf{G}} \simeq \mathrm{Sp}_{2n}(\C), \\
2n+1 & \text{if }\mathbf{G}_{\bar{F}} \simeq \left(\mathbf{Sp}_{2n}\right)_{\bar{F}}\text{, i.e.\ } \widehat{\mathbf{G}} \simeq \mathrm{SO}_{2n+1}(\C), \\
2n & \text{if }\mathbf{G}_{\bar{F}} \simeq \left(\mathbf{SO}_{2n}\right)_{\bar{F}}\text{, i.e.\ } \widehat{\mathbf{G}} \simeq \mathrm{SO}_{2n}(\C).
\end{cases} $$
In the first two cases $\det \circ\, \mathrm{Std}$ is trivial, whereas in the third case it takes values in $\{ \pm 1 \}$ and factors through a character $\mathrm{Gal}(\overline{F} / F) \rightarrow \{ \pm 1 \}$, which by local class field theory we can also see as a character $\eta_{\mathbf{G}} : F^{\times} \rightarrow \{ \pm 1 \}$.
If $\widehat{\mathbf{G}} = \mathrm{Sp}_{2n}(\C)$ (resp.\ $\widehat{\mathbf{G}} = \mathrm{SO}_{2n+1}(\C)$), the standard representation $\mathrm{Std}$ induces a bijection from the set of conjugacy classes of Arthur parameters $\psi : W_F' \times \mathrm{SL}_2(\C) \rightarrow \widehat{\mathbf{G}}$ to the set of conjugacy classes of Arthur parameters $\psi' : W_F' \times \mathrm{SL}_2(\C) \rightarrow \mathrm{GL}_N(\C)$ such that $\det \circ\, \psi'$ is trivial and there exists a non-degenerate alternate (resp.\ symmetric) bilinear form on $\C^N$ preserved by $\mathrm{Im}(\psi')$.
The third case, where $\mathbf{G}$ is an even special orthogonal group, induces a small complication.
Composing with $\mathrm{Std}$ still induces a surjective map from the set of conjugacy classes of Arthur parameters $\psi : W_F' \times \mathrm{SL}_2(\C) \rightarrow {}^L \mathbf{G}$ to the set of conjugacy classes of Arthur parameters $\psi' : W_F' \times \mathrm{SL}_2(\C) \rightarrow \mathrm{GL}_N(\C)$ having determinant $\eta_{\mathbf{G}}$ and such that there exists a non-degenerate bilinear form on $\C^N$ preserved by $\mathrm{Im}(\psi')$.
However, the fibers of this map can have cardinality one or two, the latter case occurring if and only if all the self-dual irreducible constituents of $\psi'$ have even dimension.
The Arthur packet $\Pi_{\psi}$ along with the characters $\langle \cdot, \pi \rangle$ of $S_{\psi}$ are characterised \cite{Arthur}[Theorem 2.2.1] using the representation of $\mathbf{GL}_N(F)$ associated with $\mathrm{Std} \circ \psi$ by the local Langlands correspondence, and twisted and ordinary endoscopic character identities.
The characters $(\langle \cdot, \pi \rangle)_{\pi \in \Pi_{\psi}}$ of $S_{\psi}$ are well-defined only once we have fixed an equivalence class of Whittaker datum for $\mathbf{G}$, since this choice has to be made to normalise the transfer factors involved in the ordinary endoscopic character identities.

In the $p$-adic case, we will mainly be interested in \emph{unramified} Arthur parameters $\psi$, i.e.\ such that $\psi|_{W_F'}$ is trivial on the inertia subgroup and on $\mathrm{SU}(2)$.
Of course these exist only if $\mathbf{G}$ is unramified, so let us make this assumption.
We refer to \cite{CasSha} for the definition of unramified Whittaker data with respect to a choice of hyperspecial maximal compact subgroup.
Note that several conjugacy classes of Whittaker data can correspond to the same conjugacy class of hyperspecial subgroups, and that $\mathbf{G}_{\mathrm{ad}}(F)$ acts transitively on both sets of conjugacy classes.

The following lemma is implicit in \cite{Arthur}.
Note that a weak version of it is needed to make sense of the main global theorem \cite{Arthur}[Theorem 1.5.2].
\begin{lemm} \label{unrApacket}
Let $\psi : W_F' \times \mathrm{SL}_2(\C) \rightarrow {}^L \mathbf{G}$ be an Arthur parameter for the $p$-adic field $F$.
Then $\Pi_{\psi}$ contains an unramified representation if and only if $\psi$ is unramified.
In that case, $\Pi_{\psi}$ contains a unique unramified representation $\pi$, which satisfies $\langle \cdot, \pi \rangle = 1$.
\end{lemm}
\begin{proof}
This is a consequence of the proof of \cite{Arthur}[Lemma 7.3.4].
We borrow Arthur's notations for this (sketch of) proof.
Let $\widetilde{f}$ be the characteristic function of $\mathbf{GL}_N(\mathcal{O}_F) \rtimes \theta \subset \widetilde{\mathbf{GL}}_N(F)$.
Arthur shows that $\widetilde{f}_N(\psi) = 1$ if $\psi$ is unramified.
If $\psi$ is ramified, the representation of $\mathbf{GL}_N(F)$ associated with $\mathrm{Std} \circ\, \psi$ is ramified, thus $\widetilde{f}_N(\psi) = 0$.
The statement of the lemma follows easily from these two identities, the characterization \cite{Arthur}[Theorem 2.2.1] of the local Arthur packets by endoscopic character relations, and the twisted fundamental lemma (which applies even when the residual characteristic of $F$ is small!) proved in \cite{Arthur}[Lemma 7.3.4].
\end{proof}

To state Arthur's global theorem we only consider the split groups $\mathbf{SO}_{2n+1}$, $\mathbf{Sp}_{2n}$ and $\mathbf{SO}_{2n}$ over $\Q$.
From now on $\mathbf{G}$ denotes one of these groups.
By \cite{Arthur}[Theorem 1.4.1], any self-dual cuspidal automorphic representation $\pi$ of $\mathbf{GL}_M$ over a number field has a sign $s(\pi) \in \{ \pm 1 \}$, which intuitively is the type of the conjectural Langlands parameter of $\pi$: $s(\pi) = 1$ (resp.\ $-1$) if this parameter is orthogonal (resp.\ symplectic).
Unsurprisingly if $M$ is odd then $s(\pi)=1$, and if $M$ is even and $s(\pi)=-1$ then the central character $\chi_{\pi}$ of $\pi$ is trivial.
Moreover Arthur characterises $s(\pi)$ using $\mathrm{Sym}^2$ and $\bigwedge^2$ L-functions \cite{Arthur}[Theorem 1.5.3].
This partition of the set of self-dual cuspidal automorphic representations of general linear groups allows to define substitutes for discrete Arthur-Langlands parameters for the group $\mathbf{G}$.
Define $s(\mathbf{G}) = -1$ in the first case ($\widehat{\mathbf{G}} = \mathrm{Sp}_{2n}(\C)$) and $s(\mathbf{G}) = 1$ in the last two cases ($\widehat{\mathbf{G}} = \mathrm{SO}_{2n+1}(\C)$ or $\mathrm{SO}_{2n}(\C)$).
Define $\Psi(\mathbf{G})$ as the set of formal sums $\psi = \boxplus_{i \in I} \pi_i[d_i]$ where
\begin{enumerate}
\item for all $i \in I$, $\pi_i$ is a self-dual cuspidal automorphic representation of $\mathbf{GL}_{n_i}/\Q$,
\item for all $i \in I$, $d_i \in \Z_{\geq 1}$ is such that $s(\pi_i)(-1)^{d_i-1} = s(\mathbf{G})$,
\item $N = \sum_{i \in I} n_id_i$,
\item the pairs $(\pi_i, d_i)$ are distinct,
\item $\prod_{i \in I} \chi_{\pi_i}^{d_i} = 1$, where $\chi_{\pi_i}$ is the central character of $\pi_i$.
\end{enumerate}
The last condition is automatically satisfied if $\widehat{\mathbf{G}} = \mathrm{Sp}_{2n}(\C)$.
The notation $\pi_i[d_i]$ suggests taking the tensor product of the putative Langlands parameter of $\pi_i$ with the $d_i$-dimensional algebraic representation of $\mathrm{SL}_2(\C)$.
Each factor $\pi_i[d_i]$ corresponds to a discrete automorphic representation of $\mathbf{GL}_{n_id_i}$ over $\Q$ by \cite{MoeWal}.

Let $v$ denote a place of $\Q$.
Thanks to the local Langlands correspondence for general linear groups applied to the $(\pi_i)_v$'s, for $\psi \in \Psi(\mathbf{G})$, $\psi$ specialises into a local Arthur parameter $\psi_v : W'_{\Q_v} \times \mathrm{SL}_2(\C) \rightarrow \mathrm{GL}_N(\C)$.
By \cite{Arthur}[Theorem 1.4.2] we can see $\psi_v$ as a genuine local Arthur parameter $W'_{\Q_v} \times \mathrm{SL}_2(\C) \rightarrow {}^L \mathbf{G}$, but in the even orthogonal case $\psi_v$ is well-defined only up to outer automorphism.
To be honest it is not known in general that $\psi_v(W_{\Q_v}')$ is bounded (this would be the Ramanujan-Petersson conjecture), but we will not comment more on this technicality and refer to the discussion preceding \cite{Arthur}[Theorem 1.5.2] for details.
Thus we have a finite multiset $\Pi_{\psi_v}$ of irreducible unitary representations of $\mathbf{G}(\Q_v)$, each of these representations being well-defined only up to outer conjugacy in the even orthogonal case.

As in the local case we want to define $C_{\psi} = \mathrm{Cent}(\psi, \widehat{\mathbf{G}})$ and
$$S_{\psi} = C_{\psi} / C_{\psi}^0Z(\widehat{\mathbf{G}})^{\mathrm{Gal}(\Qbar / \Q)} = C_{\psi}/Z(\widehat{\mathbf{G}}). $$
Observe that this can be done formally for $\psi = \boxplus_{i \in I} \pi_i[d_i]$.
An element $s$ of $C_{\psi}$ is described by $J \subset I$ such that $\sum_{i \in J} n_id_i$ is even, and $s$ corresponds formally to $-\mathrm{Id}$ on the space of $\boxplus_{i \in J} \pi_i[d_i]$ and $\mathrm{Id}$ on the space of $\boxplus_{i \in I \smallsetminus J} \pi_i[d_i]$.
Thus one can define a finite $2$-group $S_{\psi}$ along with a natural morphism $S_{\psi} \rightarrow S_{\psi_v}$ for any place $v$ of $\Q$.
The last ingredient in Arthur's global theorem is the character $\epsilon_{\psi}$ of $S_{\psi}$.
It is defined in terms of the root numbers $\epsilon(\pi_i \times \pi_j, 1/2)$ just after \cite{Arthur}[Theorem 1.5.2].
If all the $d_i$'s are equal to $1$, in which case we say that $\psi$ is formally tempered, then $\epsilon_{\psi} = 1$.

Fix a global Whittaker datum for $\mathbf{G}$, inducing a family of Whittaker data for $\mathbf{G}_{\Q_v}$ where $v$ ranges over the places of $\Q$.
Our reductive group is defined over $\Z$, and the global Whittaker datum can be chosen so that for any prime number $p$ it induces an unramified Whittaker datum on $\mathbf{G}(\Qp)$ with respect to the hyperspecial subgroup $\mathbf{G}(\Zp)$.
Let $K_{\infty}$ be a maximal compact subgroup of $\mathbf{G}(\R)$, and denote $\mathfrak{g} = \C \otimes_{\R} \mathrm{Lie}(\mathbf{G}(\R))$.
The following is a specialization of the general theorem \cite{Arthur}[Theorem 1.5.2] to the ``everywhere unramified'' case, using Lemma \ref{unrApacket}.
\begin{theo} \label{theoArthurunr}
Recall that $\mathcal{A}_{\mathrm{disc}}(\mathbf{G}(\Q) \backslash \mathbf{G}(\A))$ is the space of $K_{\infty} \times \mathbf{G}(\widehat{\Z})$-finite and $Z(U(\mathfrak{g}))$-finite functions in the discrete spectrum $L^2_{\mathrm{disc}}(\mathbf{G}(\Q) \backslash \mathbf{G}(\A))$.
Let $\Psi(\mathbf{G})^{\mathrm{unr}}$ be the set of $\psi = \boxplus_i \pi_i[d_i] \in \Psi(\mathbf{G})$ such that for any $i$, $\pi_i$ is unramified at every prime.
There is a $\mathcal{H}'(\mathbf{G}(\R))$-equivariant isomorphism
$$ \mathcal{A}_{\mathrm{disc}}(\mathbf{G}(\Q) \backslash \mathbf{G}(\A))^{\mathbf{G}(\widehat{\Z})} \simeq \bigoplus_{\psi \in \Psi(\mathbf{G})^{\mathrm{unr}}}\ \bigoplus_{\substack{\pi_{\infty} \in \Pi_{\psi_{\infty}} \\ \langle \cdot, \pi_{\infty} \rangle = \epsilon_{\psi}}} m_{\psi} \pi_{\infty} $$
where $m_{\psi} = 1$ except if $\mathbf{G}$ is even orthogonal and for all $i$ $n_id_i$ is even, in which case $m_{\psi} = 2$.

For $\pi_{\infty} \in \Pi_{\psi_{\infty}}$ the character $\langle \cdot, \pi_{\infty} \rangle$ of $S_{\psi_{\infty}}$ induces a character of $S_{\psi}$ using the morphism $S_{\psi} \rightarrow S_{\psi_{\infty}}$, and the inner direct sum ranges over the $\pi_{\infty}$'s such that this character of $S_{\psi}$ is equal to $\epsilon_{\psi}$.
\end{theo}
In the even orthogonal case, $\pi_{\infty}$ is only an $\mathrm{Out}(\mathbf{G}_{\R})$-orbit of irreducible representations, and it does not seem possible to resolve this ambiguity at the moment.
Nevertheless it disappears in the global setting.
There is a splitting $\mathrm{Out}(\mathbf{G}) \rightarrow \mathrm{Aut}(\mathbf{G})$ such that $\mathrm{Out}(\mathbf{G})$ preserves $\mathbf{G}(\widehat{\Z})$, and thus if $\{X_1, X_2 \}$ is an $\mathrm{Out}(\mathbf{G}_{\R})$-orbit of isomorphism classes of irreducible unitary $(\mathfrak{g}, K_{\infty})$-modules, then $X_1$ and $X_2$ have the same multiplicity in $ \mathcal{A}_{\mathrm{disc}}(\mathbf{G}(\Q) \backslash \mathbf{G}(\A))^{\mathbf{G}(\widehat{\Z})}$.

\subsubsection{The spectral side from an endoscopic perspective}

We keep the notations from the previous section.
Suppose now that $\mathbf{G}(\R)$ has discrete series, i.e.\ $\mathbf{G}$ is not $\mathbf{SO}_{2n}$ with $n$ odd.
Let $\lambda$ be a dominant weight for $\mathbf{G}_{\C}$.
Using Theorem \ref{theoArthurunr} we can write the spectral side of the trace formula \ref{spectralside} as
\begin{equation} \label{spectralsideendo}
\sum_{\psi \in \Psi(\mathbf{G})^{\mathrm{unr}}}\ \sum_{\substack{\pi_{\infty} \in \Pi_{\psi_{\infty}} \\ \langle \cdot, \pi_{\infty} \rangle = \epsilon_{\psi}}} m_{\psi} \mathrm{EP}(\pi_{\infty} \otimes V_{\lambda}^*).
\end{equation}
We need to be cautious here since $\mathrm{EP}(\pi_{\infty} \otimes V_{\lambda}^*)$ is not well-defined in the even orthogonal case.
If $\pi_{\infty}$ is the restriction to $\mathcal{H}'(\mathbf{G}(\R))$ of two non-isomorphic $(\mathfrak{g}, K_{\infty})$-modules $\pi_{\infty}^{(1)}$ and $\pi_{\infty}^{(2)}$, we \emph{define}
$$ \mathrm{EP}(\pi_{\infty} \otimes V_{\lambda}^*) = \frac{1}{2} \mathrm{EP}\left((\pi_{\infty}^{(1)} \oplus \pi_{\infty}^{(2)}) \otimes V_{\lambda}^*\right). $$
In \ref{spectralsideendo} we can restrict the sum to $\pi_{\infty}$'s whose infinitesimal character equals that of $V_{\lambda}$ (up to outer automorphism in the even orthogonal case), which is $\lambda + \rho$ via Harish-Chandra's isomorphism, where $2\rho$ is the sum of the positive roots.
Thanks to the work of Mezo, we can identify the infinitesimal character of the elements of $\Pi_{\psi_{\infty}}$.
To lighten notation, we drop the subscript $\infty$ temporarily and consider an archimedean Arthur parameter $\psi : W_{\R} \times \mathrm{SL}_2(\C) \rightarrow {}^L \mathbf{G}$.
Recall that $W_{\C} = \C^{\times}$, $W_{\R} = W_{\C} \sqcup j W_{\C}$ where $j^2 = -1 \in W_{\C}$ and for any $z \in W_{\C}$, $jzj^{-1} = \bar{z}$.
Define a Langlands parameter $\varphi_{\psi}$ by composing $\psi$ with $W_{\R} \rightarrow W_{\R} \times \mathrm{SL}_2(\C)$ mapping $w \in W_{\R}$ to
$$ \left( w, \begin{pmatrix} ||w||^{1/2} & 0 \\ 0 & ||w||^{-1/2} \end{pmatrix}\right) $$
where $||\cdot|| : W_{\R} \rightarrow \R_{>0}$ is the unique morphism mapping $z \in W_{\C}$ to $z \bar{z}$.
Let $\mathcal{T}$ be a maximal torus in $\widehat{\mathbf{G}}$.
Conjugating by an element of $\widehat{\mathbf{G}}$ if necessary, we can assume that $\varphi_{\psi}(W_{\C}) \subset \mathcal{T}$ and write $\varphi_{\psi}(z) = \mu_1(z) \mu_2(\bar{z})$ for $z \in W_{\C}$, where $\mu_1, \mu_2 \in \C \otimes_{\Z} X_*(\mathcal{T})$ are such that $\mu_1 - \mu_2 \in X_*(\mathcal{T})$.
The conjugacy class of $(\mu_1, \mu_2)$ under the Weyl group $W(\mathcal{T}, \widehat{\mathbf{G}})$ is well-defined.
Note that for any maximal torus $\mathbf{T}$ of $\mathbf{G}_{\C}$ we can see $\mu_1, \mu_2$ as elements of $\C \otimes_{\Z} X^*(\mathbf{T})$, again canonically up to the action of the Weyl group.
\begin{lemm} \label{lemmCharInfArthur}
The Weyl group orbit of $\mu_1$ is the infinitesimal character of any element of $\Pi_{\psi}$.
\end{lemm}
\begin{proof}
Recall \cite{Arthur}[Theorem 2.2.1] that the packet $\Pi_{\psi}$ is characterised by twisted and standard endoscopic character identities involving the representation of $\mathbf{GL}_N(\R)$ having Langlands parameter $\mathrm{Std} \circ \varphi_{\psi}$.
The lemma follows from \cite{Mezo}[Lemma 24] (see also \cite{WalTTF1}[Corollaire 2.8]), which establishes the equivariance of twisted endoscopic transfer for the actions of the centers of the envelopping algebras.
\end{proof}
Attached to $\lambda$ is a unique (up to $\widehat{\mathbf{G}}$-conjugacy) \emph{discrete} parameter $\varphi_{\lambda} : W_{\R} \rightarrow {}^L \mathbf{G}$ having infinitesimal character $\lambda + \rho$.
We explicit the $\mathrm{GL}_N(\C)$-conjugacy class of $\mathrm{Std} \circ \varphi_{\lambda}$ in each case.
For $w \in \frac{1}{2}\Z_{\geq 0}$ it is convenient to denote the Langlands parameter $W_{\R} \rightarrow \mathrm{GL}_2(\C)$
$$ I_w = \mathrm{Ind}_{W_{\C}}^{W_{\R}} \left(z \mapsto (z/|z|)^{2w}\right) :\ z \in W_{\C} \mapsto \begin{pmatrix}
(z/|z|)^{2w} & 0 \\ 0 & (z/|z|)^{-2w} \end{pmatrix},\ j \mapsto \begin{pmatrix} 0 & (-1)^{2w} \\ 1 & 0 \end{pmatrix}. $$
Note that this was denoted $I_{2w}$ in \cite{ChRe} to emphasise motivic weight in a global setting.
We choose to emphasise Hodge weights, i.e.\ eigenvalues of the infinitesimal character: our $I_w$ has Hodge weights $w$ and $-w$.
Let $\epsilon_{\C / \R}$ be the non-trivial continuous character $W_{\R} \rightarrow \{ \pm 1 \}$, so that $I_0 = 1 \oplus \epsilon_{\C / \R}$.
If $\mathbf{G} = \mathbf{SO}_{2n+1}$, we can write $\lambda = k_1 e_1 + \dots + k_n e_n$ where $k_1 \geq \dots \geq k_n \geq 0$ are integers, and $\rho = (n-\frac{1}{2}) e_1 + (n-\frac{3}{2}) e_2 + \dots + \frac{1}{2}e_n$.
In this case $\mathrm{Std} \circ \varphi_{\lambda}$ is
$$ \bigoplus_{r=1}^n I_{k_r + n + 1/2 - r}.  $$
If $\mathbf{G} = \mathbf{Sp}_{2n}$, we can write $\lambda = k_1 e_1 + \dots + k_n e_n$ where $k_1 \geq \dots \geq k_n \geq 0$ are integers, and $\rho = n e_1 + (n-1) e_2 + \dots + e_n$.
Then $\mathrm{Std} \circ \varphi_{\lambda}$ is
$$ \epsilon_{\C / \R}^n \oplus \bigoplus_{r=1}^n I_{k_r + n + 1 - r}.  $$
Finally, if $\mathbf{G} = \mathbf{SO}_{4n}$, we can write $\lambda = k_1 e_1 + \dots + k_{2n} e_{2n}$ where $k_1 \geq \dots \geq k_{2n-1} \geq |k_{2n}|$ are integers, and $\rho = (2n-1) e_1 + (2n-2) e_2 + \dots + e_{2n-1}$.
Then $\mathrm{Std} \circ \varphi_{\lambda}$ is
$$ \bigoplus_{r=1}^{2n} I_{k_r + 2n - r}.  $$
Replacing $(k_1, \dots, k_{2n-1}, k_{2n})$ with $(k_1, \dots, k_{2n-1}, -k_{2n})$ yields the same conjugacy class under $\mathrm{GL}_N(\C)$.

From this explicit description one can deduce several restrictions on the global parameters $\psi \in \Psi(\mathbf{G})^{\mathrm{unr}}$ contributing non-trivially to the spectral side \ref{spectralsideendo}.
These observations were already made in \cite{ChRe}, using a different formulation.
We define $\Psi(\mathbf{G})^{\lambda}$ as the subset of $\Psi(\mathbf{G})$ consisting of $\psi$ such that the infinitesimal character of $\psi_{\infty}$ is equal to $\lambda + \rho$.
Define also $\Psi(\mathbf{G})^{\mathrm{unr}, \lambda} = \Psi(\mathbf{G})^{\mathrm{unr}} \cap \Psi(\mathbf{G})^{\lambda}$.
\begin{enumerate}
\item In the first two cases ($\mathbf{G} = \mathbf{SO}_{2n+1}$ of $\mathbf{Sp}_{2n}$) the infinitesimal character of $\mathrm{Std} \circ \varphi_{\lambda}$ is algebraic and regular in the sense of Clozel \cite{Clozel}.
Clozel's definition of ``algebraic'' is ``C-algebraic'' in the sense of \cite{BuzGee}, and we will also use the term ``C-algebraic'' to avoid confusion.
In the third case ($\mathbf{G} = \mathbf{SO}_{4n}$) we have that $||\cdot||^{1/2} \otimes \left(\mathrm{Std} \circ \varphi_{\lambda}\right)$ is C-algebraic, but not always regular.
It is regular if and only if $k_{2n} \neq 0$.
In all cases, Clozel's purity lemma \cite{Clozel}[Lemme 4.9] implies that if $\psi = \boxplus_i \pi_i[d_i] \in \Psi(\mathbf{G})^{\lambda}$, then for all $i$ the self-dual cuspidal automorphic representation $\pi_i$ of $\mathbf{GL}_{n_i} / \Q$ is \emph{tempered} at the real place.
Equivalently, $\psi_{\infty}(W_{\R})$ is bounded.
\item Let $\Psi(\mathbf{G})_{\mathrm{sim}}$ be the set of \emph{simple} formal Arthur parameters in $\Psi(\mathbf{G})$, i.e.\ those $\psi = \boxplus_{i \in I} \pi_i[d_i]$ such that $I = \{i_0\}$ and $d_{i_0} = 1$.
Denote $\Psi(\mathbf{G})_{\mathrm{sim}}^{\lambda} = \Psi(\mathbf{G})_{\mathrm{sim}} \cap \Psi(\mathbf{G})^{\lambda}$.
Then $\Psi(\mathbf{G})_{\mathrm{sim}}^{\lambda}$ is the set of self-dual cuspidal automorphic representations of $\mathbf{GL}_N / \Q$ such that the central character of $\pi$ is trivial and the local Langlands parameter of $\pi_{\infty}$ is $\mathrm{Std} \circ \varphi_{\lambda}$.
Indeed in all three cases $\mathrm{Std} \circ \varphi_{\lambda}$ is either orthogonal or symplectic, and thus $\pi_{\infty}$ \emph{determines} $s(\pi)$.
\item Let $m \geq 1$ and consider a self-dual cuspidal automorphic representation $\pi$ of $\mathbf{GL}_{2m} / \Q$ such that $|\det|^{1/2} \otimes \pi$ is C-algebraic regular.
Self-duality implies that the central character $\chi_{\pi}$ of $\pi$ is quadratic, i.e.\ $\chi_{\pi} : \A^{\times} / \Q^{\times} \rightarrow \{ \pm 1 \}$.
Since $|\det|^{1/2} \otimes \pi$ is C-algebraic and regular, there are unique integers $w_1 > \dots > w_m > 0$ such that the local Langlands parameter of $\pi_{\infty}$ is
$$ \bigoplus_{r=1}^m I_{w_r}, $$
which implies that $\chi_{\pi}|_{\R^{\times}}(-1) = (-1)^m$.
If moreover we assume that $\pi$ is everywhere unramified, then $\chi_{\pi}$ is trivial on $\prod_p \Zp^{\times}$.
Since $\A^{\times} = \Q^{\times} \R_{>0} \prod_p \Zp^{\times}$, this implies that $\chi_{\pi}$ is trivial, and thus $m$ must be even.
\item The previous point has the following important consequence for our inductive computations.
Let $\mathbf{G}$ be a split symplectic or special orthogonal group admitting discrete series at the real place, and $\lambda$ a dominant weight for $\mathbf{G}$.
Let $\psi = \boxplus_i \pi_i[d_i] \in \Psi(\mathbf{G})^{\mathrm{unr}, \lambda}$.
Then for any $i$, there is a split symplectic or special orthogonal group $\mathbf{G}'$ admitting discrete series at the real place and a dominant weight $\lambda'$ for $\mathbf{G}'$ such that $\pi_i \in \Psi(\mathbf{G}')^{\mathrm{unr}, \lambda'}_{\mathrm{sim}}$.
We emphasise that this holds even if $\mathbf{G} = \mathbf{SO}_{4n}$ and $\lambda = k_1 e_1 + \dots + k_{2n} e_{2n}$ with $k_{2n} = 0$.
To be precise, we have the following classification:
\begin{enumerate}
\item $\mathbf{G} = \mathbf{SO}_{2n+1}$ and thus $\widehat{\mathbf{G}} = \mathrm{Sp}_{2n}(\C)$.
For a dominant weight $\lambda$ and $\psi = \boxplus_{i \in I} \pi_i[d_i] \in \Psi(\mathbf{G})^{\mathrm{unr}, \lambda}$, there is a canonical decomposition $I = I_1 \sqcup I_2 \sqcup I_3$ where
\begin{enumerate}
\item for all $i \in I_1$, $d_i$ is odd, $n_i$ is even and $\pi_i \in \Psi(\mathbf{SO}_{n_i+1})^{\mathrm{unr}, \lambda'}_{\mathrm{sim}}$,
\item for all $i \in I_2$, $d_i$ is even, $n_i$ is divisible by $4$ and $\pi_i \in \Psi(\mathbf{SO}_{n_i})^{\mathrm{unr}, \lambda'}_{\mathrm{sim}}$,
\item $\mathrm{card}(I_3) \in \{0,1\}$ and if $I_3 = \{ i \}$, $d_i$ is even, $n_i$ is odd and $\pi_i \in \Psi(\mathbf{Sp}_{n_i-1})^{\mathrm{unr}, \lambda'}_{\mathrm{sim}}$.
\end{enumerate}
\item $\mathbf{G} = \mathbf{Sp}_{2n}$ and thus $\widehat{\mathbf{G}} = \mathrm{SO}_{2n+1}(\C)$.
For a dominant weight $\lambda$ and $\psi = \boxplus_{i \in I} \pi_i[d_i] \in \Psi(\mathbf{G})^{\mathrm{unr}, \lambda}$, there is a canonical decomposition $I = I_1 \sqcup I_2 \sqcup I_3$ where
\begin{enumerate}
\item $I_1 = \{ j \}$, $d_j$ is odd, $n_j$ is odd and $\pi_j \in \Psi(\mathbf{Sp}_{n_j-1})^{\mathrm{unr}, \lambda'}_{\mathrm{sim}}$,
\item for all $i \in I_2$, $d_i$ is odd, $n_i$ is divisible by $4$ and $\pi_i \in \Psi(\mathbf{SO}_{n_i})^{\mathrm{unr}, \lambda'}_{\mathrm{sim}}$,
\item for all $i \in I_3$, $d_i$ is even, $n_i$ is even and $\pi_i \in \Psi(\mathbf{SO}_{n_i+1})^{\mathrm{unr}, \lambda'}_{\mathrm{sim}}$.
\end{enumerate}
Note that $n_jd_j = 2n+1 \mod 4$.
\item $\mathbf{G} = \mathbf{SO}_{4n}$ and thus $\widehat{\mathbf{G}} = \mathrm{SO}_{4n}(\C)$.
For a dominant weight $\lambda$ and $\psi = \boxplus_{i \in I} \pi_i[d_i] \in \Psi(\mathbf{G})^{\mathrm{unr}, \lambda}$, there is a canonical decomposition $I = I_1 \sqcup I_2 \sqcup I_3$ where
\begin{enumerate}
\item for all $i \in I_1$, $d_i$ is odd, $n_i$ is divisible by $4$ and $\pi_i \in \Psi(\mathbf{SO}_{n_i})^{\mathrm{unr}, \lambda'}_{\mathrm{sim}}$,
\item for all $i \in I_2$, $d_i$ is even, $n_i$ is even and $\pi_i \in \Psi(\mathbf{SO}_{n_i+1})^{\mathrm{unr}, \lambda'}_{\mathrm{sim}}$,
\item $\mathrm{card}(I_3) \in \{0,2\}$.
If $I_3 = \{ i,j \}$ and up to exchanging $i$ and $j$, $d_i=1$ and $d_j$ is odd, $n_i$ and $n_j$ are odd, and $\pi_i \in \Psi(\mathbf{Sp}_{n_i-1})^{\mathrm{unr}, \lambda'}_{\mathrm{sim}}$ and $\pi_j \in \Psi(\mathbf{Sp}_{n_j-1})^{\mathrm{unr}, \lambda'}_{\mathrm{sim}}$.
\end{enumerate}
\end{enumerate}
Note that in all three cases, if $\lambda$ is \emph{regular} then for any $\psi = \boxplus_{i \in I} \pi_i[d_i] \in \Psi(\mathbf{G})^{\mathrm{unr}, \lambda}$ we have that $\psi_{\infty} = \varphi_{\lambda}$ and thus all $d_i$'s are equal to $1$ (i.e.\ $\psi$ is formally tempered) and moreover in the third case $I_3 = \emptyset$.
\end{enumerate}
As in the introduction, it will be convenient to have a more concrete notation for the sets $\Psi(\mathbf{G})^{\mathrm{unr}, \lambda}_{\mathrm{sim}}$.
\begin{enumerate}
\item For $n \geq 1$, the dominant weights for $\mathbf{G} = \mathbf{SO}_{2n+1}$ are the characters $\lambda = k_1 e_1 + \dots + k_n e_n$ such that $k_1 \geq \dots \geq k_n \geq 0$.
Then $\lambda + \rho = w_1 e_1 + \dots + w_n e_n$ where $w_r = k_r +n+\frac{1}{2}-r$, so that $w_1 > \dots > w_n > 0$ belong to $\frac{1}{2}\Z \smallsetminus \Z$.
Define $S(w_1, \dots, w_n) = \Psi(\mathbf{SO}_{2n+1})^{\mathrm{unr}, \lambda}_{\mathrm{sim}}$, that is the set of self-dual automorphic cuspidal representations of $\mathbf{GL}_{2n} / \Q$ which are everywhere unramified and with Langlands parameter at the real place
$$ I_{w_1} \oplus \dots \oplus I_{w_n}. $$
Equivalently we could replace the last condition by ``with infinitesimal character having eigenvalues $\{ \pm w_1, \dots, \pm w_n \}$''.
Here $S$ stands for ``symplectic'', as $\widehat{\mathbf{G}} = \mathrm{Sp}_{2n}(\C)$.
\item For $n \geq 1$, the dominant weights for $\mathbf{G} = \mathbf{Sp}_{2n}$ are the characters $\lambda = k_1 e_1 + \dots + k_n e_n$ such that $k_1 \geq \dots \geq k_n \geq 0$.
Then $\lambda + \rho = w_1 e_1 + \dots + w_n e_n$ where $w_r = k_r +n+1-r$, so that $w_1 > \dots > w_n > 0$ are integers.
Define $O_o(w_1, \dots, w_n) = \Psi(\mathbf{Sp}_{2n})^{\mathrm{unr}, \lambda}_{\mathrm{sim}}$, that is the set of self-dual automorphic cuspidal representations of $\mathbf{GL}_{2n+1} / \Q$ which are everywhere unramified and with Langlands parameter at the real place
$$ I_{w_1} \oplus \dots \oplus I_{w_n} \oplus \epsilon_{\C / \R}^n. $$
Equivalently we could replace the last condition by ``with infinitesimal character having eigenvalues $\{ \pm w_1, \dots, \pm w_n , 0\}$''.
Here $O_o$ stands for ``odd orthogonal'', as $\widehat{\mathbf{G}} = \mathrm{SO}_{2n+1}(\C)$.
\item For $n \geq 1$, the dominant weights for $\mathbf{G} = \mathbf{SO}_{4n}$ are the characters $\lambda = k_1 e_1 + \dots + k_{2n} e_{2n}$ such that $k_1 \geq \dots \geq k_{2n-1} \geq |k_{2n}|$.
Since we only consider quantities invariant under outer conjugation we assume $k_{2n} \geq 0$.
Then $\lambda + \rho = w_1 e_1 + \dots + w_{2n} e_{2n}$ where $w_r = k_r +n-r$, so that $w_1 > \dots > w_{2n-1} > w_{2n} \geq 0$ are integers.
Define $O_e(w_1, \dots, w_{2n}) = \Psi(\mathbf{SO}_{4n})^{\mathrm{unr}, \lambda}_{\mathrm{sim}}$, that is the set of self-dual automorphic cuspidal representations of $\mathbf{GL}_{4n} / \Q$ which are everywhere unramified and with Langlands parameter at the real place
$$ I_{w_1} \oplus \dots \oplus I_{w_{2n}}. $$
In this case also we could replace the last condition by ``with infinitesimal character having eigenvalues $\{ \pm w_1, \dots, \pm w_{2n} \}$'', even when $k_{2n}=0$.
Here $O_e$ stands for ``even orthogonal'', as $\widehat{\mathbf{G}} = \mathrm{SO}_{4n}(\C)$.
\end{enumerate}

It is now natural to try to compute the cardinality of $\Psi(\mathbf{G})^{\mathrm{unr}, \lambda}_{\mathrm{sim}}$, inductively on the dimension of $\mathbf{G}$.
Observe that for $\psi \in \Psi(\mathbf{G})_{\mathrm{sim}}$, the group $S_{\psi}$ is trivial.
Thus the contribution of any $\psi \in \Psi(\mathbf{G})^{\mathrm{unr}, \lambda}_{\mathrm{sim}}$ to the spectral side \ref{spectralsideendo} is simply
$$ \sum_{\pi_{\infty} \in \Pi_{\psi_{\infty}}} \mathrm{EP}\left( \pi_{\infty} \otimes V_{\lambda}^* \right). $$
Recall that for such a $\psi$, the local Arthur parameter $\psi_{\infty}$ is $\varphi_{\lambda}$.
In that case Arthur defines $\Pi_{\varphi_{\lambda}}$ as the L-packet that Langlands \cite{Langlands} associates with $\varphi_{\lambda}$.
In the next section we will review these packets in more detail, in particular Shelstad's definition of $\langle \cdot, \pi_{\infty} \rangle$ for $\pi_{\infty} \in \Pi_{\varphi_{\lambda}}$, but since $S_{\psi}$ is trivial all that matters for now is that $\mathrm{card}(\Pi_{\varphi_{\lambda}})$ is positive (and easily computed) and that all the representations in $\Pi_{\varphi_{\lambda}}$ are discrete series.
By \cite{BoWa}[ch. III, Thm. 5.1] for any $\pi_{\infty} \in \Pi_{\varphi_{\lambda}}$,
$$ \mathrm{EP} \left(\pi_{\infty} \otimes V_{\lambda}^* \right) = (-1)^{q(\mathbf{G}(\R))} $$
and thus to compute the cardinality of $\Psi(\mathbf{G})^{\mathrm{unr}, \lambda}_{\mathrm{sim}}$ we want to compute the contribution of $\Psi(\mathbf{G})^{\mathrm{unr}, \lambda} \smallsetminus \Psi(\mathbf{G})^{\mathrm{unr}, \lambda}_{\mathrm{sim}}$ to the spectral side \ref{spectralsideendo}.

This is particularly easy if $\lambda$ is regular, since as we observed above in that case any $\psi \in \Psi(\mathbf{G})^{\mathrm{unr}, \lambda}$ is ``formally tempered'' or ``formally of Ramanujan type'', i.e.\ $\psi_{\infty} = \varphi_{\lambda}$.
Moreover $\epsilon_{\psi}$ is trivial.
Shelstad's results reviewed in the next section allow the explicit determination of the number of $\pi_{\infty} \in \Pi_{\varphi_{\lambda}}$ such that $\langle \cdot, \pi_{\infty} \rangle$ is equal to a given character of $S_{\psi_{\infty}}$.

The general case is more interesting.
The determination of $\epsilon_{\psi}$ in the ``conductor one'' case was done in \cite{ChRe}, and the result is simple since it involves only epsilon factors at the real place of $\Q$.
In all three cases, for any $\psi = \boxplus_{i \in I} \pi_i[d_i] \in \Psi(\mathbf{G})^{\mathrm{unr}, \lambda}$ the abelian $2$-group $S_{\psi}$ is generated by $(s_i)_{i \in J}$ where $J = \{ i \in I\,|\, n_id_i \text{ is even}\}$ and $s_i \in C_{\psi}$ is formally $-\mathrm{Id}$ on the space of $\pi_i[d_i]$ and $\mathrm{Id}$ on the space of $\pi_j[d_j]$ for $j \neq i$.
By \cite{ChRe}[(3.10)]
$$ \epsilon_{\psi}(s_i) = \prod_{j \in I \smallsetminus \{i\}} \epsilon(\pi_i \times \pi_j)^{\mathrm{min}(d_i,d_j)} $$
and since $\pi_i$ and $\pi_j$ are everywhere unramified $\epsilon(\pi_i \times \pi_j)$ can be computed easily from the tensor product of the local Langlands parameters of $(\pi_i)_{\infty}$ and $(\pi_j)_{\infty}$.
Note that by \cite{Arthur}[Theorem 1.5.3] $\epsilon(\pi_i \times \pi_j) = 1$ if $s(\pi_i)s(\pi_j)=1$.
The explicit computation of $\Pi_{\psi_{\infty}}$, along with the map $\Pi_{\psi_{\infty}} \rightarrow S_{\psi_{\infty}}^{\wedge}$, does not follow directly from Arthur's work, even in our special case where the infinitesimal character of $\psi_{\infty}$ is that of an algebraic representation $V_{\lambda}$.
We will need to make an assumption (Assumption \ref{assumAJweak}) relating Arthur's packet $\Pi_{\psi_{\infty}}$ to the packets constructed by Adams and Johnson in \cite{AdJo}.
The latter predate Arthur's recent work, in fact \cite{AdJo} has corroborated Arthur's general conjectures: see \cite{ArthurUnip}[§5].
Under this assumption, we will also be able to compute the Euler-Poincaré characteristic of any element of $\Pi_{\psi_{\infty}}$ in section \ref{sectionAJ}.

\begin{rema}
Our original goal was to compute, for a given group $\mathbf{G}/\Q$ as above, dominant weight $\lambda$ and simple $(\mathfrak{g}, K_{\infty})$-module module $X$ with infinitesimal character $\lambda + \rho$, the multiplicity of $X$ in $\mathcal{A}_{\mathrm{disc}}(\mathbf{G}(\Q) \backslash \mathbf{G}(\A))^{\mathbf{G}(\widehat{\Z})}$.
This is possible once the cardinalities of $\Psi(\mathbf{G}')^{\mathrm{unr}, \lambda'}_{\mathrm{sim}}$ are computed, under Assumption \ref{assumAJ} if we do not assume that $\lambda$ is regular.
However, Arthur's endoscopic classification shows that computing $\mathrm{card} \left(\Psi(\mathbf{G}')^{\mathrm{unr}, \lambda'}_{\mathrm{sim}} \right)$ is a more interesting problem from an arithmetic perspective, since conjecturally we are counting the number of self-dual motives over $\Q$ with conductor $1$ and given Hodge weights.
\end{rema}
\begin{rema}
Except in the even orthogonal case with $\lambda = k_1 e_1 + \dots + k_{2n}e_{2n}$ and $k_{2n}=0$, it is known that any $\psi \in \Psi(\mathbf{G})^{\mathrm{unr}, \lambda}_{\mathrm{sim}}$ is tempered also at the finite places by \cite{ClozelPurity}.
\end{rema}
\begin{rema}
If $\mathbf{G}$ is symplectic or even orthogonal, it has non-trivial center $\mathbf{Z}$ isomorphic to $\mu_2$.
Thus $\mathbf{Z}(\R) \subset \mathbf{Z}(\Q) \mathbf{Z}(\widehat{\Z})$, and $\mathbf{Z}(\R)$ acts trivially on $\mathcal{A}_{\mathrm{disc}}(\mathbf{G}(\Q) \backslash \mathbf{G}(\A))^{\mathbf{G}(\widehat{\Z})}$.
This implies that $\Psi(\mathbf{G})^{\mathrm{unr}, \lambda}_{\mathrm{sim}}$ is empty if $\lambda|_{\mathbf{Z}(\R)}$ is not trivial, since $\mathbf{Z}(\R)$ acts by $\lambda$ on any discrete series representation with infinitesimal character $\lambda + \rho$.
Using the concrete description above, it is elementary to deduce that in fact $\Psi(\mathbf{G})^{\mathrm{unr}, \lambda}$ is empty if $\lambda|_{\mathbf{Z}(\R)}$ is not trivial.
\end{rema}

\subsection{Euler-Poincaré characteristic of cohomological archimedean Arthur packets}
\subsubsection{Tempered case: Shelstad's parametrization of L-packets}
\label{sectionDS}

For archimedean local fields in the tempered case the A-packets $\Pi_{\psi}$ in \cite{Arthur} are not defined abstractly using the global twisted trace formula.
Rather, Arthur \emph{defines} $\Pi_{\varphi_{\lambda}}$ as the L-packet that Langlands \cite{Langlands} associates with $\varphi_{\lambda}$, and the map $\Pi_{\varphi_{\lambda}} \rightarrow S_{\varphi_{\lambda}}^{\wedge}, \pi \mapsto \langle \cdot , \pi \rangle$ is defined by Shelstad's work, which we review below.
Mezo \cite{Mezo2} has shown that these Langlands-Shelstad L-packets satisfy the twisted endoscopic character relation \cite{Arthur}[Theorem 2.2.1 (a)], and Shelstad's work contains the ``standard'' endoscopic character relations \cite{Arthur}[Theorem 2.2.1 (b)].

In this section we will only be concerned with the local field $\R$ and thus we drop the subscripts $\infty$, and we denote $\mathrm{Gal}(\C / \R) = \{1, \sigma\}$.
Let $\mathbf{G}$ be a reductive group over $\R$, and denote by $\mathbf{A}_{\mathbf{G}}$ the biggest split torus in the connected center $\mathbf{Z}_{\mathbf{G}}$ of $\mathbf{G}$.
Let us assume that $\mathbf{G}$ has a maximal torus (defined over $\R$) which is anisotropic modulo $\mathbf{A}_{\mathbf{G}}$, i.e.\ $\mathbf{G}(\R)$ has essentially discrete series.
Consider a dominant weight $\lambda_0$ for $(\mathbf{G}_{\mathrm{der}})_{\C}$ defining an algebraic representation $V_{\lambda_0}$ of $\mathbf{G}_{\mathrm{der}}(\C)$ and a continuous character $\chi_0 : \mathbf{Z}_{\mathbf{G}}(\R) \rightarrow \C^{\times}$ such that $\chi_0$ and $\lambda_0$ coincide on $\mathbf{Z}_{\mathbf{G}}(\R) \cap \mathbf{G}_{\mathrm{der}}(\C)$.
Let $\Pi_{\mathrm{disc}}(\lambda_0, \chi_0)$ be the finite set of essentially discrete series representations $\pi$ of $\mathbf{G}(\R)$ such that
\begin{itemize}
\item 
$\pi|_{\mathbf{G}_{\mathrm{der}}(\R)}$ has the same infinitesimal character as $V_{\lambda_0}|_{\mathbf{G}_{\mathrm{der}}(\R)}$,
\item
$\pi|_{\mathbf{Z}_{\mathbf{G}}(\R)} = \chi_0$.
\end{itemize}
Harish-Chandra has shown that inside this L-packet of essentially discrete series, the representations are parameterised by the conjugacy classes (under $\mathbf{G}(\R)$) of pairs $(\mathbf{B}, \mathbf{T})$ where $\mathbf{T}$ is a maximal torus of $\mathbf{G}$ anisotropic modulo $\mathbf{A}_{\mathbf{G}}$ and $\mathbf{B}$ is a Borel subgroup of $\mathbf{G}_{\C}$ containing $\mathbf{T}_{\C}$.
For such a pair $(\mathbf{B}, \mathbf{T})$, $\chi_0$ and the character $\lambda_0$ of $\mathbf{T}_{\mathrm{der}}(\R)$ which is dominant for $\mathbf{B}$ extend uniquely to a character $\lambda_{\mathbf{B}}$ of $\mathbf{T}(\R)$.
If we fix such a pair $(\mathbf{B}, \mathbf{T})$, the pairs $(\mathbf{B}', \mathbf{T})$ which are in the same conjugacy class form an orbit under the subgroup $W_c := W(\mathbf{G}(\R), \mathbf{T}(\R))$ of $W := W(\mathbf{G}(\C), \mathbf{T}(\C))$.
Concretely, if $\pi \in \Pi_{\mathrm{disc}}(\lambda_0, \chi_0)$ is the representation associated with this conjugacy class, then for any $\gamma \in \mathbf{T}(\R)_{\mathbf{G}-\mathrm{reg}}$,
$$ \Theta_{\pi}(\gamma) = (-1)^{q(\mathbf{G})} \sum_{w \in W_c} \frac{\lambda_{w \mathbf{B} w^{-1}}(\gamma)}{\Delta_{w \mathbf{B} w^{-1}}(\gamma)}$$
where $\Theta_{\pi}$ is Harish-Chandra's character for $\pi$, and $\Delta_{\mathbf{B}}(\gamma) = \prod_{\alpha \in R(\mathbf{T}, \mathbf{B})} (1 - \alpha(\gamma)^{-1})$.
Therefore the choice of $(\mathbf{B}, \mathbf{T})$ as a base point identifies the set of conjugacy classes with $W_c \backslash W$, by $g \in N(\mathbf{G}(\C), \mathbf{T}(\C)) \mapsto (g \mathbf{B} g^{-1}, \mathbf{T})$.

Langlands \cite{Langlands} and Shelstad \cite{She1}, \cite{She2}, \cite{She3} gave another formulation for the parameterisation inside an L-packet, more suitable for writing endoscopic character relations.
By definition of the L-group we have a splitting $(\mathcal{B}, \mathcal{T}, ( \mathcal{X}_{\alpha} )_{\alpha \in \Delta})$ of $\widehat{\mathbf{G}}$ which defines a section of $\mathrm{Aut}(\widehat{\mathbf{G}}) \rightarrow \mathrm{Out}(\widehat{\mathbf{G}})$ and ${}^L \mathbf{G} = \widehat{\mathbf{G}} \rtimes W_{\R}$.
Let $(\mathbf{B}, \mathbf{T})$ be as above.
Thanks to $\mathbf{B}$ we have a canonical isomorphism $\widehat{\mathbf{T}} \rightarrow \mathcal{T}$, which can be extended into an embedding of L-groups $\iota : {}^L \mathbf{T} \rightarrow {}^L \mathbf{G}$ as follows.
For $z \in W_{\C}$, define $\iota(z) = \prod_{\alpha \in R_{\mathcal{B}}} \alpha^{\vee}(z/|z|) \rtimes z$ where $R_{\mathcal{B}}$ is the set of roots of $\mathcal{T}$ in $\mathcal{B}$.
Define $\iota(j) = n_0 \rtimes j$ where $n_0 \in N(\widehat{\mathbf{G}}, \mathcal{T}) \cap \widehat{\mathbf{G}}_{\mathrm{der}}$ represents the longest element of the Weyl group $W(\widehat{\mathbf{G}}, \mathcal{T})$ for the order defined by $\mathcal{B}$.
Then $\iota$ is well-defined thanks to \cite{Langlands}[Lemma 3.2].
Since conjugation by $n_0 \rtimes j$ acts by $t \mapsto t^{-1}$ on $\mathcal{T} \cap \widehat{\mathbf{G}}_{\mathrm{der}}$, the conjugacy class of $\iota$ does not depend on the choice of $n_0$.
The character $\lambda_{\mathbf{B}}$ of $\mathbf{T}(\R)$ corresponds to a Langlands parameter $\varphi_{\lambda_{\mathbf{B}}} : W_{\R} \rightarrow {}^L \mathbf{T}$.
If $\mathbf{G}$ is semisimple, $\lambda_{\mathbf{B}}$ is the restriction to $\mathbf{T}(\R)$ of an element of $X^*(\mathbf{T}) = X_*(\mathcal{T})$ and for any $z \in W_{\C}$, $\varphi_{\lambda_{\mathbf{B}}}(z) = \lambda_{\mathbf{B}}(z/|z|)$.
Composing $\varphi_{\lambda_{\mathbf{B}}}$ with $\iota$ we get a Langlands parameter $\varphi : W_{\R} \rightarrow {}^L \mathbf{G}$, whose conjugacy class under $\widehat{\mathbf{G}}$ does not depend on the choice of $(\mathbf{B}, \mathbf{T})$.
Langlands has shown that the map $(\lambda_0, \chi_0) \mapsto \varphi$ is a bijection onto the set of conjugacy classes of discrete Langlands parameters, i.e.\ Langlands parameters $\varphi$  such that $S_{\varphi} := \mathrm{Cent}(\varphi, \widehat{\mathbf{G}})/Z(\widehat{\mathbf{G}})^{\mathrm{Gal}(\C/\R)}$ is finite.

Consider a discrete Langlands parameter $\varphi$, and denote by $\Pi_{\varphi} = \Pi(\lambda_0, \chi_0)$ the corresponding L-packet.
Assume that $\mathbf{G}$ is quasisplit and fix a Whittaker datum (see \cite{Kal} for the general case).
Then Shelstad defines an injective map $\Pi_{\varphi} \rightarrow S_{\varphi}^{\wedge}$, $\pi \mapsto \langle \cdot, \pi \rangle$.
It has the property that $\langle \cdot, \pi \rangle$ is trivial if $\pi$ is the unique generic (for the given Whittaker datum) representation in the L-packet.

Recall the relation between these two parametrizations of the discrete L-packets.
Let $(\mathbf{B}, \mathbf{T})$ be as above, defining an embedding $\iota : {}^L \mathbf{T} \rightarrow {}^L \mathbf{G}$ and recall that $W$ and $W_c$ denote the complex and real Weyl groups.
Let $C_{\varphi} = \mathrm{Cent}(\varphi, \widehat{\mathbf{G}})$, so that $S_{\varphi} = C_{\varphi}/Z(\widehat{\mathbf{G}})^{\mathrm{Gal}(\C/\R)}$.
Using $\iota$ we have an isomorphism between $H^1(\R, \mathbf{T})$ and $\pi_0(C_{\varphi})^{\wedge}$.
We have a bijection
$$ W_c \backslash W \rightarrow \ker \left( H^1(\R, \mathbf{T}) \rightarrow H^1(\R, \mathbf{G}) \right)$$
mapping $g \in N_{\mathbf{G}(\C)}(\mathbf{T}(\C))$ to $(\sigma \mapsto g^{-1} \sigma(g))$.
Kottwitz \cite{KottEllSing} has defined a natural morphism $H^1(\R, \mathbf{G}) \rightarrow \pi_0\left(Z(\widehat{\mathbf{G}})^{\mathrm{Gal}(\C/\R)}\right)^{\wedge}$ and thus the above bijection yields an injection $\eta : W_c \backslash W \rightarrow S_{\varphi}^{\wedge}$.
If $\pi \in \Pi_{\varphi}$ corresponds to (the conjugacy class of) $(\mathbf{B}, \mathbf{T})$ and $\pi' \in \Pi_{\varphi}$ corresponds to $(g\mathbf{B}g^{-1}, \mathbf{T})$, then for any $s \in S_{\varphi}$,
$$ \frac{\langle s, \pi \rangle}{\langle s, \pi' \rangle} = \eta(g)(s). $$
Finally, the generic representation in $\Pi_{\varphi}$ corresponds to a pair $(\mathbf{B}, \mathbf{T})$ as above such that all the simple roots for $\mathbf{B}$ are noncompact.
This is a consequence of \cite{Kostant}[Theorem 3.9] and \cite{Vogan}[Theorem 6.2].
In particular there \emph{exists} such a pair $(\mathbf{B}, \mathbf{T})$.
We will make use of the converse in the non-tempered case.
\begin{lemm} \label{lemmRoppqs}
Let $\mathbf{H}$ be a reductive group over $\R$.
Assume that $\mathbf{T}$ is a maximal torus of $\mathbf{H}$ which is anisotropic modulo $\mathbf{A}_{\mathbf{H}}$, and assume that there exists a Borel subgroup $\mathbf{B} \supset \mathbf{T}_{\C}$ of $\mathbf{H}_{\C}$ such that all the simple roots of $\mathbf{T}$ in $\mathbf{B}$ are non-compact.
Then $\mathbf{H}$ is quasisplit.
\end{lemm}
\begin{proof}
We can assume that $\mathbf{H}$ is semisimple.
We use the ``$\R$-opp splittings'' of \cite{She3}[§12].
Let $\Delta$ be the set of simple roots of $\mathbf{T}$ in $\mathbf{B}$.
For any $\alpha \in \Delta$ we can choose an $\mathfrak{sl}_2$-triple $(H_{\alpha}, X_{\alpha}, Y_{\alpha})$ in $\mathfrak{h} = \C \otimes_{\R} \mathrm{Lie}(\mathbf{H}(\R))$.
The pair $(X_{\alpha}, Y_{\alpha})$ is not unique: it could be replaced by $(xX_{\alpha}, x^{-1} Y_{\alpha})$ for any $x \in \C^{\times}$.
Since $\sigma(\alpha) = -\alpha$, $\sigma(X_{\alpha}) = y Y_{\alpha}$ for some $y \in \C^{\times}$, and $y \in \R^{\times}$ because $\sigma$ is an involution.
The sign of $y$ does not depend on the choice of $(X_{\alpha}, Y_{\alpha})$, and making some other choice if necessary, we can assume that $y = \pm 1$.
It is easy to check that $\alpha$ is non-compact if and only if $y>0$.
Thus the hypotheses imply the existence of an $\R$-opp splitting, that is a splitting $(X_{\alpha})_{\alpha \in \Delta}$ such that $\sigma(X_{\alpha})=Y_{\alpha}$ for any $\alpha$.
Note that this splitting is unique up to the action of $\mathbf{T}(\R)$.

Let $\mathbf{H}'$ be the quasisplit reductive group over $\R$ such that $\mathbf{H}'$ admits an anisotropic maximal torus and $\mathbf{H}_{\C} \simeq \mathbf{H}'_{\C}$.
We know that $\mathbf{H}'$ admits a pair $(\mathbf{B}', \mathbf{T}')$ where $\mathbf{T}'$ is an anisotropic maximal torus and all the simple roots of $\mathbf{B}'$ are non-compact.
Therefore there exists an $\R$-opp splitting $(X_{\alpha}')_{\alpha' \in \Delta'}$ for $(\mathbf{B}', \mathbf{T}')$.

There is a unique isomorphism $f : \mathbf{H}_{\C} \rightarrow \mathbf{H}'_{\C}$ identifying $(\mathbf{B}, \mathbf{T}_{\C}, (X_{\alpha})_{\alpha \in \Delta})$ with $(\mathbf{B}', \mathbf{T}'_{\C}, (X'_{\alpha})_{\alpha \in \Delta'})$ and to conclude we only have to show that it is defined over $\R$, i.e.\ that it is Galois-equivariant.
It is obviously the case on $\mathbf{T}$, since any automorphism of $\mathbf{T}_{\C}$ is defined over $\R$.
Moreover by construction $f(\sigma(X_{\alpha})) = \sigma(X'_{f(\alpha)})$ for any $\alpha \in \Delta$.
Since $\mathbf{T}_{\C}$ and the one-dimensional unipotent groups corresponding to $\pm \alpha$ for $\alpha \in \Delta$ generate $\mathbf{H}_{\C}$, $f$ is $\sigma$-equivariant.
\end{proof}
There are as many conjugacy classes of such pairs $(\mathbf{B}, \mathbf{T})$ such that all the simple roots are non-compact as there are conjugacy classes of Whittaker datum.
For the adjoint group $\mathbf{SO}_{2n+1}$ there is a single conjugacy class, whereas for $\mathbf{G} = \mathbf{Sp}_{2n}$ or $\mathbf{SO}_{4n}$ there are two.
However, for our purposes it will fortunately not be necessary to precise which pair $(\mathbf{B}, \mathbf{T})$ corresponds to each conjugacy class of Whittaker datum.

For the quasi-split group $\mathbf{G} = \mathbf{SO}(V,q)$ where $\dim V \geq 3$ and $\mathrm{disc}(q) > 0$, $\mathbf{T}$ is the stabiliser of a direct orthogonal sum
$$ P_1 \oplus \dots \oplus P_n $$
where each $P_i$ is a definite plane and $n = \lfloor \dim V / 2 \rfloor$.
Let $I_+$ (resp.\ $I_-$) be the set of $i \in \{1, \dots, n\}$ such that $P_i$ is positive (resp.\ negative), $V_- = \bigoplus_{i \in I_-} P_i$ and $V_+ = V_-^{\perp}$.
The group $K$ of real points of
$$ \mathbf{S} \left( \mathbf{O}(V_+,q) \times \mathbf{O}(V_-,q) \right) $$
is the maximal compact subgroup of $\mathbf{G}(\R)$ containing $\mathbf{T}(\R)$.
For each $i$, choose an isomorphism $e_i : \mathbf{SO}(P_i,q)_{\C} \rightarrow \mathbf{G}_m$ arbitrarily.
For $\dim V$ even, the roots $e_1 - e_2, \dots, e_{n-1}-e_n, e_{n-1}+e_n$ are all noncompact if and only if
$$ \{ I_+, I_- \} = \left\{ \{1,3,5,\dots\}, \{2,4,\dots\} \right\} $$
and modulo conjugation by $W_c = N(K, \mathbf{T}(\R))/\mathbf{T}(\R)$ there are two Borel subgroups $\mathbf{B} \supset \mathbf{T}_{\C}$ whose simple roots are all noncompact.
For $\dim V$ odd the roots $e_1 - e_2, \dots, e_{n-1}-e_n, e_n$ are all noncompact if and only if
$$ I_- = \{n, n-2, n-4, \dots\} \text{ and } I_+ = \{n-1, n-3, \dots \} $$
and there is just one conjugacy class of such Borel subgroups.
In both cases
$$ \ker \left( H^1(\R, \mathbf{T}) \rightarrow H^1(\R, \mathbf{G}) \right) $$
is isomorphic to the set of $(\epsilon_i)_{1 \leq i \leq n}$ where $\epsilon_i \in \{\pm 1 \}$ and
$$ \mathrm{card} \{ i \in I_+\ |\ \epsilon_i = -1 \} = \mathrm{card} \{ i \in I_-\ |\ \epsilon_i = -1 \}. $$

For the symplectic group $\mathbf{G} = \mathbf{Sp}(V, a)$ (where $a$ is a non-degenerate alternate form) $H^1(\R, \mathbf{G})$ is trivial, so that the set of $\langle \pi, \cdot \rangle$ ($\pi \in \Pi_{\varphi}$) is simply the whole group $S_{\varphi}^{\wedge}$.
However, for the non-tempered case and for the application to Siegel modular forms it will be necessary to have an explicit description of the pairs $(\mathbf{B}, \mathbf{T})$ as for the special orthogonal groups.
There exists $J \in \mathbf{G}(\R)$ such that $J^2 = - \mathrm{Id}$ and for any $v \in V \smallsetminus \{0\}$, $a(Jv,v)>0$.
Then $J$ is a complex structure on $V$ and
$$ h(v_1,v_2) := a(Jv_1, v_2) + i a(v_1,v_2) $$
defines a positive definite hermitian form $h$ on $V$.
Choose an orthogonal (for $h$) decomposition $ V = \bigoplus_{i=1}^n P_i $ where each $P_i$ is a complex line, then we can define $\mathbf{T}$ as the stabiliser of this decomposition.
The maximal compact subgroup of $\mathbf{G}(\R)$ containing $\mathbf{T}(\R)$ is $K = \mathbf{U}(V, h)(\R)$, and $W_c \simeq S_n$.
Thanks to the complex structure there are canonical isomorphisms $e_i : \mathbf{U}(P_i, h) \rightarrow \mathbf{U}_1$ (for $i \in \{1, \dots, n\}$).
Modulo conjugation by $W_c$, the two Borel subgroups containing $\mathbf{T}_{\C}$ and having non-compact simple roots correspond to the sets of simple roots
$$ \{e_1 + e_2, -e_2-e_3, \dots, (-1)^n (e_{n-1}+e_n), (-1)^{n+1}2e_n \},$$
$$\{-e_1 - e_2, e_2+e_3, \dots, (-1)^{n-1}(e_{n-1}+e_n), (-1)^n2e_n \}. $$

\subsubsection{Adams-Johnson packets and Euler-Poincaré characteristics}
\label{sectionAJ}

Let us now consider the general case, which as we observed above is necessary only when the dominant weight $\lambda$ is not regular.
For a quasisplit special orthogonal or symplectic group $\mathbf{G}$ and an Arthur parameter $\psi : W_{\R} \times \mathrm{SL}_2(\C) \rightarrow {}^L \mathbf{G}$ having infinitesimal character $\lambda + \rho$, we would like to describe explicitly the multiset $\Pi_{\psi}$ along with the map $\Pi_{\psi} \rightarrow S_{\psi}^{\wedge}$.
We would also like to compute the Euler-Poincaré characteristic $\mathrm{EP}(\pi \otimes V_{\lambda}^*)$ for any $\pi \in \Pi_{\psi}$.
Unfortunately it does not seem possible to achieve these tasks directly from Arthur's characterisation \cite{Arthur}[Theorem 2.2.1].
We will review Adams and Johnson's construction of packets $\Pi_{\psi}^{\mathrm{AJ}}$ using Arthur's formulation, which will lead us naturally to Assumption \ref{assumAJweak} relating Arthur's $\Pi_{\psi}$ with $\Pi_{\psi}^{\mathrm{AJ}}$.
This review was done in \cite{ArthurUnip}, \cite{KottAA} and \cite{ChRe} but we need to recall Adams and Johnson's results precisely in order to compute Euler-Poincaré characteristics.
Moreover we will uncover a minor problem in \cite{ArthurUnip}[§5].
Finally, \cite{AdJo} was written before Shahidi's conjecture \cite{Shahidi}[Conjecture 9.4] was formulated, and thus we need to adress the issue of normalization of transfer factors by Whittaker datum.
This is necessary to get a precise and explicit formulation of \cite{AdJo} in our setting, which is a prerequisite for writing an algorithm.

As in the previous section $\mathbf{G}$ could be any reductive algebraic group over $\R$ such that $\mathbf{G}(\R)$ has essentially discrete series.
To simplify notations we assume that $\mathbf{G}$ is semisimple.
To begin with, we consider general Arthur parameters $ \psi : W_{\R} \times \SL_2(\C) \rightarrow {}^L \mathbf{G}$, i.e.\ continuous morphisms such that
\begin{itemize}
\item composing with ${}^L \mathbf{G} \rightarrow W_{\R}$, we get $\mathrm{Id}_{W_{\R}}$,
\item $\psi|_{W_{\C}}$ is semisimple and bounded,
\item $\psi|_{\SL_2(\C)}$ is algebraic.
\end{itemize}
As before we fix a $\Gal(\C/\R)$-invariant splitting $(\mathcal{B}, \mathcal{T}, ( \mathcal{X}_{\alpha} )_{\alpha \in \Delta})$ in $\widehat{\mathbf{G}}$.
Assume that $\psi$ is \emph{pure}, i.e.\ the restriction of $\psi$ to $\R_{>0} \subset W_{\C}$ is trivial.
Otherwise $\psi$ would factor through a Levi subgroup of ${}^L \mathbf{G}$.
After conjugating by an element of $\widehat{\mathbf{G}}$ we have a $\mathcal{B}$-dominant $\tau_0 \in \frac{1}{2} X_*(\mathcal{T})$ such that for any $z \in W_{\C}$, $\psi(z) = (2 \tau_0)(z/|z|)$.
The set of roots $\alpha \in R(\mathcal{T}, \widehat{\mathbf{G}})$ such that $\langle \tau_0, \alpha \rangle \geq 0$ defines a parabolic subgroup $\mathcal{Q} = \mathcal{L} \mathcal{U}$ of $\widehat{\mathbf{G}}$ with Levi $\mathcal{L} = \mathrm{Cent}(\psi(W_{\C}), \widehat{\mathbf{G}})$ and $\psi(\SL_2(\C)) \subset \mathcal{L}_{\mathrm{der}}$.
After conjugating we can assume that
$$ z \in \C^{\times} \mapsto \psi\left( \begin{pmatrix} z & 0 \\ 0 & z^{-1} \end{pmatrix} \in \SL_2(\C) \right) $$
takes values in $\mathcal{T} \cap \mathcal{L}_{\mathrm{der}}$ and is dominant with respect to $\mathcal{B} \cap \mathcal{L}_{\mathrm{der}}$.
Let us restrict our attention to parameters $\psi$ such that $\psi|_{\SL_2(\C)} : \SL_2(\C) \rightarrow \mathcal{L}_{\mathrm{der}}$ is the principal morphism.
After conjugating we can assume that
$$\mathrm{d} \left(\psi|_{\SL_2(\C)}\right) \left( \begin{pmatrix} 0 & 1 \\ 0 & 0 \end{pmatrix} \in \mathfrak{sl}_2 \right) = \sum_{\alpha \in \Delta_{\mathcal{L}}} \mathcal{X}_{\alpha}.$$

We claim that $\psi(j) \in \widehat{\mathbf{G}} \rtimes \{j\}$ is now determined modulo left multiplication by $Z(\mathcal{L})$.
Let $n : W(\widehat{\mathbf{G}}, \mathcal{T}) \rtimes W_{\R} \rightarrow N({}^L \mathbf{G}, \mathcal{T}) =  N(\widehat{\mathbf{G}}, \mathcal{T}) \rtimes W_{\R}$ be the set-theoretic section defined in \cite{LanShe}[§2.1].
Let $w_0 \in W(\widehat{\mathbf{G}}, \mathcal{T})$ be the longest element in the Weyl group (with respect to $\mathcal{B}$).
Since $\mathbf{G}$ has an anisotropic maximal torus, conjugation by (any representative of) $w_0 \rtimes j$ acts by $t \mapsto t^{-1}$ on $\mathcal{T}$.
Let $w_1$ be the longest element of the Weyl group $W(\mathcal{L}, \mathcal{T})$.
Then $w_1w_0 \rtimes j$ preserves $\Delta_{\mathcal{L}}$ and acts by $t \mapsto t^{-1}$ on $Z(\mathcal{L})$.
By \cite{Springer}[Proposition 9.3.5] $n( w_1w_0 \rtimes j) = n(w_1w_0) \rtimes j$ preserves the splitting $(\mathcal{X}_{\alpha})_{\alpha \in \Delta_{\mathcal{L}}}$, and thus commutes with $\psi(\SL_2(\C))$.
The following lemma relates $\psi(j)$ and $n(w_1w_0 \rtimes j)$.

\begin{lemm}
There is a unique element $a \in Z(\mathcal{L}) \backslash \left( \widehat{\mathbf{G}} \rtimes \{j\} \right)$ commuting with $\psi(\SL_2(\C))$ and such that for any $z \in W_{\C}$, $a \psi(z) a^{-1} = \psi(z^{-1})$.
\end{lemm}
\begin{proof}
If $a$ and $b$ are two such elements, $ab^{-1} \in \widehat{\mathbf{G}}$ commutes with $\psi(W_{\C})$, thus $ab^{-1} \in \mathcal{L}$.
Furthermore $ab$ commutes with $\psi(\SL_2(\C))$, hence $ab^{-1} \in Z(\mathcal{L})$.
\end{proof}
Since $n(w_1w_0 \rtimes j)$ and $\psi(j)$ satisfy these two conditions, they coincide modulo $Z(\mathcal{L})$.
In particular conjugation by $\psi(j)$ acts by $t \mapsto t^{-1}$ on $Z(\mathcal{L})$, and thus the group
$$C_{\psi} := \mathrm{Cent}(\psi, \widehat{\mathbf{G}}) = \{t \in Z(\mathcal{L})\ |\ t^2 = 1 \}$$
is finite, and so is $S_{\psi} := C_{\psi}/Z(\widehat{\mathbf{G}})^{\Gal(\C/\R)}$.
In addition, $(2 \tau_0)(-1) = \psi(j)^2 = n(w_1 w_0 \rtimes j)^2$ only depends on $\mathcal{L}$.
By \cite{LanShe}[Lemma 2.1.A], $n(w_1 w_0 \rtimes j)^2 = \prod_{\alpha \in R_{\mathcal{Q}}} \alpha^{\vee}(-1)$ where $R_{\mathcal{Q}}$ is the set of roots of $\mathcal{T}$ occurring in the unipotent radical $\mathcal{U}$ of $\mathcal{Q}$.
Thus
$$ \tau_0 \in X_*(Z(\mathcal{L})^0) + \frac{1}{2}\sum_{\alpha \in R_{\mathcal{Q}}} \alpha^{\vee}. $$
Conversely, using the element $n(w_1w_0 \rtimes j)$ we see that for any standard parabolic subgroup $\mathcal{Q} =\mathcal{L} \mathcal{U} \supset \mathcal{B}$ of $\widehat{\mathbf{G}}$ and any strictly dominant (for $R_{\mathcal{Q}}$) $ \tau_0 \in X_*(Z(\mathcal{L})^0) + \frac{1}{2}\sum_{\alpha \in R_{\mathcal{Q}}} \alpha^{\vee}$, there is at least one Arthur parameter mapping $z \in W_{\C}$ to $(2 \tau_0)(z/|z|)$ and $\begin{pmatrix} 0 & 1 \\ 0 & 0\end{pmatrix} \in \mathfrak{sl}_2$ to $\sum_{\alpha \in \Delta_{\mathcal{L}}} \mathcal{X}_{\alpha}$.
Finally, for any $u \in Z(\mathcal{L})$, we can form another Arthur parameter $\psi'$ by imposing $\psi'|_{W_{\C} \times \SL_2(\C)} = \psi|_{W_{\C} \times \SL_2(\C)}$ and $\psi'(j) = u \psi(j)$.
It follows that the set of conjugacy classes of Arthur parameters $\psi'$ such that $\psi'|_{W_{\C} \times \SL_2(\C)}$ is conjugated to $\psi|_{W_{\C} \times \SL_2(\C)}$ is a torsor under
$$Z(\mathcal{L})/\{t^2\ |\ t \in Z(\mathcal{L})\} = H^1(\Gal(\C/\R), Z(\mathcal{L})) \text{ where } \sigma \text{ acts by } w_1w_0 \rtimes j \text{ on } Z(\mathcal{L}).$$

Recall the norm $|| \cdot || : W_{\R} \rightarrow \R_{>0}$ which maps $j$ to $1$ and $z \in W_{\C}$ to $z \bar{z}$, which is used to define the morphism $W_{\R} \rightarrow W_{\R} \times \SL_2(\C)$ mapping $w$ to
$$\left(w, \begin{pmatrix} ||w||^{1/2} & 0 \\ 0 & ||w||^{-1/2} \end{pmatrix} \right).$$
Composing $\psi$ with this morphism we get a Langlands parameter $\varphi_{\psi} : W_{\R} \rightarrow {}^L \mathbf{G}$ which is not tempered in general.
For $z \in W_{\C}$, $\varphi_{\psi}(z) = (\tau - \tau')(z/|z|)(\tau+\tau')(|z|)$ (formally $\tau(z) \tau'(\bar{z})$) where
$$ \tau = \tau_0 + \frac{1}{2} \sum_{\alpha \in R_{\mathcal{B} \cap \mathcal{L}}} \alpha^{\vee} \quad \text{ and } \quad \tau' = -\tau_0 + \frac{1}{2} \sum_{\alpha \in R_{\mathcal{B} \cap \mathcal{L}}} \alpha^{\vee}. $$
Then $\tau \in \frac{1}{2} \sum_{\alpha \in R_{\mathcal{B}}} \alpha^{\vee} + X_*(\mathcal{T})$ and the following are equivalent:
\begin{enumerate}
\item $\tau$ is regular,
\item $\tau - \frac{1}{2} \sum_{\alpha \in R_{\mathcal{B}}} \alpha^{\vee}$ is dominant with respect to $R_{\mathcal{B}}$,
\item $\tau_0 - \frac{1}{2} \sum_{\alpha \in R_{\mathcal{Q}}} \alpha^{\vee}$ is dominant with respect to $R_{\mathcal{Q}}$.
\end{enumerate}
In fact for any pure Arthur parameter $\psi$, without assuming a priori that $\psi|_{\SL_2(\C)} \rightarrow \mathcal{L}$ is principal, if the holomorphic part $\tau$ of $\varphi_{\psi}|_{W_{\C}}$ is regular, then $\psi|_{\SL_2(\C)} \rightarrow \mathcal{L}$ is principal.
The orbit of $\tau$ under the Weyl group is the infinitesimal character associated with $\psi$, and we have seen that it is the infinitesimal character of any representation in the packet $\Pi_{\psi}$ associated with $\psi$ (Lemma \ref{lemmCharInfArthur}).
For quasisplit special orthogonal or symplectic groups we checked this (up to outer conjugacy in the even orthogonal case) in Lemma \ref{lemmCharInfArthur}.

From now on we also assume that the infinitesimal character $\tau$ of $\psi$ is regular.
Note that $\tau$ is then the infinitesimal character of the restriction to $\mathbf{G}(\R)$ of the irreducible algebraic representation $V_{\lambda}$ of $\mathbf{G}_{\C}$, where $\tau = \lambda + \rho$.
Let us describe the set of representations $\Pi_{\psi}^{\mathrm{AJ}}$ that Adams and Johnson associate with $\psi$ as well as the pairing $\Pi_{\psi} \rightarrow S_{\psi}^{\wedge}$.
To be honest Adams and Johnson do not consider parameters $\psi$, they only work with representations, but \cite{ArthurUnip}[§5] interpreted their construction in terms of parameters.
We will only add details concerning Whittaker normalisation.
As in the tempered case we begin by considering pairs $(\mathbf{B}, \mathbf{T})$ where $\mathbf{T}$ is an anisotropic maximal torus of $\mathbf{G}$ and $\mathbf{B}$ a Borel subgroup of $\mathbf{G}_{\C}$ containing $\mathbf{T}_{\C}$.
We have a canonical isomorphism between the based root data
$$(X^*(\mathbf{T}_{\C}), \Delta_{\mathbf{B}}, X_*(\mathbf{T}_{\C}), \Delta_{\mathbf{B}}^{\vee}) \quad\text{ and }\quad (X_*(\mathcal{T}), \Delta_{\mathcal{B}}^{\vee}, X^*(\mathcal{T}), \Delta_{\mathcal{B}})$$
and we can associate with $(\mathcal{Q}, \mathcal{L})$ a parabolic subgroup $\mathbf{Q} \supset \mathbf{B}$ of $\mathbf{G}_{\C}$ and a Levi subgroup $\mathbf{L}_{\C} \supset \mathbf{T}_{\C}$ of $\mathbf{G}_{\C}$.
As the notation suggests $\mathbf{L}_{\C}$ is defined over $\R$ (for any root $\alpha$ of $\mathbf{T}_{\C}$ in $\mathbf{G}_{\C}$, $\sigma(\alpha)=-\alpha$), and we denote this real subgroup of $\mathbf{G}$ by $\mathbf{L}$.
Consider the set $\Sigma_{\mathcal{Q}}$ of conjugacy classes of pairs $(\mathbf{Q}, \mathbf{L})$ ($\mathbf{Q}$ a parabolic subgroup of $\mathbf{G}_{\C}$ and $\mathbf{L}$ a real subgroup of $\mathbf{G}$ such that $\mathbf{L}_{\C}$ is a Levi subgroup of $\mathbf{Q}$) obtained this way.
The finite set $\Sigma_{\mathcal{B}}$ of conjugacy classes of pairs $(\mathbf{B}, \mathbf{T})$ surjects to $\Sigma_{\mathcal{Q}}$.
If we fix a base point $(\mathbf{B}, \mathbf{T})$, we have seen that $\Sigma_{\mathcal{B}}$ is identified with $W_c \backslash W$.
This base point allows to identify $\Sigma_{\mathcal{Q}}$ with $W_c \backslash W / W_{\mathbf{L}}$ where $W_{\mathbf{L}} = W(\mathbf{L}(\C), \mathbf{T}(\C))$, and
$$W_c \backslash W / W_{\mathbf{L}} \simeq \ker \left( H^1(\R, \mathbf{L}) \rightarrow H^1(\R, \mathbf{G}) \right).$$
For any $\mathrm{cl}(\mathbf{Q}, \mathbf{L}) \in \Sigma_{\mathcal{Q}}$ there is a canonical isomorphism $\widehat{\mathbf{L}} \simeq \mathcal{L}$ identifying the splittings.
Given another $\mathrm{cl}(\mathbf{Q}', \mathbf{L}') \in \Sigma_{\mathcal{Q}}$, there is a unique $g \in \mathbf{G}(\C)/\mathbf{L}(\C)$ conjugating $(\mathbf{Q}, \mathbf{L})$ into $(\mathbf{Q}', \mathbf{L}')$, yielding a canonical isomorphism of L-groups ${}^L \mathbf{L} \simeq {}^L \mathbf{L}'$.
As in the tempered case we want to extend $\widehat{\mathbf{L}} \simeq \mathcal{L}$ into an embedding $\iota : {}^L \mathbf{L} \rightarrow {}^L \mathbf{G}$ as follows.
For $z \in W_{\C}$, define $\iota(z) = \prod_{\alpha \in R_{\mathcal{Q}}} \alpha^{\vee}(z/|z|) \rtimes z$.
Define $\iota(j) = n(w_1w_0 \rtimes j)$.
We have computed $n(w_1w_0 \rtimes j)^2 = \prod_{\alpha \in R_{\mathcal{Q}}} \alpha^{\vee}(-1)$ above and thus $\iota$ is well-defined.
Note that contrary to the tempered case, there are other choices for $\iota(j)$ even up to conjugation by $Z(\mathcal{L})$: we could replace $\iota(j)$ by $u \iota(j)$ where $u \in Z(\mathcal{L})$, and it can happen that $u$ is not a square in $Z(\mathcal{L})$.
This issue seems to have been overlooked in \cite{ArthurUnip}[§5].
We will not try to determine whether $n(w_1w_0 \rtimes j)$ is the correct choice here and we will consider this problem in a separate note, since for our present purpose this choice does not matter.

For any class $\mathrm{cl}(\mathbf{Q}, \mathbf{L}) \in \Sigma_{\mathcal{Q}}$ there is a unique Arthur parameter
$$ \psi_{\mathbf{Q}, \mathbf{L}} : W_{\R} \times \SL_2(\C) \rightarrow {}^L \mathbf{L} $$
such that up to conjugation by $\widehat{\mathbf{G}}$, $\psi = \iota \circ \psi_{\mathbf{Q}, \mathbf{L}}$.
Now $\psi_{\mathbf{Q}, \mathbf{L}}|_{\SL_2(\C)} : \SL_2(\C) \rightarrow \widehat{\mathbf{L}}$ is the principal morphism.
Thus $\psi_{\mathbf{Q}, \mathbf{L}}|_{W_{\R}}$ takes values in $Z(\widehat{\mathbf{L}}) \rtimes W_{\R}$, and the conjugacy class of $\psi_{\mathbf{Q}, \mathbf{L}}$ is determined by the resulting element of $H^1(W_{\R}, Z(\widehat{\mathbf{L}}))$, which has compact image.

Recall that for any real reductive group $\mathbf{H}$ there is a natural morphism
$$ \nu_{\mathbf{H}} : H^1(W_{\R}, Z(\widehat{\mathbf{H}})) \rightarrow \Hom_{\mathrm{cont}}(\mathbf{H}(\R), \C^{\times}) $$
which is surjective and maps cocyles with compact image to unitary characters of $\mathbf{H}(\R)$.
To define this morphism we can use the same arguments as \cite{KottEllSing}[§1].
If $\mathbf{H}$ is simply connected, then $\widehat{\mathbf{H}}$ is adjoint and $\mathbf{H}(\R)$ is connected.
More generally, if $\mathbf{H}_{\mathrm{der}}$ is simply connected then the torus $\mathbf{C} = \mathbf{H} / \mathbf{H}_{\mathrm{der}}$ is such that $Z(\widehat{\mathbf{H}}) = \widehat{\mathbf{C}}$ and
$$\mathbf{H}(\R)^{\mathrm{ab}} = \ker \left( \mathbf{C}(\R) \rightarrow H^1(\R, \mathbf{H}_{\mathrm{der}}) \right).$$
Finally if $\mathbf{H}$ is arbitrary there exists a z-extension $\mathbf{C} \hookrightarrow \widetilde{\mathbf{H}} \twoheadrightarrow \mathbf{H}$ where $\mathbf{C}$ is an induced torus and $\widetilde{\mathbf{H}}_{\mathrm{der}}$ is simply connected.
Then $H^1(\Gal(\C/\R), \mathbf{C}(\C))$ is trivial, thus $\widetilde{\mathbf{H}}(\R) \twoheadrightarrow \mathbf{H}(\R)$ and
$$ \Hom_{\mathrm{cont}}(\mathbf{H}(\R), \C^{\times}) = \ker \left( \Hom_{\mathrm{cont}}(\widetilde{\mathbf{H}}(\R), \C^{\times}) \rightarrow \Hom_{\mathrm{cont}}(\mathbf{C}(\R), \C^{\times})\right). $$
Parallelly, $\widehat{\mathbf{C}}^{W_{\R}}$ is connected so that $\widehat{\mathbf{C}}^{W_{\R}} \rightarrow H^1(W_{\R}, Z(\widehat{\mathbf{H}}))$ is trivial and thus
$$ H^1(W_{\R}, Z(\widehat{\mathbf{H}})) = \ker \left( H^1(W_{\R}, Z(\widehat{\widetilde{\mathbf{H}}})) \rightarrow H^1(W_{\R}, \widehat{\mathbf{C}}) \right). $$
As in \cite{KottEllSing}[§1] the morphism $\nu_{\mathbf{H}}$ obtained this way does not depend on the choice of a z-extension.
Note that when $\mathbf{H}$ is quasi-split, $\nu_{\mathbf{H}}$ is an isomorphism, by reduction to the case where $\mathbf{H}_{\mathrm{der}}$ is simply connected and using the fact that a maximally split maximal torus in a simply connected quasi-split group is an induced torus.
It is not injective in general, e.g.\ when $\mathbf{H}$ is the group of invertible quaternions.

Hence $\psi_{\mathbf{Q}, \mathbf{L}}$ defines a one-dimensional unitary representation $\pi^0_{\psi, \mathbf{Q}, \mathbf{L}}$ of $\mathbf{L}(\R)$, and applying cohomological induction as defined by Zuckerman, Adams and Johnson define the representation $\pi_{\psi, \mathbf{Q}, \mathbf{L}} = R^i_{\mathfrak{q}}(\pi^0_{\psi, \mathbf{Q}, \mathbf{L}})$ of $\mathbf{G}(\R)$, where $\mathfrak{q} = \mathrm{Lie} (\mathbf{Q})$ and $i=q(\mathbf{G})-q(\mathbf{L})$.
Vogan has shown that this representation is unitary.
They define the set $\Pi_{\psi}^{\mathrm{AJ}}$ in bijection with $\Sigma_{\mathcal{Q}}$:
$$ \Pi_{\psi}^{\mathrm{AJ}} = \left\{ \pi_{\psi, \mathbf{Q}, \mathbf{L}}\ |\ \mathrm{cl}(\mathbf{Q}, \mathbf{L}) \in \Sigma_{\mathcal{Q}} \right\}. $$
The endoscopic character relations that they prove \cite{AdJo}[Theorem 2.21] allow to identify the map $\Pi_{\psi} \rightarrow S_{\psi}^{\wedge}$, as Arthur did in \cite{ArthurUnip}[§5].
Assume that $\mathbf{G}$ is quasisplit (this is probably unnecessary as in the tempered case using the constructions of \cite{Kal}), and fix a Whittaker datum for $\mathbf{G}$.
Then any $\mathrm{cl}(\mathbf{B}, \mathbf{T}) \in \Sigma_{\mathcal{B}}$ determines an element of $S_{\varphi}^{\wedge}$ (here $\varphi$ could be any discrete parameter, the group $S_{\varphi}$ is described in terms of $\mathcal{B}, \mathcal{T}$ independently).
It is easy to check that if $(\mathbf{B}, \mathbf{T})$ and $(\mathbf{B}', \mathbf{T}')$ give rise to pairs $(\mathbf{Q}, \mathbf{L})$ and $(\mathbf{Q}', \mathbf{L}')$ which are conjugated under $\mathbf{G}(\R)$, then the restrictions to $S_{\psi}$ of the characters of $S_{\varphi}$ associated with $(\mathbf{B}, \mathbf{T})$ and $(\mathbf{B}', \mathbf{T}')$ coincide.
We get a map $\Pi_{\psi}^{\mathrm{AJ}} \rightarrow S_{\psi}^{\wedge}$ which is not injective in general.

Adams and Johnson (\cite{AdJo}[Theorem 8.2], reformulating the main result of \cite{Johnson}) give a resolution of $\pi_{\psi, \mathbf{Q}, \mathbf{L}}$ by direct sums of standard modules 
\begin{equation} \label{resoJohnson} 0 \rightarrow \pi_{\psi, \mathbf{Q}, \mathbf{L}} \rightarrow X^{q(\mathbf{L})} \rightarrow \dots \rightarrow X^0 \rightarrow 0. \end{equation}
Recall that a standard module is a parabolic induction of an essentially tempered representation of a Levi subgroup of $\mathbf{G}$, with a certain positivity condition on its central character.
Johnson's convention is opposite to that of Langlands, so that $\pi_{\psi, \mathbf{Q}, \mathbf{L}}$ \emph{embeds} in a standard module.
Apart from its length, the only two properties of this resolution that we need are
\begin{enumerate}
\item $X^0$ is the direct sum of the discrete series representations of $\mathbf{G}(\R)$ having infinitesimal character $\tau$ and corresponding to the $\mathrm{cl}(\mathbf{B}, \mathbf{T}) \in \Sigma_{\mathcal{B}}$ mapping to $\mathrm{cl}(\mathbf{Q}, \mathbf{L}) \in \Sigma_{\mathcal{Q}}$,
\item for any $i>0$, $X^i$ is a direct sum of standard modules induced from \emph{proper} parabolic subgroups of $\mathbf{G}$, therefore $\mathrm{EP}(X^i \otimes V_{\lambda}^*) = 0$.
\end{enumerate}
Thus we have the simple formula
$$ \mathrm{EP}(\pi_{\psi, \mathbf{Q}, \mathbf{L}} \otimes V_{\lambda}^*) = (-1)^{q(\mathbf{G})-q(\mathbf{L})} \mathrm{card} \left( \text{fiber of } \mathrm{cl}(\mathbf{Q}, \mathbf{L}) \text{ by } \Sigma_{\mathcal{B}} \rightarrow \Sigma_{\mathcal{Q}} \right). $$
Note that $\pi_{\psi, \mathbf{Q}, \mathbf{L}}$ is a discrete series representation if and only if $\mathbf{L}$ is anisotropic.

Let us be more precise about the endoscopic character relations afforded by Adams-Johnson representations, since Shahidi's conjecture was only formulated after both \cite{AdJo} and \cite{ArthurUnip}.
Let $s_{\psi}$ be the image by $\psi$ of $-1 \in \mathrm{SL}_2(\C)$, which we will see as an element of $S_{\psi}$.
Arthur and Kottwitz have shown that for $\mathrm{cl}(\mathbf{Q}, \mathbf{L}), \mathrm{cl}(\mathbf{Q}', \mathbf{L}') \in \Sigma_{\mathcal{Q}}$, we have $\langle s_{\psi}, \pi_{\psi, \mathbf{Q}, \mathbf{L}} \rangle = (-1)^{q(\mathbf{L})-q(\mathbf{L}')} \langle s_{\psi}, \pi_{\psi, \mathbf{Q}', \mathbf{L}'} \rangle$.
Let $(\mathbf{B}_0, \mathbf{T}_0)$ be a pair in $\mathbf{G}$ corresponding to the base point (i.e.\ the generic representation for our fixed Whittaker datum) for any discrete L-packet.
It determines a pair $(\mathbf{Q}_0, \mathbf{L}_0)$ such that $\mathrm{cl}(\mathbf{Q}_0, \mathbf{L}_0) \in \Sigma_{\mathcal{Q}}$.
The simple roots of $\mathbf{B}_0$ are all non-compact and thus the same holds for the Borel subgroup $\mathbf{B}_0 \cap (\mathbf{L}_0)_{\C}$ of $(\mathbf{L}_0)_{\C}$.
By Lemma \ref{lemmRoppqs} the group $\mathbf{L}_0$ is quasisplit.
Thus for any $\mathrm{cl}(\mathbf{Q}, \mathbf{L}) \in \Sigma_{\mathcal{Q}}$ we have $\langle s_{\psi}, \pi_{\psi, \mathbf{Q}, \mathbf{L}} \rangle = (-1)^{q(\mathbf{L}_0) - q(\mathbf{L})}$.
Note that if $(\mathbf{B}_1, \mathbf{T}_1)$ corresponds to the generic element in tempered L-packets for \emph{another} Whittaker datum, the pair $(\mathbf{L}_1, \mathbf{Q}_1)$ that it determines also has the property that $\mathbf{L}_1$ is quasisplit.
Since $\mathbf{L}_0$ and $\mathbf{L}_1$ are inner forms of each other, they are isomorphic and $q(\mathbf{L}_0) = q(\mathbf{L}_1)$.
This shows that the map
$$ f(g)dg \mapsto \sum_{\pi \in \Pi_{\psi}^{\mathrm{AJ}}} \langle s_{\psi}, \pi \rangle \mathrm{Tr} \left( \pi(f(g)dg)\right), $$
defined on smooth compactly supported distributions on $\mathbf{G}(\R)$, is canonical: it does not depend on the choice of a Whittaker datum for the quasisplit group $\mathbf{G}$.
By \cite{AdJo}[Theorem 2.13] it is \emph{stable}, i.e.\ it vanishes if all the stable orbital integrals of $f(g)dg$ vanish.
Consider an arbitrary element $x \in S_{\psi}$.
It determines an endoscopic group $\mathbf{H}$ of $\mathbf{G}$ and an Arthur parameter $\psi_{\mathbf{H}} : W_{\R} \times \mathrm{SL}_2(\C) \rightarrow {}^L \mathbf{H}$ whose infinitesimal character is regular.
Thanks to the choice of a Whittaker datum we have a well-defined \emph{transfer map} $f(g)dg \mapsto f^{\mathbf{H}}(h)dh$ from smooth compactly supported distributions on $\mathbf{G}(\R)$ to smooth compactly supported distributions on $\mathbf{H}(\R)$.
Adams and Johnson have proved \cite{AdJo}[Theorem 2.21] that there is some $t \in \C^{\times}$ such that
\begin{equation} \label{endorelAJ}
\sum_{\pi \in \Pi_{\psi}^{\mathrm{AJ}}} \langle s_{\psi}x, \pi \rangle \mathrm{Tr} \left( \pi(f(g)dg)\right) = t \sum_{\pi \in \Pi_{\psi_{\mathbf{H}}}^{\mathrm{AJ}}} \langle s_{\psi}, \pi \rangle \mathrm{Tr} \left( \pi(f^{\mathbf{H}}(h)dh)\right)
\end{equation}
for any smooth compactly supported distribution $f(g)dg$ on $\mathbf{G}(\R)$.
We check that $t=1$.
Let $\varphi : W_{\R} \rightarrow {}^L \mathbf{G}$ be the discrete Langlands parameter having infinitesimal character $\tau$.
Conjugating if necessary, we can assume that the holomorphic parts of $\varphi|_{W_{\C}}$ and $\varphi_{\psi}|_{W_{\C}}$ are equal and not just conjugated.
In this way we see $S_{\psi}$ as a subgroup of $S_{\varphi}$.
We restrict to distributions $f(g)dg$ whose support is contained in the set of semisimple regular elliptic elements of $\mathbf{G}(\R)$.
In that case by Johnson's resolution \ref{resoJohnson}
\begin{eqnarray*}
\sum_{\pi \in \Pi_{\psi}^{\mathrm{AJ}}} \langle s_{\psi}x, \pi \rangle \mathrm{Tr} \left( \pi(f(g)dg)\right) & = & (-1)^{q(\mathbf{L}_0)} \sum_{\pi \in \Pi_{\varphi}} \langle x, \pi \rangle \mathrm{Tr} \left( \pi(f(g)dg)\right) \\
& = & (-1)^{q(\mathbf{L}_0)} \sum_{\pi \in \Pi_{\varphi_{\mathbf{H}}}} \mathrm{Tr} \left( \pi(f^{\mathbf{H}}(h)dh)\right)
\end{eqnarray*}
where the second equality is the endoscopic character relation for $(\varphi, x)$.
Let $(\mathbf{B}_0^{\mathbf{H}}, \mathbf{T}_0^{\mathbf{H}})$ be a pair for $\mathbf{H}$ such that the simple roots of $\mathbf{B}_0^{\mathbf{H}}$ are all non-compact.
Then the pair $(\mathbf{Q}_0^{\mathbf{H}}, \mathbf{L}_0^{\mathbf{H}})$ that it determines is such that $\mathbf{L}_0^{\mathbf{H}}$ is quasisplit and has same Langlands dual group as $\mathbf{L}_0$, thus $\mathbf{L}_0^{\mathbf{H}} \simeq \mathbf{L}_0$.
In particular $q(\mathbf{L}_0^{\mathbf{H}}) = q(\mathbf{L}_0)$ and
$$ (-1)^{q(\mathbf{L}_0)} \sum_{\pi \in \Pi_{\varphi_{\mathbf{H}}}} \mathrm{Tr} \left( \pi(f^{\mathbf{H}}(h)dh) \right)= \sum_{\pi \in \Pi^{\mathrm{AJ}}_{\psi_{\mathbf{H}}}} \langle s_{\psi}, \pi \rangle \mathrm{Tr} \left( \pi(f^{\mathbf{H}}(h)dh)\right). $$
Therefore the endoscopic character relation \ref{endorelAJ} holds with $t=1$ for such distributions $f(g)dg$.
By choosing $f(g)dg$ positive with small support around a well-chosen semisimple regular elliptic element we can ensure that both sides do not vanish, so that $t=1$.

This concludes the precise determination of the map $\pi \mapsto \langle \cdot, \pi \rangle$, normalised using Whittaker datum as in the tempered case.
Note that this normalised version of \cite{AdJo}[Theorem 2.21] is completely analogous to \cite{Arthur}[Theorem 2.2.1(b)].
We are led to make the following assumption.

\begin{assu} \label{assumAJ}
Let $\mathbf{G}$ be a quasisplit special orthogonal or symplectic group over $\R$ having discrete series.
Fix a Whittaker datum for $\mathbf{G}$.
Let $\psi$ be an Arthur parameter for $\mathbf{G}$ with regular infinitesimal character $\tau = \lambda + \rho$.
Then for any $\chi \in S_{\psi}^{\wedge}$,
\begin{equation} \label{equassumAJ} \bigoplus_{\substack{\pi \in \Pi_{\psi}^{\mathrm{AJ}} \\ \langle \cdot, \pi \rangle = \chi}} \pi \simeq \bigoplus_{\substack{\pi \in \Pi_{\psi} \\ \langle \cdot, \pi \rangle = \chi}} \pi .\end{equation}
Note that in the even orthogonal case, this only assumes an isomorphism of $\mathcal{H}'(\mathbf{G}(\R))$-modules.
\end{assu}

To compute Euler-Poincaré characteristics we only need the character of the direct sum appearing in Assumption \ref{assumAJ} on an anisotropic maximal torus.
This follows from the fact that the standard modules form a basis of the Grothendieck group of finite length $(\mathfrak{g}, K)$-modules.
Using also the fact that Arthur and Adams-Johnson packets satisfy the same endoscopic relations, we can formulate a weaker assumption which is enough to compute the Euler-Poincaré characteristic of the right hand side of \ref{equassumAJ} for any $\chi \in S_{\psi}^{\wedge}$.
\begin{assu} \label{assumAJweak}
Let $\mathbf{G}$ be a quasisplit special orthogonal or symplectic group over $\R$ having discrete series.
Let $\psi$ be an Arthur parameter for $\mathbf{G}$ with regular infinitesimal character $\tau = \lambda + \rho$, and let $\mathbf{T}$ be a maximal torus of $\mathbf{G}$ which is anisotropic.
Let $\mathbf{L}_0$ denote the quasisplit reductive group defined in the discussion above.
If $\mathbf{G}$ is symplectic or odd orthogonal, the assumption is that for any $\gamma \in \mathbf{T}_{\mathrm{reg}}(\R)$,
$$ \sum_{\pi \in \Pi_{\psi}} \langle s_{\psi}, \pi \rangle \Theta_{\pi}(\gamma) = (-1)^{q(\mathbf{G})-q(\mathbf{L}_0)} \mathrm{Tr}( \gamma | V_{\lambda}).$$
In the even orthogonal case, this identity takes the following meaning.
Let $\gamma \in \mathbf{T}_{\mathrm{reg}}(\R)$ and consider a $\gamma' \in \mathbf{G}(\R)$ outer conjugated to $\gamma$.
For $\pi$ in $\Pi_{\psi}$, which is only an $\mathrm{Out}(\mathbf{G})$-orbit of representations, we still denote by $\pi$ any element of this orbit.
The assumption is
$$ \sum_{\pi \in \Pi_{\psi}} \langle s_{\psi}, \pi \rangle \left(\Theta_{\pi}(\gamma) + \Theta_{\pi}(\gamma') \right) = (-1)^{q(\mathbf{G})-q(\mathbf{L}_0)} \left(\mathrm{Tr}( \gamma | V_{\lambda}) + \mathrm{Tr}( \gamma' | V_{\lambda}) \right).$$
Of course it does not depend on the choice made in each orbit.
\end{assu}

Thus under this assumption we have an algorithm to compute inductively the cardinality of each $\Psi(\mathbf{G})^{\mathrm{unr}, \lambda}_{\mathrm{sim}}$.

\begin{rema}
For this algorithm it is not necessary to enumerate the sets
$$ W_c \backslash W / W_{\mathbf{L}} \simeq \ker \left( H^1(\R, \mathbf{L}) \rightarrow H^1(\R, \mathbf{G}) \right) $$
parametrizing the elements of each $\Pi_{\psi}$.
It is enough to compute, for each discrete series $\pi$ represented by a collection of signs as in the previous section, the restriction of $\langle \cdot, \pi \rangle$ to $S_{\psi}$ and the sign $(-1)^{q(\mathbf{L})}$.
\end{rema}

See the tables in section \ref{tablestwGL} for some values for $\mathrm{card}\left( \Psi(\mathbf{G})^{\mathrm{unr}, \lambda}_{\mathrm{sim}} \right)$ in low weight $\lambda$ ordered lexicographically.

\section{Application to vector-valued Siegel modular forms}

Let us give a classical application of the previous results, to the computation of dimensions of spaces $S_r(\Gamma_n)$ of vector-valued Siegel cusp forms in genus $n \geq 1$, weight $r$ and level one.
It is certainly well-known that, under a natural assumption on the weight $r$, this dimension is equal to the multiplicity in $L^2_{\mathrm{disc}}(\mathbf{PGSp}_{2n}(\Q) \backslash \mathbf{PGSp}_{2n}(\A) / \mathbf{PGSp}_{2n}(\widehat{\Z}))$ of the holomorphic discrete series representation corresponding to $r$.
Although \cite{AsgariSchmidt} contains ``half'' of the argument, we could not find a complete reference for the full statement.
To set our mind at rest we give details for the other half.
We begin with a review of holomorphic discrete series.
We do so even though it is redundant with \cite{Knapp} and \cite{AsgariSchmidt}, in order to give precise references, to set up notation and to identify the holomorphic discrete series in Shelstad's parametrisation.

Note that it is rather artificial to restrict our attention to symplectic groups.
For any $n \geq 3$ such that $n \neq 2 \mod 4$, the split group $\mathbf{G} = \mathbf{SO}_n$ has an inner form $\mathbf{H}$ which is split at all the finite places of $\Q$ and such that
\begin{itemize}
\item if $n = -1,0,1 \mod 8$, $\mathbf{H}(\R)$ is compact,
\item if $n = 3,4,5 \mod 8$, $\mathbf{H}(\R) \simeq \mathrm{SO}(n-2,2)$.
\end{itemize}
In the second case $\mathbf{H}(\R)$ has holomorphic discrete series which can be realised on a hermitian symmetric space of complex dimension $n-2$.
In the first case $\mathbf{H}(\R)$ also has holomorphic discrete series which can be realised on a zero-dimensional hermitian symmetric space.

\subsection{Bounded symmetric domains of symplectic type and holomorphic discrete series}

Let us recall Harish-Chandra's point of view on bounded symmetric domains and his construction of holomorphic discrete series (see \cite{BruhatHC}, \cite{HCRepIV}, \cite{HCRepV}, \cite{HCRepVI}) in the case of symplectic groups.
Let $n \geq 1$ and $\mathbf{G} = \mathbf{Sp}_{2n}$, over $\R$ in this section, and denote $G = \mathbf{G}(\R)$, $\mathfrak{g}_0 = \mathrm{Lie}(G)$ and $\mathfrak{g} = \C \otimes_{\R} \mathfrak{g}_0$.
Then $\mathbf{G}$ is the stabiliser of a non-degenerate alternate form $a$ on a $2n$-dimensional real vector space $V$.
As before choose $J \in G$ such that $J^2 = -1$ and for any $v \in V \smallsetminus \{0\}$, $a(Jv,v)>0$, which endows $V$ with a complex structure and realises $a$ as the imaginary part of the positive definite hermitian form $h$ defined by
$$ h(v_1,v_2) = a(Jv_1,v_2)+ia(v_1,v_2). $$
Then $\mathbf{K} = \mathbf{U}(V,h)$ is a reductive subgroup of $\mathbf{G}$, and $K = \mathbf{K}(\R)$ is a maximal compact subgroup of $G$.
Note that both $G$ and $K$ are connected.
The center $\mathbf{Z}_{\mathbf{K}}$ of $\mathbf{K}$ is one-dimensional and anisotropic, and the complex structure $J$ yields a canonical isomorphism $\mathbf{Z}_{\mathbf{K}} \simeq \mathbf{U}_1$.
Let $\mathfrak{u}_+$ (resp.\ $\mathfrak{u}_-$) be the subspace of $\mathfrak{g}$ such that the adjoint action of $z \in \mathbf{Z}_{\mathbf{K}}(\R)$ on $\mathfrak{u}_+$ (resp.\ $\mathfrak{u}_-$) is by multiplication by $z^2$ (resp.\ $z^{-2}$).
Then $\mathfrak{g} = \mathfrak{u}_+ \oplus \mathfrak{k} \oplus \mathfrak{u}_-$ and $[\mathfrak{u}_+, \mathfrak{u}_+] = [\mathfrak{u}_-, \mathfrak{u}_-] = 0$.
Moreover $\mathfrak{u}_+ \oplus \mathfrak{u}_- = \C \otimes_{\R} \mathfrak{p}_0$ where $\mathfrak{p}_0$ is the subspace of $\mathfrak{g}_0 = \mathrm{Lie}(G)$ on which $J$ acts by $-1$, i.e.\ $\mathfrak{g}_0 = \mathfrak{p}_0 \oplus \mathfrak{k}_0$ is the Cartan decomposition of $\mathfrak{g}_0$ for the Cartan involution $\theta = \mathrm{Ad}(J)$.
There are unipotent abelian subgroups $\mathbf{U}_+, \mathbf{U}_-$ of $\mathbf{G}_{\C}$ associated with $\mathfrak{u}_+, \mathfrak{u}_-$, and the subgroups $\mathbf{K}_{\C} \mathbf{U}_+$ and $\mathbf{K}_{\C} \mathbf{U}_-$ are opposite parabolic subgroups of $\mathbf{G}_{\C}$ with common Levi subgroup $\mathbf{K}_{\C}$.
It follows that the multiplication map $\mathbf{U}_+ \times \mathbf{K}_{\C} \times \mathbf{U}_- \rightarrow \mathbf{G}_{\C}$ is an open immersion.
Furthermore $G \subset \mathbf{U}_+(\C) \mathbf{K}(\C) \mathbf{U}_-(\C)$.
For $g \in G$, we can thus write $g = g_+ g_0 g_-$ where $(g_+,g_0,g_-) \in \mathbf{U}_+(\C) \times \mathbf{K}(\C) \times \mathbf{U}_-(\C)$, and Harish-Chandra showed that $g \mapsto \log(g_+)$ identifies $G/K$ with a bounded domain $D \subset \mathfrak{u}_+$.
This endows $G/K$ with a structure of complex manifold, and for any $g \in G$, left multiplication by $g$ yields a holomorphic map $G/K \rightarrow G/K$.

\begin{rema}
Let us compare this point of view with the classical one.
Let $V = \R^{2n}$ and choose the alternate form $a(\cdot, \cdot)$ having matrix $ A = \begin{pmatrix} 0 & 1_n \\ -1_n & 0 \end{pmatrix}$, that is $a(v_1,v_2) = {}^t v_1 A v_2$.
The complex structure $J$ whose matrix is also $A$ satisfies the above conditions, and the resulting maximal compact subgroup $K$ is the stabiliser of $i1_n$ for the usual action of $G$ on the Siegel upper half plane $\mathcal{H}_g = \{ \tau \in M_n(\C)\ |\ {}^t \tau = \tau \text{ and } \mathrm{Im}(\tau)>0 \}$: for $a,b,c,d \in M_n(\R)$ such that $g = \begin{pmatrix} a & b \\ c & d \end{pmatrix} \in G$ and $\tau \in \mathcal{H}_g$, $g(\tau) = (a \tau +b)(c \tau + d)^{-1} $.
We now have two identifications of $G/K$ with domains, $D$ and $\mathcal{H}_n$, and they differ by the Cayley transform $\mathcal{H}_n \rightarrow D$, $\tau \mapsto (\tau - i1_n)(\tau + i 1_n)^{-1}$.
\end{rema}

Observe that $G \mathbf{K}(\C) \mathbf{U}_-(\C) = \exp(D) \mathbf{K}(\C) \mathbf{U}_-(\C)$ is open in $\mathbf{G}(\C)$.
Consider an irreducible unitary representation $r : K \rightarrow \GL(W)$, i.e.\ an irreducible algebraic representation of $\mathbf{K}_{\C}$ endowed with a $K$-invariant positive definite hermitian form.
Harish-Chandra considered the space of holomorphic functions $f : G \mathbf{K}(\C) \mathbf{U}_-(\C) \rightarrow W$ such that
\begin{enumerate}
\item for any $(s,k,n) \in G \mathbf{K}(\C) \mathbf{U}_-(\C) \times \mathbf{K}(\C) \times \mathbf{U}_-(\C)$, $f(skn) = r(k)^{-1}f(s)$,
\item $\int_G ||f(g)||^2 dg < \infty$.
\end{enumerate}
It has an action of $G$ defined by $(g \cdot f)(s) = f(g^{-1} s)$, and we get a unitary representation of $G$ on a Hilbert space $\mathscr{H}_{r}$.
Since $G/K \simeq G \mathbf{K}(\C) \mathbf{U}_-(\C) / \mathbf{K}(\C) \mathbf{U}_-(\C)$, $\mathscr{H}_{r}$ is isomorphic to the space of $f \in L^2(G, W)$ such that
\begin{enumerate}
\item for any $(g,k) \in G \times K$, $f(gk) = r(k)^{-1}f(g)$,
\item the function $G/K \rightarrow W,\ g \mapsto r(g_0)f(g)$ is holomorphic.
\end{enumerate}
Harish-Chandra proved that $\mathscr{H}_r$ is zero or irreducible, by observing that in any closed invariant subspace, there is an $f$ such that $G/K \rightarrow W,\ g \mapsto r(g_0)f(g)$ is constant and nonzero.
Actually this a special case of \cite{HCRepV}[Lemma 12, p. 20]).
Hence when $\mathscr{H}_r \neq 0$, there is a $K$-equivariant embedding $\phi : W \rightarrow \mathscr{H}_r$, and any vector in its image is $\mathfrak{u}_+$-invariant.
More generally, using the simple action of $\mathbf{Z}_{\mathbf{K}}(\R)$ on $\mathbf{U}_+$ we see that when $\mathscr{H}_r \neq 0$ the $K$-finite vectors of $\mathscr{H}_r$ are exactly the polynomial functions on $D$.
Note that when $\mathscr{H}_r \neq 0$ it is square-integrable by definition, i.e.\ it belongs to the discrete series of $G$.

Harish-Chandra determined necessary and sufficient conditions for $\mathscr{H}_r \neq 0$.
Let $\mathbf{T}$ be a maximal torus of $\mathbf{K}$, and choose an order on the roots of $\mathbf{T}$ in $\mathbf{K}$.
This determines a unique order on the roots of $\mathbf{T}$ in $\mathbf{G}$ such that the parabolic subgroup $\mathbf{K}_{\C} \mathbf{U}_+$ is standard, i.e.\ contains the Borel subgroup $B$ of $\mathbf{G}_{\C}$ such that the positive roots are the ones occurring in $\mathbf{B}$.
To be explicit in the symplectic case, $\mathbf{T}$ is determined by a decomposition of $V$ as an orthogonal (for the hermitian form $h$) direct sum $V = V_1 \oplus \dots \oplus V_n$ where each $V_k$ is a line over $\C$.
For any $k$ we have a canonical isomorphism $e_k : \mathbf{U}(V_k,h) \simeq \mathbf{U}_1$.
We can choose the order on the roots so that the simple roots are $e_1-e_2, \dots, e_{n-1}-e_n, 2e_n$.
Note that among these simple roots, only $2e_n$ is noncompact.
Let $\lambda = m_1 e_1 + \dots + m_n e_n$ be the highest weight of $r$, so that $m_1 \geq \dots \geq m_n$.
This means that up to multiplication by a scalar there is a unique highest weight vector $v \in W \smallsetminus \{0\}$, that is such that for any $b \in \mathbf{K}(\C) \cap \mathbf{B}(\C)$, $r(b)v = \lambda(b) v$.
Let $\rho = ne_1 + \dots + e_n$ be half the sum of the positive roots of $\mathbf{T}$ in $\mathbf{G}$.
Then $\mathscr{H}_r \neq 0$ if and only if for any root $\alpha$ of $\mathbf{T}$ in $\mathbf{U}_+$, $\langle \alpha^{\vee}, \lambda + \rho \rangle < 0$ (see \cite{HCRepVI}[Lemma 29, p. 608]).
In our case this condition is equivalent to $m_1+n<0$.

Assume that $\mathscr{H}_r \neq 0$.
Note that $\phi(v)$ is a highest weight in the $\mathfrak{g}$-module $(\mathscr{H}_r)_{K \mathrm{-fin}}$, i.e.\ the Lie algebra of the unipotent radical of $\mathbf{B}$ cancels $\phi(v)$.
Since $\mathscr{H}_r$ is irreducible and unitary, $(\mathscr{H}_r)_{K \mathrm{-fin}}$ is a simple $\mathfrak{g}$-module whose isomorphism class determines that of $\mathscr{H}_r$ (see \cite{Knapp}[chapter VIII]), and thus it is the unique simple quotient of the Verma module defined by $\mathbf{B}$ and $\lambda$.
In particular, $\lambda+\rho$ is a representative for the infinitesimal character of $\mathscr{H}_r$.
One can show that $(\mathscr{H}_r)_{K \mathrm{-fin}} = U(\mathfrak{g}) \otimes_{U(\mathfrak{k} \oplus \mathfrak{u}_+)} W$, where $W$ is seen as a $\mathfrak{k} \oplus \mathfrak{u}_+$-module by letting $\mathfrak{u}_+$ act trivially.

\begin{rema}
Before Harish-Chandra realised these holomorphic discrete series concretely, in \cite{HCRepIV} he considered the simple quotient of the Verma module defined by $\lambda$ and $\mathbf{B}$, for $\lambda$ an arbitrary dominant weight for $\mathbf{K}_{\C} \cap \mathbf{B}$.
He determined a necessary condition for this $\mathfrak{g}$-module to be unitarisable \cite{HCRepIV}[Corollary 1 p.768]: for any root $\alpha$ of $\mathbf{T}$ in $\mathbf{U}_+$, $\langle \alpha^{\vee}, \lambda \rangle \leq 0$ (in our case this is equivalent to $m_1 \leq 0$).
He also determined a sufficient condition \cite{HCRepIV}[Theorem 3 p.770]: for any root $\alpha$ of $\mathbf{T}$ in $\mathbf{U}_+$, $\langle \alpha^{\vee}, \lambda + \rho \rangle \leq 0$ (in our case this is equivalent to $m_1 +n \leq 0$).
For classical groups Enright and Parthasarathy \cite{EnriPart} gave a necessary and sufficient condition for unitarisability.
In our symplectic case, this condition is
$$ -m_1 \geq \min_{1 \leq j \leq n} \left( n-i + \sum_{2 \leq j \leq i} \frac{m_1-m_j}{2} \right). $$
It would be interesting to determine whether all these unitary representations are globally relevant, i.e.\ belong to some Arthur packet.
\end{rema}

The character of $\mathscr{H}_r$ was computed explicitely in \cite{Schmid}, \cite{Martens} and \cite{Hecht}.
There exists a unique Borel subgroup $\mathbf{B}' \supset \mathbf{T}_{\C} \mathbf{U}_-$ of $\mathbf{G}_{\C}$ such that $\mathbf{B}' \cap \mathbf{K}_{\C} = \mathbf{B} \cap \mathbf{K}_{\C}$.
The order on the roots defined by $\mathbf{B}'$ is such that $\lambda + \rho$ is strictly dominant, i.e.\ for any root $\alpha$ occurring in $\mathbf{B}'$, $\langle \alpha^{\vee}, \lambda + \rho \rangle > 0$.
Let $W_c = W(\mathbf{T}(\R), G) = W(\mathbf{T}(\R), K)$.
Then among the discrete series of $G$ with infinitesimal character $\lambda + \rho$, $\mathscr{H}_r$ is determined by the $G$-conjugacy class of the pair $(\mathbf{B}', \mathbf{T})$ (see section \ref{sectionDS}).
In our case the simple roots for $\mathbf{B}'$ are $e_1 - e_2, \dots, e_{n-1}-e_n$ and $-2e_1$.
\begin{rema}
This characterisation of the holomorphic discrete series in their L-packet is enough to determine which Adams-Johnson representations are holomorphic discrete series.
Using the notations of section \ref{sectionAJ}, the representation $\pi_{\psi, \mathbf{Q}, \mathbf{L}}$ is a holomorphic discrete series if and only if $\mathbf{Q} \supset \mathbf{B}'$ and $\mathbf{L}$ is anisotropic.
By \cite{ChRe}[Lemma 9.4] the packet $\Pi_{\psi}^{\mathrm{AJ}}$ contains a holomorphic discrete series representation if and only if $\mathrm{Std} \circ \psi$ does not contain $[d]$ or $\epsilon_{\C / \R}[d]$ as a factor for some $d>1$ (necessarily odd).
\end{rema}

We have made an arbitrary choice between $\mathbf{U}_+$ and $\mathbf{U}_-$.
We could have also identified $G/K$ with a bounded domain $D' \subset \mathfrak{u}_-$:
$$G/K \subset \mathbf{U}_-(\C) \mathbf{K}(\C) \mathbf{U}_+(\C) / \mathbf{K}(\C) \mathbf{U}_+(\C) \simeq \mathbf{U}_-(\C).$$
The resulting isomorphism of manifolds $D \simeq D'$ is antiholomorphic.
Given an infinitesimal character $\tau$ which occurs in a finite-dimensional representation of $G$, we have a discrete series representations of $G$ in the L-packet associated with $\tau$, $\pi_{\tau,+}^{\mathrm{hol}} := \left( \mathscr{H}_r\right)_{K\mathrm{-fin}}$ (resp.\ $\pi_{\tau,-}^{\mathrm{hol}}$).
It is characterised among irreducible unitary representations having infinitesimal character $\tau$ by the fact that it has a nonzero $K$-finite vector cancelled by $\mathfrak{u}_+$ (resp.\ $\mathfrak{u}_-$).
Since $K$ stabilises $\mathfrak{u}_+$ and $\mathfrak{u}_-$, $\pi_{\tau,+}^{\mathrm{hol}} \not\simeq \pi_{\tau,-}^{\mathrm{hol}}$.

Let us now define holomorphic discrete series for the group $G' = \mathrm{PGSp}(V,a)$.
Assume that $\sum_{k=1}^n m_k$ is even, i.e.\ the center of $G$ acts trivially in $\pi_{\tau,+}^{\mathrm{hol}}$ (and $\pi_{\tau,-}^{\mathrm{hol}}$).
The image of $G$ in $G'$ has index two, and there is an element of $G'$ normalizing $K$ and exchanging $\mathbf{U}_+$ and $\mathbf{U}_-$.
Thus if $\tau$ is such that the kernel of $\pi_{\tau,\pm}^{\mathrm{hol}}$ contains the center of $G$, $\pi_{\tau}^{\mathrm{hol}} := \Ind_G^{G'} \left( \pi_{\tau,+}^{\mathrm{hol}} \right)$ is irreducible and isomorphic to $\Ind_G^{G'} \left( \pi_{\tau,-}^{\mathrm{hol}} \right)$.
Among irreducible unitary representations having infinitesimal character $\tau$, $\pi_{\tau}^{\mathrm{hol}}$ is characterised by the fact that it has a nonzero $K$-finite vector cancelled by $\mathfrak{u}_+$.
Of course we could replace $\mathfrak{u}_+$ by $\mathfrak{u}_-$.

\subsection{Siegel modular forms and automorphic forms}

Let us recall the link between Siegel modular forms and automorphic cuspidal representations for the group $\mathbf{PGSp}$.
Almost all that we will need is contained in \cite{AsgariSchmidt}, in which the authors construct an isometric Hecke-equivariant map from the space of cuspidal Siegel modular forms to a certain space of cuspidal automorphic forms.
We will simply add a characterisation of the image of this map.

For the definitions and first properties of Siegel modular forms, see \cite{123mod} or \cite{Freitag}.
We will use the classical conventions and consider the alternate form $a$ on $\Z^{2n}$ whose matrix is $A = \begin{pmatrix} 0 & 1_n \\ -1_n & 0 \end{pmatrix} \in M_{2n}(\Z)$ for some integer $n \geq 1$.
Let $\mu : \mathbf{GSp}(A) \rightarrow \mathbf{GL}_1$ be the multiplier, defined by the relation $a(g(v_1), g(v_2)) = \mu(g) a(v_1,v_2)$.
Let $\mathbf{G} = \mathbf{Sp}(A) = \ker(\mu)$ and $\mathbf{G}' = \mathbf{PGSp}(A) = \mathbf{G}_{\mathrm{ad}}$, both reductive over $\Z$.

Recall the automorphy factor $j(g,\tau) = c \tau +d \in \mathrm{GL}_n(\C)$ for $g = \begin{pmatrix} a & b \\ c & d \end{pmatrix} \in \mathbf{GSp}(A, \R)$ and $\tau \in \mathcal{H}_n$.
As in the previous section denote by $K$ the stabiliser of $i 1_n \in \mathcal{H}_n$ under the action of $\mathbf{G}(\R)$.
Let $K'$ be the maximal compact subgroup of $\mathbf{G}'(\R)$ containing the image of $K$ by the natural morphism $\mathbf{G}(\R) \rightarrow \mathbf{G}'(\R)$.
Observe that the map $k = \begin{pmatrix} a & b \\ -b & a \end{pmatrix} \in K \mapsto j(k, i1_n) = a-ib$ is an isomorphism between $K$ and the unitary group $\mathrm{U}(1_n)$.
In the previous section, using the complex structure $J$ whose matrix is equal to $A$, we have identified $K$ with the unitary group $\mathrm{U}(h)$ for a positive definite hermitian form $h$ on $\R^{2n}$ with the complex structure $J$.
We emphasise that the the resulting isomorphism $\mathrm{U}(1_n) \simeq \mathrm{U}(h)$ is \emph{not} induced by an isomorphism between the hermitian spaces: one has to compose with the outer automorphism $x \mapsto {}^t x^{-1}$ on one side.

Let $(V, r)$ be an algebraic representation of $\mathbf{GL}_n$.
We can see the highest weight of $r$ as $(m_1, \dots, m_g)$ where $m_1 \geq \dots m_g$ are integers.
The representation $k \in K \mapsto r(j(k, i 1_n))$ is the restriction to $K$ of an algebraic representation $r'$ of $\mathbf{K}_{\C}$.
As in the previous section we choose a Borel pair $(\mathbf{B}_c, \mathbf{T})$ in $\mathbf{K}$ and denote by $e_1-e_2, \dots, e_{n-1}-e_n$ the corresponding simple roots.
Then the highest weight of $r'$ is $-m_ne_1 - \dots -m_1e_n$.

Let $\Gamma_n = \mathrm{Sp}(A, \Z)$, and denote by $S_r(\Gamma_n)$ the space of vector-valued Siegel modular forms of weight $r$.
When $m_1 = \dots = m_g$, that is when $r$ is one-dimensional, this is the space of scalar Siegel modular forms of weight $m_1$.
Asgari and Schmidt associate with any $f \in S_r(\Gamma_n)$ a function $\widetilde{\Phi}_f \in L^2( \mathbf{G}'(\Q) \backslash \mathbf{G}'(\A), V)$ such that
\begin{enumerate}
\item $\widetilde{\Phi}_f$ is right $\mathbf{G}'(\widehat{\Z})$-invariant,
\item for any $g \in \mathbf{G}'(\A)$, the function $\mathbf{G}'(\R) \rightarrow W, h \mapsto \widetilde{\Phi}_f(gh)$ is smooth,
\item for any $X \in \mathfrak{u}_-$ and any $g \in \mathbf{G}'(\A)$, $(X \cdot \widetilde{\Phi}_f)(g) = 0$,
\item for any $g \in \mathbf{G}'(\A)$ and any $k \in K$, $\widetilde{\Phi}_f(g k) = r(j(k, i1_n)) \widetilde{\Phi}_f(g)$,
\item $\widetilde{\Phi}_f$ is cuspidal.
\end{enumerate}
The third condition translates the Cauchy-Riemann equation for the holomorphy of $f$ into a condition on $\widetilde{\Phi}_f$.
If the measures are suitably normalised, $f \mapsto \widetilde{\Phi}_f$ is isometric for the Petersson hermitian product on $S_r(\Gamma_n)$.
Finally, $f \mapsto \widetilde{\Phi}_f$ is equivariant for the action of the unramified Hecke algebra at each finite place.

Let $\mathbf{N}_c$ be the unipotent radical of $\mathbf{B}_c$, let $\mathfrak{n}_c$ be its Lie algebra and let $\mathfrak{h}_0$ be the Lie algebra of $\mathbf{T}$.
The representation $r'$ allows to see $V$ as a simple $\mathfrak{k}$-module, and $\mathfrak{n}_c V$ has codimension one in $V$.
Let $L$ be a linear form on $V$ such that $\ker(L) = \mathfrak{n}_c V$.
We can see $X^*(\mathbf{T})$ as a lattice in $\mathrm{Hom}_{\R}(\mathfrak{h}_0, i\R) \subset \mathfrak{h}^*$.
Let $\lambda = m_1 e_1 + \dots + m_n e_n$ which we can see as an element of $(\mathfrak{h} \oplus \mathfrak{n}_c \oplus \mathfrak{u}_-)^*$ trivial on $\mathfrak{n}_c \oplus \mathfrak{u}_-$.
For any $v \in V$ and any $X \in \mathfrak{h} \oplus \mathfrak{n}_c \oplus \mathfrak{u}_-$, $L(-r(X)v) = \lambda(X)$.
For $g \in \mathbf{G}'(\A)$, define $\Phi_f(g)= L(\widetilde{\Phi}_f(g))$.
Then $\Phi_f \in L^2(\mathbf{G}'(\Q) \backslash \mathbf{G}'(\A))$ satisfies the following properties
\begin{enumerate}
\item $\Phi_f$ is right $\mathbf{G}'(\widehat{\Z})$-invariant and right $K'$-finite,
\item for any $g \in \mathbf{G}'(\A)$, the function $\mathbf{G}'(\R) \rightarrow W, h \mapsto \Phi_f(gh)$ is smooth,
\item for any $X \in \mathfrak{h} \oplus \mathfrak{n}_c \oplus \mathfrak{u}_-$ and any $g \in \mathbf{G}'(\A)$, $(X \cdot \Phi_f)(g) = \lambda(X) \Phi_f(g)$,
\item $\Phi_f$ is cuspidal.
\end{enumerate}
Again $f \mapsto \Phi_f$ is equivariant for the action of the unramified Hecke algebras at the finite places, and is isometric (up to a scalar).
The third condition implies that $\Phi_f$ is an eigenvector for $Z(U(\mathfrak{g}))$ and the infinitesimal character $\lambda + \rho_{\mathfrak{n}_c \oplus \mathfrak{u}_-} = (m_1-1)e_1 + \dots + (m_n-n) e_n$.
In particular $\Phi_f$ is a cuspidal automorphic form in the sense of \cite{BoJaCorv}, which we denote by $\Phi_f \in \mathcal{A}_{\mathrm{cusp}}(\mathbf{G}'(\Q) \backslash \mathbf{G}'(\A))$.

\begin{lemm} \label{lemmSiegelcusp}
Any $\Phi \in \mathcal{A}_{\mathrm{cusp}}(\mathbf{G}'(\Q) \backslash \mathbf{G}'(\A))$ satisfying the four conditions above is equal to $\Phi_f$ for a unique $f \in S_r(\Gamma_n)$.
\end{lemm}
\begin{proof}
Since $\Phi$ is $K'$-finite and transforms under $\mathfrak{h} \oplus \mathfrak{n}_c$ according to $\lambda$, $\Phi = L(\widetilde{\Phi})$ for a unique function $\widetilde{\Phi} : \mathbf{G}'(\Q) \backslash \mathbf{G}'(\A) \rightarrow V$ such that for $k \in K$, $\widetilde{\Phi}(gk) = r(j(k, i1_n))^{-1} \widetilde{\Phi}(g)$.
It is completely formal to check that there is a unique $f \in M_r(\Gamma_n)$ such that $\widetilde{\Phi} = \widetilde{\Phi}_f$, and thanks to the Koecher principle we only need to use that $\Phi$ has moderate growth when $n=1$.
We are left to show that $f$ is cuspidal.
Write $f(\tau) = \sum_{s \in \mathrm{Sym}_n} c(s) e^{2 i \pi \mathrm{Tr}(s \tau)}$ where $c_s \in V$ and the sum ranges over the set $\mathrm{Sym}_n$ of symmetric half-integral semi-positive definite $n \times n$.
We need to show that for any $s' \in \mathrm{Sym}_{n-1}$, $c\left(\begin{pmatrix} 0 & 0 \\ 0 & s' \end{pmatrix}\right) = 0.$
We use the cuspidality condition on $\Phi$ for the parabolic subgroup $\mathbf{P}$ of $\mathbf{G}$ defined over $\Z$ by
$$ \mathbf{P} = \left\{ \bordermatrix{ & 1 & n-1 & 1 & n-1 \cr 1 & * & * & * & * \cr n-1 & 0 & * & * & * \cr 1 & 0 & 0 & * & 0 \cr n-1 & 0 & * & * & * } \in \mathbf{G} \right\}. $$
Denote $\mathbf{N}$ the unipotent radical of $\mathbf{P}$, and observe that $\mathbf{N} = \mathbf{N}_0 \rtimes \mathbf{N}_1$ where
$$ \mathbf{N}_0 = \left\{  \begin{pmatrix} 1 & 0 & t_1 & t_2 \\ 0 & 1_{n-1} & {}^tt_2 & 0 \\ 0 & 0 & 1 & 0 \\ 0 & 0 & 0 & 1_{n-1} \end{pmatrix} \right\} \qquad \text{and} \qquad \mathbf{N}_1 = \left\{  \begin{pmatrix} 1 & t_3 & 0 & 0 \\ 0 & 1_{n-1} & 0 & 0 \\ 0 & 0 & 1 & 0 \\ 0 & 0 & -{}^tt_3 & 1_{n-1} \end{pmatrix} \right\} $$
are vector groups.
Moreover $\mathbf{N}_0(\Q) \backslash \mathbf{N}_0(\A) \simeq \mathbf{N}_0(\Z) \backslash \mathbf{N}_0(\R)$ and similarly for $\mathbf{N}_1$.
Therefore for any $g \in \mathbf{G}(\R)$,
$$ \int_{\mathbf{N}_1(\Z) \backslash \mathbf{N}_1(\R)} \int_{\mathbf{N}_0(\Z) \backslash \mathbf{N}_0(\R)} \widetilde{\Phi}(n_0n_1g) dn_0 dn_1 = 0. $$
By definition of $\widetilde{\Phi}_{\cdot}$, for some $m \in \R$ depending only on $r$,
$$ \widetilde{\Phi}(n_0n_1g) = \mu(g)^m r(j(n_0n_1g, i1_n))^{-1} f(n_0n_1g(i1_n)). $$
Fix $\tau \in \mathcal{H}_n$ of the form $\begin{pmatrix} i T & 0 \\ 0 & \tau' \end{pmatrix}$ where $T \in \R_{> 0}$ and $\tau' \in \mathcal{H}_{n-1}$, and let $g \in \mathbf{G}(\R)$ be such that $\tau = g(i1_n)$.
We will evaluate the inner integral first.
Fix $n_1 \in \mathbf{N}_1(\R)$ determined by $t_3 \in \R^{n-1}$ as above.
For any $n_0 \in \mathbf{N}(\R)$ determined by $(t_1,t_2) \in \R \times \R^{n-1}$ as above, $j(n_0n_1g, i1_n) = j(n_1g, i1_n)$ and we have the Fourier expansion
\begin{eqnarray*} \widetilde{\Phi}(n_0n_1g) &=& \mu(g)^m r(j(n_1g, i1_n))^{-1} \sum_{s_1 \in \Z, s_2 \in 1/2\Z^{n-1}} \left( \sum_{s' \in \mathrm{Sym}_{n-1}} c\left( \begin{pmatrix} s_1 & s_2 \\ {}^t s_2 & s' \end{pmatrix} \right) e^{2 i \pi \mathrm{Tr}(s' \tau')} \right) \\ & & \qquad \times \exp\left(2 i \pi (s_1 (t_3 \tau' {}^t t_3 + iT + t_1) + 2s_2(\tau' {}^t t_3 {}^tt_2))\right) \end{eqnarray*}
and thus
\begin{eqnarray*} \int_{\mathbf{N}_0(\Z) \backslash \mathbf{N}_0(\R)} \widetilde{\Phi}(n_0n_1g) dn_0 &=& \mu(g)^m r(j(n_1g, i1_n))^{-1} \sum_{s' \in \mathrm{Sym}_{n-1}} c\left(\begin{pmatrix} 0 & 0 \\ 0 & s' \end{pmatrix} \right) e^{2 i \pi \mathrm{Tr}(s' \tau')} \\
& = & \mu(g)^m r(j(g, i1_n))^{-1} \sum_{s' \in \mathrm{Sym}_{n-1}} c\left(\begin{pmatrix} 0 & 0 \\ 0 & s' \end{pmatrix} \right) e^{2 i \pi \mathrm{Tr}(s' \tau')} \end{eqnarray*}
does not depend on $n_1$.
Note that to get the last expression we used
$$ r(j(n_1, \tau))^{-1} c\left(\begin{pmatrix} 0 & 0 \\ 0 & s' \end{pmatrix} \right) = c\left(\begin{pmatrix} 1 & 0 \\ {}^t t_3 & 1 \end{pmatrix} \begin{pmatrix} 0 & 0 \\ 0 & s' \end{pmatrix} \begin{pmatrix} 1 & t_3 \\ 0 & 1 \end{pmatrix}\right) = c\left(\begin{pmatrix} 0 & 0 \\ 0 & s' \end{pmatrix} \right) .$$
Hence we can conclude that for any $s' \in \mathrm{Sym}_{n-1}$, $c\left(\begin{pmatrix} 0 & 0 \\ 0 & s' \end{pmatrix} \right) = 0$.
\end{proof}

Assume that $m_n \geq n+1$, i.e.\ that $\lambda + \rho_{\mathfrak{n}_c \oplus \mathfrak{u}_-}$ is the infinitesimal character of an L-packet of discrete series for $\mathbf{G}'(\R)$.
Assume also that $\sum_{k=1}^n m_k$ is even, since otherwise $S_r(\Gamma_n) = 0$.
By the theorem of Gelfand, Graev and Piatetski-Shapiro
$$\mathcal{A}_{\mathrm{cusp}}(\mathbf{G}'(\Q) \backslash \mathbf{G}'(\A)) \simeq \bigoplus_{\pi \in \Pi_{\mathrm{cusp}}(\mathbf{G}')} m_{\pi} \pi$$
where $\Pi_{\mathrm{cusp}}(\mathbf{G}')$ is the set of isomorphism classes of irreducible admissible $(\mathfrak{g}, K') \times \mathbf{G}'(\A_f)$-modules occurring in $\mathcal{A}_{\mathrm{cusp}}(\mathbf{G}'(\Q) \backslash \mathbf{G}'(\A))$ and $m_{\pi} \in \Z_{\geq 1}$.
Consider a $\pi \in \Pi_{\mathrm{cusp}}(\mathbf{G}')$.
For any prime $p$, $\pi_p^{\mathbf{G}'(\Zp)} \neq 0$ if and only if $\pi_p$ is unramified, and in that case $\dim_{\C} \pi_p^{\mathbf{G}'(\Zp)} = 1$.
Since $\pi_{\infty}$ is unitary, it has a highest weight vector for $(\lambda, \mathfrak{n}_c \oplus \mathfrak{u}_-)$ if and only if $\pi_{\infty}$ is the holomorphic discrete series with infinitesimal character $(m_1-1)e_1 + \dots + (m_n-n)e_n$, and in that case the space of highest weight vectors has dimension one.
Thus $\dim S_r(\Gamma_n)$ is equal the sum of the $m_{\pi}$ for $\pi = \otimes'_v \pi_v \in \Pi_{\mathrm{cusp}}(\mathbf{G}')$ such that $\pi_{\infty}$ is a holomorphic discrete series with infinitesimal character $(m_1-1)e_1 + \dots + (m_n-n) e_n$ and for any prime number $p$, $\pi_p$ is unramified.
By \cite{Wallach} any $\pi \in \Pi_{\mathrm{disc}}(\mathbf{G}') \smallsetminus \Pi_{\mathrm{cusp}}(\mathbf{G}')$ is such that $\pi_{\infty}$ is not tempered.
Therefore $\dim S_r(\Gamma_n)$ is equal to the sum of the multiplicities $m_{\pi}$ for $\pi \in \Pi_{\mathrm{disc}}(\mathbf{G}')$ such that
\begin{itemize}
\item for any prime number $p$, $\pi_p$ is unramified,
\item $\pi_{\infty}$ is the holomorphic discrete series representation $\pi_{\tau}^{\mathrm{hol}}$ with infinitesimal character $\tau = (m_1-1)e_1 + \dots + (m_n-n)e_n$.
\end{itemize}
Recall that $\mathbf{G} = \mathbf{Sp}_{2n}$.
Thanks to \cite{ChRe}[Proposition 4.7] we have that $\dim S_r(\Gamma_n)$ is also equal to the sum of the multiplicities $m_{\pi}$ for $\pi \in \Pi_{\mathrm{disc}}(\mathbf{G})$ such that $\pi$ is unramified everywhere and $\pi_{\infty} \simeq \pi_{\tau,+}^{\mathrm{hol}}$.
\begin{rema}
For any central isogeny $\mathbf{G} \rightarrow \mathbf{G}'$ between semisimple Chevalley groups over $\Z$, the integer denoted $[\pi_{\infty}, \pi_{\infty}']$ in \cite{ChRe}[Proposition 4.7] is always equal to $1$.
This follows from the fact that $\mathbf{G}'(\R) / \mathbf{G}(\R)$ is a finite abelian group.
\end{rema}
Thus we have an algorithm to compute $\dim S_r(\Gamma_n)$ from the cardinalities of $S(\cdot)$, $O_o(\cdot)$ and $O_e(\cdot)$, under Assumption \ref{assumAJ} if $m_1, \dots, m_n$ are not distinct.
Note that since the Adams-Johnson packets $\Pi_{\psi}^{\mathrm{AJ}}$ have multiplicity one, under Assumption \ref{assumAJ} the multiplicites $m_{\pi}$ for $\pi$ as above are all equal to $1$, and thus Siegel eigenforms in level one and weight $r$ satisfying $m_n \geq n+1$ have multiplicity one: up to a scalar they are determined by their Hecke eigenvalues at primes in a set of density one.
This was already observed in \cite{ChRe}[Corollary 4.10].

\begin{rema}
Without assuming that $m_n \geq n+1$, the construction in \cite{AsgariSchmidt} shows that $f \mapsto \Phi_f$ is an isometry from the space of square-integrable modular forms (for the Petersson scalar product) to the space of square-integrable automorphic forms which are $\lambda$-equivariant under $\mathfrak{n}_c \oplus \mathfrak{u}_-$ and $\mathbf{G}'(\Z)$-invariant.

In fact for $m_n \geq n +1$ (even $m_n \geq n$) we could avoid using \cite{Wallach} and Lemma \ref{lemmSiegelcusp} and use the fact \cite{Weissauer}[Satz 3] that for $m_n \geq n$ square-integrable Siegel modular forms are cusp forms.
\end{rema}

\subsection{Example: genus $4$}

Let us give more details in case $n=4$, which is interesting because there an endoscopic contribution from the group $\mathbf{SO}_8$ (the formal parameter $O_e(w_1,w_2,w_3,w_4) \boxplus 1$ below) which cannot be explained using lower genus Siegel eigenforms.
First we list the possible Arthur parameters for the group $\mathbf{Sp}_8$ in terms of the sets $S(w_1, \dots)$, $O_o(w_1, \dots)$ and $O_e(w_1, \dots)$.
The non-tempered ones only occur when $\lambda' = (m_1-n-1)e_1 + \dots + (m_n-n-1)e_n$ is orthogonal to a non-empty subset of the simple coroots $\{e_1^*-e_2^*, \dots, e_{n-1}^*-e_n^*, e_n^* \}$.
The convention in the following table is that the weights $w_i \in \frac{1}{2}\Z_{\geq 0}$ are decreasing with $i$.
For example $S(w_3)[2] \boxplus O_o(w_1,w_2)$ occurs only if $m_3 = m_4$, and if this is the case then
$$ (m_1,m_2,m_3,m_4) = \left(w_1+1, w_2+2, w_3+\frac{7}{2}, w_3+\frac{7}{2}\right). $$
\begin{table}[H] \centering
\renewcommand*{\arraystretch}{1.2}
\caption{Unramified cohomological Arthur parameters for $\mathbf{Sp}_8$}
\makebox[\linewidth]{
\begin{tabular}{c|c|c}
\hline
$O_o(w_1,w_2,w_3,w_4)$ & $O_e(w_1,w_2,w_3,w_4) \boxplus 1$ & $O_e(w_1, w_4) \boxplus O_e(w_2, w_3) \boxplus 1$ \\
$O_e(w_2, w_3) \boxplus O_o(w_1, w_4)$ & $O_e(w_1, w_4) \boxplus O_o(w_2, w_3)$ & $O_e(w_1, w_3) \boxplus O_e(w_2, w_4) \boxplus 1$ \\
$O_e(w_2, w_4) \boxplus O_o(w_1, w_3)$ & $O_e(w_1, w_3) \boxplus O_o(w_2, w_4)$ & $O_e(w_1, w_2) \boxplus O_e(w_3, w_4) \boxplus 1$ \\
$O_e(w_3, w_4) \boxplus O_o(w_1, w_2)$ & $O_e(w_1, w_2) \boxplus O_o(w_3, w_4)$ & $O_e(w_1, w_2) \boxplus S(w_3)[2] \boxplus 1$ \\
$S(w_3)[2] \boxplus O_o(w_1, w_2)$ & $O_e(w_1, w_4) \boxplus S(w_2)[2] \boxplus 1$ & $S(w_2)[2] \boxplus O_o(w_1, w_4)$ \\
$O_e(w_3, w_4) \boxplus S(w_1)[2] \boxplus 1$ & $S(w_1)[2] \boxplus O_o(w_3, w_4)$ & $S(w_1, w_3)[2] \boxplus 1$ \\
$S(w_1)[2] \boxplus S(w_3)[2] \boxplus 1$ & $S(w_1)[4] \boxplus 1$ & $S(w_1)[2] \boxplus [5]$ \\
$O_e(w_1, w_2) \boxplus [5]$ & $O_o(w_1)[3]$ & $[9]$ \\
\hline
\end{tabular}
}
\end{table}

Among these $24$ types for $\psi \in \Psi(\mathbf{Sp}_8)^{\mathrm{unr},\lambda'}$, some never yield Siegel modular forms.
In the last four cases ($S(w_1)[2] \boxplus [5]$, $O_e(w_1, w_2) \boxplus [5]$, $O_o(w_1)[3]$ and $[9]$), $\Pi_{\psi_{\infty}}$ does not contain the holomorphic discrete series.
In the other $20$ cases, $\Pi_{\psi_{\infty}}$ contains the holomorphic discrete series representation $\pi_{\tau,+}^{\mathrm{hol}}$ but it can happen that $\langle \cdot, \pi_{\tau,+}^{\mathrm{hol}} \rangle |_{S_{\psi}}$ never equals $\epsilon_{\psi}$.
For example if $\psi$ is tempered (the first $11$ cases) $\epsilon_{\psi}$ is always trivial, whereas $\langle \cdot, \pi_{\tau,+}^{\mathrm{hol}} \rangle |_{S_{\psi}}$ is trivial if and only if $\psi$ does not contain $O_e(w_1,w_2)$ or $O_e(w_1,w_4)$ or $O_e(w_2,w_3)$ as a factor.

In the following table we list the $11$ types that yield Siegel modular forms for some dominant weight $\lambda'$ for $\mathbf{Sp}_8$.
In the last column we give a necessary and sufficient condition on the weights for having $\langle \cdot, \pi_{\tau,+}^{\mathrm{hol}} \rangle |_{S_{\psi}} = \epsilon_{\psi}$.
\begin{table}[H] \centering
\renewcommand*{\arraystretch}{1.4}
\caption{The $11$ possible Arthur parameters of Siegel eigenforms for $\Gamma_4$}
\makebox[\linewidth]{
\begin{tabular}{c|c|c}
Type & $(m_1,m_2,m_3,m_4)$ & Occurs iff \\
\hline
$O_o(w_1,w_2,w_3,w_4)$                      & $(w_1+1,w_2+2,w_3+3,w_4+4)$ & always \\
$O_e(w_1,w_2,w_3,w_4) \boxplus 1$             & $(w_1+1,w_2+2,w_3+3,w_4+4)$ & always \\
$O_e(w_1,w_3) \boxplus O_e(w_2,w_4) \boxplus 1$ & $(w_1+1,w_2+2,w_3+3,w_4+4)$ & always \\
$O_e(w_2,w_4) \boxplus O_o(w_1,w_3)$          & $(w_1+1,w_2+2,w_3+3,w_4+4)$ & always \\
$O_e(w_1,w_3) \boxplus O_o(w_2,w_4)$          & $(w_1+1,w_2+2,w_3+3,w_4+4)$ & always \\
$S(w_3)[2] \boxplus O_o(w_1,w_2)$             & $(w_1+1,w_2+2,w_3+\frac{7}{2},w_3+\frac{7}{2})$ & $w_3+\frac{1}{2}$ is odd \\
$S(w_2)[2] \boxplus O_e(w_1,w_4) \boxplus 1$    & $(w_1+1,w_2+\frac{5}{2},w_2+\frac{5}{2},w_4+4)$ & $w_2+\frac{1}{2}$ is even \\
$S(w_2)[2] \boxplus O_o(w_1,w_4)$             & $(w_1+1,w_2+\frac{5}{2},w_2+\frac{5}{2},w_4+4)$ & $w_2+\frac{1}{2}$ is even \\
$S(w_1)[2] \boxplus O_o(w_3,w_4)$             & $(w_1+\frac{3}{2},w_1+\frac{3}{2},w_3+3,w_4+4)$ & $w_1+\frac{1}{2}$ is odd \\
$S(w_1,w_3)[2] \boxplus 1$                             & $(w_1+\frac{3}{2},w_1+\frac{3}{2},w_3+\frac{7}{2},w_3+\frac{7}{2})$ & $w_1+w_3$ is odd \\
$S(w_1)[4] \boxplus 1$                                 & $(w_1+\frac{3}{2},w_1+\frac{3}{2},w_1+\frac{3}{2},w_1+\frac{3}{2})$ & $w_1+\frac{1}{2}$ is even \\
\hline
\end{tabular}
}
\end{table}

\subsection{Some dimensions in the scalar case}
\label{scalarSiegel}

In genus $n$ greater than $4$ the enumeration of the possible Arthur parameters of Siegel eigenforms is best left to a computer.
Our implementation currently allows to compute $\dim S_r(\Gamma_n)$ for $n \leq 7$ and any algebraic representation $r$ of $\mathbf{GL}_n$ such that its highest weight $m_1 \geq \dots \geq m_n$ satisfies $m_n \geq n+1$.

Table \ref{tableScalarSiegel} displays the dimensions of some spaces of \emph{scalar} Siegel cusp forms.
Note that our method does \emph{not} allow to compute $\dim S_k(\Gamma_n)$ when $k \leq n$ (question marks in the bottom left corner), and that for scalar weights is is necessary to make Assumption \ref{assumAJ}.
We do not include the values $\dim S_k(\Gamma_n)$ when $n+1 \leq k \leq 7$ because they all vanish.
The question marks on the right side could be obtained simply by computing more traces in algebraic representations ($\mathrm{Tr}(\gamma\,|\,V_{\lambda})$ in the geometric side of the trace formula).
For more data see \url{http://www.math.ens.fr/~taibi/dimtrace/}.
For $n \geq 8$ we have not (yet) managed to compute the masses for $\mathbf{Sp}_{2n}$.
Nevertheless we can enumerate some endoscopic parameters, and thus give lower bounds for $\dim S_k(\Gamma_n)$: these are the starred numbers.
\begin{table}[H] \centering
\caption{Dimensions of spaces of scalar Siegel cusp forms}
\makebox[\linewidth]{
\begin{tabular}{r|ccccccccccccccc}\label{tableScalarSiegel}
$k$                    & $8$ & $9$ & $10$ & $11$ & $12$ & $13$ & $14$ & $15$  & $16$ & $17$ & $18$ & $19$ & $20$ & $21$ & $22$ \\ \hline
$\dim S_k(\Gamma_1)$   & $0$ & $0$ & $0$  & $0$  & $1$  & $0$  & $0$  & $0$   & $1$  & $0$  & $1$  & $0$  & $1$  & $0$  & $1$  \\
$\dim S_k(\Gamma_2)$   & $0$ & $0$ & $1$  & $0$  & $1$  & $0$  & $1$  & $0$   & $2$  & $0$  & $2$  & $0$  & $3$  & $0$  & $4$  \\
$\dim S_k(\Gamma_3)$   & $0$ & $0$ & $0$  & $0$  & $1$  & $0$  & $1$  & $0$   & $3$  & $0$  & $4$  & $0$  & $6$  & $0$  & $9$  \\
$\dim S_k(\Gamma_4)$   & $1$ & $0$ & $1$  & $0$  & $2$  & $0$  & $3$  & $0$   & $7$  & $0$  & $12$ & $1$  & $22$ & $1$  & $38$ \\
$\dim S_k(\Gamma_5)$   & $0$ & $0$ & $0$  & $0$  & $2$  & $0$  & $3$  & $0$   & $13$ & $0$  & $28$ & $0$  & $76$ & $0$  & $186$\\
$\dim S_k(\Gamma_6)$   & $0$ & $0$ & $1$  & $0$  & $3$  & $0$  & $9$  & $0$   & $33$ & $0$  & $117$& $1$  & $486$& ?    & ?    \\
$\dim S_k(\Gamma_7)$   & $0$ & $0$ & $0$  & $0$  & $3$  & $0$  & $9$  & $0$   & $83$ & $0$  & ?    & $0$  & ?    & $0$  & ?    \\ \hline
$\dim S_k(\Gamma_8)$   & ?   & $0^*$ & $1^*$ & $0^*$ & $4^*$  & $1^*$ & $23^*$ & $2^*$ & $234^*$ \\
$\dim S_k(\Gamma_9)$   & ?   & ?     & $0^*$ & $0^*$ & $2^*$  & $0^*$ & $25^*$ & $0^*$ & $843^*$ \\
$\dim S_k(\Gamma_{10})$& ?   & ?     & ?     & $0^*$ & $2^*$  & $0^*$ & $43^*$ & $1^*$ & $1591^*$ \\
$\dim S_k(\Gamma_{11})$& ?   & ?     & ?     & ?     & $1^*$  & $0^*$ & $32^*$ & $0^*$ & $6478^*$ \\
\hline
\end{tabular}
}
\end{table}

In principle for $n \leq 7$ one can compute the generating series $\sum_{k \geq n+1} \left(\dim S_k(\Gamma_n)\right) T^k$.
We have not attempted to do so for $n \geq 4$.

\section{Reliability}

The complete algorithm computing the three families of numbers
\begin{itemize}
\item $\mathrm{card} \left(S(w_1, \dots, w_n) \right)$ for $n \geq 1$, $w_i \in \frac{1}{2}\Z \smallsetminus \Z$ and $w_1 > \dots > w_n > 0$,
\item $\mathrm{card}\left(O_o(w_1, \dots, w_n)\right)$ for $n \geq 1$, $w_i \in \Z$ and $w_1 > \dots > w_n > 0$,
\item $\mathrm{card}\left(O_e(w_1, \dots, w_{2n})\right)$ for $n \geq 1$, $w_i \in \Z$ and $w_1 > \dots > w_{2n} \geq 0$,
\end{itemize}
is long and complicated.
Our implementation consists of more than 5000 lines of source code (mainly in Python, using Sage \cite{sage}), therefore it certainly contains errors.
There are several mathematically meaningful checks suggesting that the tables produced by our program are valid:
\begin{enumerate}
\item When computing the geometric side of the trace formula we obviously always find a rational number.
The trace formula asserts that it is equal to the spectral side, which is an integer, being an Euler-Poincaré characteristic.
The first check that our tables pass is thus that the geometric sides are indeed integral.
\item With a one-line modification, our algorithm can be used to compute global orbital integrals for special orthogonal groups $\mathbf{G}/ \Q$ which are split at every finite place and such that $\mathbf{G}(\R)$ is compact.
On a space of dimension $d$ such a group exists if and only $d = -1,0,1 \mod 8$.
Recall that for $d \in \{7,8,9\}$, up to isomorphism there is a unique regular and definite positive quadratic form $q : \Z^d \rightarrow \Z$.
These are the lattices $E_7$, $E_8$ and $E_8 \oplus A_1$.
Each one of these three lattices defines a reductive group $\mathbf{G}$ over $\Z$ such that $\mathbf{G}_{\Q}$ is as above, and their uniqueness is equivalent to the fact that the arithmetic genus $\mathbf{G}(\A_f) / \mathbf{G}(\widehat{\Z})$ has one element.
Chenevier and Renard \cite{ChRe} computed the geometric side of the trace formula, which is elementary and does not depend on Arthur's work in the anisotropic case, to \emph{count} level one automorphic representations for these groups.
This is possible because $\mathbf{G}(\Z)$ is closely related to the Weyl groups of the root systems $E_7$ and $E_8$, for which Carter \cite{Carter} described the conjugacy classes and their orders.
We checked that we obtain the same ``masses'' (see section \ref{SummaryAlgo}).
\item The numbers $\mathrm{card} \left(S(w_1, \dots, w_n) \right)$, $\mathrm{card}\left(O_o(w_1, \dots, w_n)\right)$ and $\mathrm{card}\left(O_e(w_1, \dots, w_{2n})\right)$ belong to $\Z_{\geq 0}$.
Our tables pass this check.
\item In low rank there are exceptional isogenies between the groups that we consider: $\mathbf{PGSp}_2 \simeq \mathbf{SO}_3$, $\mathbf{PGSp}_4 \simeq \mathbf{SO}_5$, $\left(\mathbf{SO}_4\right)_{\mathrm{sc}} \simeq \mathbf{SL}_2 \times \mathbf{SL}_2$, which by \cite{ChRe}[Proposition 4.7] imply:
\begin{enumerate}
\item For any odd $w_1 \in \Z_{>0}$, $\mathrm{card} \left(S(w_1/2) \right) = \mathrm{card}\left(O_o(w_1)\right)$.
Note that $\mathrm{card}\left(O_o(w_1)\right)=0$ if $w_1$ is even.
\item For any integers $w_1 > w_2 > 0$ such that $w_1 + w_2$ is odd,
$$\mathrm{card} \left(S\left(\frac{w_1+w_2}{2}, \frac{w_1-w_2}{2}\right) \right) = \mathrm{card}\left(O_o(w_1,w_2)\right).$$
Note that $\mathrm{card}\left(O_o(w_1,w_2)\right) = 0$ if $w_1+w_2$ is even.
\item For any integers $w_1 > w_2 > 0$ such that $w_1+w_2$ is odd,
$$\mathrm{card} \left(S\left(\frac{w_1+w_2}{2}\right)\right) \times  \mathrm{card}\left(S\left(\frac{w_1-w_2}{2}\right) \right) = \mathrm{card}\left(O_e(w_1,w_2)\right),$$
and for any odd integer $w>0$,
$$\binom{\mathrm{card}\left(S(\frac{w}{2}) \right)}{2} = O_e(w,0).$$
Note that $\mathrm{card}\left(O_e(w_1,w_2)\right) = 0$ if $w_1+w_2$ is even.
\end{enumerate}
\item By results of Mestre \cite{Mestre}, Fermigier \cite{Fermigier} and Miller \cite{Miller}, in low motivic weight (that is $2w_1$) some of the cardinalities of $S(w_1, \dots)$, $O_o(w_1, \dots)$ and $O_e(w_1, \dots)$ are known to vanish.
In forthcoming work, Chenevier and Lannes improve their method to show that if $n \geq 1$ and $\pi$ is a self-dual cuspidal automorphic representation of $\mathbf{GL}_n / \Q$ such that
\begin{itemize}
\item for any prime number $p$, $\pi_p$ is unramified,
\item the local Langlands parameter $\varphi$ of $\pi_{\infty}$ is either
\begin{itemize}
\item a direct sum of copies of $1$, $\epsilon_{\C / \R}$ and $I_r$ for integers $1 \leq r \leq 10$, or
\item a direct sum of copies of $I_r$ for $r \in \frac{1}{2}\Z \smallsetminus \Z$ and $\frac{1}{2} \leq r \leq \frac{19}{2}$.
\end{itemize}
\end{itemize}
then $\varphi$ belongs to the following list:
\begin{itemize}
\item $1$,
\item $ I_{11/2},\ I_{15/2},\ I_{17/2},\ I_{19/2} $,
\item $\epsilon_{\C / \R} \oplus I_{10},\ \epsilon_{\C / \R} \oplus I_9 $,
\item $I_{r/2} \oplus I_{19/2}$ with $r \in \{5, 7, 9, 11, 13\}$,
\item $I_4 \oplus I_9$, $I_r \oplus I_{10}$ with $r \in \{ 2, 3, 4, 5, 6, 7 \}$,
\item $1 \oplus I_6 \oplus I_{10},\ 1 \oplus I_7 \oplus I_{10}$.
\end{itemize}
Note that they make no regularity assumption.
This implies the vanishing of $2521$ values in our tables for groups of rank $\leq 6$.
In our tables, the only non-vanishing $\mathrm{card}\left(S(w_1, \dots)\right)$, $\mathrm{card}\left(O_o(w_1, \dots)\right)$ or $\mathrm{card}\left(O_e(w_1, \dots)\right)$ with $w_1 \leq 10$ are the following.
\begin{itemize}
\item For $w_1 \in \left\{\frac{11}{2}, \frac{15}{2}, \frac{17}{2}, \frac{19}{2} \right\}$, $\mathrm{card}\left(S(w_1)\right) = 1$.
These are the well-known modular forms.
\item $\mathrm{card}\left(S\left(\frac{19}{2}, \frac{7}{2}\right)\right) = 1$.
\end{itemize}
\item Finally, we can compare the values that we obtain for the dimensions of spaces of Siegel modular forms with known ones.
Our formulae coincide with those given in \cite{Igusa} (genus two, scalar) and \cite{TsuExp} and \cite{TsuPf} (genus two, vector-valued).
Tsuyumine \cite{Tsuyumine} gave a dimension formula in the scalar case in genus $3$.
There seems to be a typographical error in the formula on page 832 of \cite{Tsuyumine}, the denominator should be
$$ (1-T^4)(1-T^{12})^2(1-T^{14})(1-T^{18})(1-T^{20})(1-T^{30}) $$
instead of
$$ (1-T^4)(1-T^{12})^3(1-T^{14})(1-T^{18})(1-T^{20})(1-T^{30}). $$
With this correction we find the same formula as Tsuyumine.
In \cite{BFG} Bergström, Faber and van der Geer conjecture a formula for the cohomology of local systems on the moduli space $\mathcal{A}_3$ in terms of motives conjecturally associated with Siegel cusp forms.
As a corollary they obtain a conjectural formula for $\dim S_r(\Gamma_3)$ where $r$ is an algebraic representation of $\mathbf{GL}_3$ of highest weight $m_1 \geq m_2 \geq m_3 \geq 4$.
For $m_1 \leq 24$ ($1771$ values) we have checked that our values coincide.
We have also checked that our tables agree with Nebe and Venkov's theorem and conjecture in weight $12$ \cite{NebeVenkov} and Poor and Yuen's results in low weight \cite{PoorYuen}.
\end{enumerate}

\section{Tables}

\subsection{Masses}
\label{tablesMasses}

\begin{table}[H] \centering
\renewcommand*{\arraystretch}{1.2}
\caption{Masses for the group $\mathbf{SO}_3$}
\makebox[\linewidth]{
\begin{tabular}{cc|cc|cc}
\hline
Char.\ pol.\ & mass & Char.\ pol.\ & mass & Char.\ pol.\ & mass \\
\hline
$\Phi_{1}^{3} $ & $-1/12$ & $\Phi_{1} \Phi_{2}^{2} $ & $1/4$ & $\Phi_{1} \Phi_{3} $ & $1/3$ \\
\hline
\end{tabular}
}
\end{table}

\begin{table}[H] \centering
\renewcommand*{\arraystretch}{1.2}
\caption{Masses for the group $\mathbf{SO}_5$}
\makebox[\linewidth]{
\begin{tabular}{cc|cc|cc|cc}
\hline
Char.\ pol.\ & mass & Char.\ pol.\ & mass & Char.\ pol.\ & mass & Char.\ pol.\ & mass \\
\hline
$\Phi_{1}^{5} $ & $-1/1440$ & $\Phi_{1}^{3} \Phi_{2}^{2} $ & $-1/48$ & $\Phi_{1} \Phi_{2}^{4} $ & $7/288$ & $\Phi_{1} \Phi_{2}^{2} \Phi_{4} $ & $1/4$ \\
$\Phi_{1} \Phi_{4}^{2} $ & $-1/24$ & $\Phi_{1}^{3} \Phi_{3} $ & $-1/36$ & $\Phi_{1} \Phi_{2}^{2} \Phi_{3} $ & $1/12$ & $\Phi_{1} \Phi_{3}^{2} $ & $-1/36$ \\
$\Phi_{1} \Phi_{2}^{2} \Phi_{6} $ & $2/9$ & $\Phi_{1} \Phi_{6}^{2} $ & $-1/36$ & $\Phi_{1} \Phi_{12} $ & $1/6$ & $\Phi_{1} \Phi_{5} $ & $2/5$ \\
\hline
\end{tabular}
}
\end{table}

\begin{table}[H] \centering
\renewcommand*{\arraystretch}{1.2}
\caption{Masses for the group $\mathbf{SO}_7$}
\makebox[\linewidth]{
\begin{tabular}{cc|cc|cc}
\hline
Char.\ pol.\ & mass & Char.\ pol.\ & mass & Char.\ pol.\ & mass \\
\hline
$\Phi_{1}^{7} $ & $1/483840$ & $\Phi_{1}^{5} \Phi_{2}^{2} $ & $-19/23040$ & $\Phi_{1}^{3} \Phi_{2}^{4} $ & $-331/13824$ \\
$\Phi_{1} \Phi_{2}^{6} $ & $1/7680$ & $\Phi_{1}^{3} \Phi_{2}^{2} \Phi_{4} $ & $-11/192$ & $\Phi_{1} \Phi_{2}^{4} \Phi_{4} $ & $1/64$ \\
$\Phi_{1}^{3} \Phi_{4}^{2} $ & $25/1152$ & $\Phi_{1} \Phi_{2}^{2} \Phi_{4}^{2} $ & $-7/384$ & $\Phi_{1} \Phi_{2}^{2} \Phi_{8} $ & $3/16$ \\
$\Phi_{1} \Phi_{4} \Phi_{8} $ & $3/16$ & $\Phi_{1}^{5} \Phi_{3} $ & $-1/1440$ & $\Phi_{1}^{3} \Phi_{2}^{2} \Phi_{3} $ & $-1/36$ \\
$\Phi_{1} \Phi_{2}^{4} \Phi_{3} $ & $7/864$ & $\Phi_{1} \Phi_{2}^{2} \Phi_{3} \Phi_{4} $ & $1/24$ & $\Phi_{1} \Phi_{3} \Phi_{4}^{2} $ & $-1/72$ \\
$\Phi_{1}^{3} \Phi_{3}^{2} $ & $7/144$ & $\Phi_{1} \Phi_{2}^{2} \Phi_{3}^{2} $ & $-1/144$ & $\Phi_{1} \Phi_{3}^{3} $ & $1/216$ \\
$\Phi_{1}^{3} \Phi_{2}^{2} \Phi_{6} $ & $-23/432$ & $\Phi_{1} \Phi_{2}^{4} \Phi_{6} $ & $1/48$ & $\Phi_{1} \Phi_{2}^{2} \Phi_{4} \Phi_{6} $ & $1/8$ \\
$\Phi_{1} \Phi_{2}^{2} \Phi_{3} \Phi_{6} $ & $5/27$ & $\Phi_{1}^{3} \Phi_{6}^{2} $ & $1/432$ & $\Phi_{1} \Phi_{2}^{2} \Phi_{6}^{2} $ & $1/48$ \\
$\Phi_{1} \Phi_{3} \Phi_{6}^{2} $ & $1/216$ & $\Phi_{1}^{3} \Phi_{12} $ & $-1/72$ & $\Phi_{1} \Phi_{2}^{2} \Phi_{12} $ & $1/24$ \\
$\Phi_{1} \Phi_{3} \Phi_{12} $ & $5/36$ & $\Phi_{1} \Phi_{9} $ & $1/3$ & $\Phi_{1}^{3} \Phi_{5} $ & $-1/15$ \\
$\Phi_{1} \Phi_{2}^{2} \Phi_{5} $ & $1/10$ & $\Phi_{1} \Phi_{3} \Phi_{5} $ & $1/15$ & $\Phi_{1} \Phi_{2}^{2} \Phi_{10} $ & $3/10$ \\
$\Phi_{1} \Phi_{7} $ & $3/7$ & & & & \\
\hline
\end{tabular}
}
\end{table}

\begin{table}[H] \centering
\renewcommand*{\arraystretch}{1.2}
\caption{Masses for the group $\mathbf{SO}_9$}
\makebox[\linewidth]{
\begin{tabular}{cc|cc|cc}
\hline
Char.\ pol.\ & mass & Char.\ pol.\ & mass & Char.\ pol.\ & mass \\
\hline
$\Phi_{1}^{9} $ & $1/116121600$ & $\Phi_{1}^{7} \Phi_{2}^{2} $ & $1/1935360$ & $\Phi_{1}^{5} \Phi_{2}^{4} $ & $-4963/1658880$ \\
$\Phi_{1}^{3} \Phi_{2}^{6} $ & $-31/92160$ & $\Phi_{1} \Phi_{2}^{8} $ & $121/116121600$ & $\Phi_{1}^{5} \Phi_{2}^{2} \Phi_{4} $ & $-67/23040$ \\
$\Phi_{1}^{3} \Phi_{2}^{4} \Phi_{4} $ & $-7/768$ & $\Phi_{1} \Phi_{2}^{6} \Phi_{4} $ & $1/2560$ & $\Phi_{1}^{5} \Phi_{4}^{2} $ & $109/138240$ \\
$\Phi_{1}^{3} \Phi_{2}^{2} \Phi_{4}^{2} $ & $37/4608$ & $\Phi_{1} \Phi_{2}^{4} \Phi_{4}^{2} $ & $-331/27648$ & $\Phi_{1} \Phi_{2}^{2} \Phi_{4}^{3} $ & $1/128$ \\
$\Phi_{1} \Phi_{4}^{4} $ & $1/7680$ & $\Phi_{1}^{3} \Phi_{2}^{2} \Phi_{8} $ & $-1/64$ & $\Phi_{1} \Phi_{2}^{4} \Phi_{8} $ & $1/64$ \\
$\Phi_{1}^{3} \Phi_{4} \Phi_{8} $ & $-1/64$ & $\Phi_{1} \Phi_{2}^{2} \Phi_{4} \Phi_{8} $ & $21/64$ & $\Phi_{1} \Phi_{8}^{2} $ & $1/32$ \\
$\Phi_{1}^{7} \Phi_{3} $ & $1/1451520$ & $\Phi_{1}^{5} \Phi_{2}^{2} \Phi_{3} $ & $-49/23040$ & $\Phi_{1}^{3} \Phi_{2}^{4} \Phi_{3} $ & $-331/41472$ \\
\hline
\end{tabular}
}
\end{table}

\begin{table}[H] \centering
\renewcommand*{\arraystretch}{1.2}
\caption{Masses for the group $\mathbf{Sp}_2$}
\makebox[\linewidth]{
\begin{tabular}{cc|cc|cc}
\hline
Char.\ pol.\ & mass & Char.\ pol.\ & mass & Char.\ pol.\ & mass \\
\hline
$\Phi_{1}^{2} $ & $-1/12$ & $\Phi_{2}^{2} $ & $-1/12$ & $\Phi_{4} $ & $1/2$ \\
$\Phi_{3} $ & $1/3$ & $\Phi_{6} $ & $1/3$ & & \\
\hline
\end{tabular}
}
\end{table}

\begin{table}[H] \centering
\renewcommand*{\arraystretch}{1.2}
\caption{Masses for the group $\mathbf{Sp}_4$}
\makebox[\linewidth]{
\begin{tabular}{cc|cc|cc|cc}
\hline
Char.\ pol.\ & mass & Char.\ pol.\ & mass & Char.\ pol.\ & mass & Char.\ pol.\ & mass \\
\hline
$\Phi_{1}^{4} $ & $-1/1440$ & $\Phi_{1}^{2} \Phi_{2}^{2} $ & $7/144$ & $\Phi_{2}^{4} $ & $-1/1440$ & $\Phi_{1}^{2} \Phi_{4} $ & $-1/24$ \\
$\Phi_{2}^{2} \Phi_{4} $ & $-1/24$ & $\Phi_{4}^{2} $ & $-1/24$ & $\Phi_{8} $ & $1/2$ & $\Phi_{1}^{2} \Phi_{3} $ & $-1/36$ \\
$\Phi_{2}^{2} \Phi_{3} $ & $-1/36$ & $\Phi_{3} \Phi_{4} $ & $1/6$ & $\Phi_{3}^{2} $ & $-1/36$ & $\Phi_{1}^{2} \Phi_{6} $ & $-1/36$ \\
$\Phi_{2}^{2} \Phi_{6} $ & $-1/36$ & $\Phi_{4} \Phi_{6} $ & $1/6$ & $\Phi_{3} \Phi_{6} $ & $4/9$ & $\Phi_{6}^{2} $ & $-1/36$ \\
$\Phi_{12} $ & $1/6$ & $\Phi_{5} $ & $2/5$ & $\Phi_{10} $ & $2/5$ & & \\ 
\hline
\end{tabular}
}
\end{table}

\begin{table}[H] \centering
\renewcommand*{\arraystretch}{1.2}
\caption{Masses for the group $\mathbf{Sp}_6$}
\makebox[\linewidth]{
\begin{tabular}{cc|cc|cc}
\hline
Char.\ pol.\ & mass & Char.\ pol.\ & mass & Char.\ pol.\ & mass \\
\hline
$\Phi_{1}^{6} $ & $1/362880$ & $\Phi_{1}^{4} \Phi_{2}^{2} $ & $31/17280$ & $\Phi_{1}^{2} \Phi_{2}^{4} $ & $31/17280$ \\
$\Phi_{2}^{6} $ & $1/362880$ & $\Phi_{1}^{4} \Phi_{4} $ & $-1/2880$ & $\Phi_{1}^{2} \Phi_{2}^{2} \Phi_{4} $ & $7/288$ \\
$\Phi_{2}^{4} \Phi_{4} $ & $-1/2880$ & $\Phi_{1}^{2} \Phi_{4}^{2} $ & $7/288$ & $\Phi_{2}^{2} \Phi_{4}^{2} $ & $7/288$ \\
$\Phi_{4}^{3} $ & $1/48$ & $\Phi_{1}^{2} \Phi_{8} $ & $-1/24$ & $\Phi_{2}^{2} \Phi_{8} $ & $-1/24$ \\
$\Phi_{4} \Phi_{8} $ & $3/4$ & $\Phi_{1}^{4} \Phi_{3} $ & $-1/4320$ & $\Phi_{1}^{2} \Phi_{2}^{2} \Phi_{3} $ & $7/432$ \\
$\Phi_{2}^{4} \Phi_{3} $ & $-1/4320$ & $\Phi_{1}^{2} \Phi_{3} \Phi_{4} $ & $-1/72$ & $\Phi_{2}^{2} \Phi_{3} \Phi_{4} $ & $-1/72$ \\
$\Phi_{3} \Phi_{4}^{2} $ & $-1/72$ & $\Phi_{3} \Phi_{8} $ & $1/6$ & $\Phi_{1}^{2} \Phi_{3}^{2} $ & $25/432$ \\
\hline
\end{tabular}
}
\end{table}

\begin{table}[H] \centering \footnotesize
\renewcommand*{\arraystretch}{1.2}
\caption{Masses for the group $\mathbf{Sp}_8$}
\makebox[\linewidth]{
\begin{tabular}{cc|cc|cc|cc}
\hline
Char.\ pol.\ & mass & Char.\ pol.\ & mass & Char.\ pol.\ & mass & Char.\ pol.\ & mass \\
\hline
$\Phi_{1}^{8} $ & $1/87091200$ & $\Phi_{1}^{6} \Phi_{2}^{2} $ & $-127/4354560$ & $\Phi_{1}^{4} \Phi_{2}^{4} $ & $871/2073600$ & $\Phi_{1}^{2} \Phi_{2}^{6} $ & $-127/4354560$ \\
$\Phi_{2}^{8} $ & $1/87091200$ & $\Phi_{1}^{6} \Phi_{4} $ & $1/725760$ & $\Phi_{1}^{4} \Phi_{2}^{2} \Phi_{4} $ & $31/34560$ & $\Phi_{1}^{2} \Phi_{2}^{4} \Phi_{4} $ & $31/34560$ \\
$\Phi_{2}^{6} \Phi_{4} $ & $1/725760$ & $\Phi_{1}^{4} \Phi_{4}^{2} $ & $31/34560$ & $\Phi_{1}^{2} \Phi_{2}^{2} \Phi_{4}^{2} $ & $-361/3456$ & $\Phi_{2}^{4} \Phi_{4}^{2} $ & $31/34560$ \\
$\Phi_{1}^{2} \Phi_{4}^{3} $ & $-7/576$ & $\Phi_{2}^{2} \Phi_{4}^{3} $ & $-7/576$ & $\Phi_{4}^{4} $ & $1/5760$ & $\Phi_{1}^{4} \Phi_{8} $ & $-1/2880$ \\
$\Phi_{1}^{2} \Phi_{2}^{2} \Phi_{8} $ & $7/288$ & $\Phi_{2}^{4} \Phi_{8} $ & $-1/2880$ & $\Phi_{1}^{2} \Phi_{4} \Phi_{8} $ & $-3/16$ & $\Phi_{2}^{2} \Phi_{4} \Phi_{8} $ & $-3/16$ \\
$\Phi_{4}^{2} \Phi_{8} $ & $-1/48$ & $\Phi_{8}^{2} $ & $1/24$ & $\Phi_{16} $ & $1$ & $\Phi_{1}^{6} \Phi_{3} $ & $1/1088640$ \\
$\Phi_{1}^{4} \Phi_{2}^{2} \Phi_{3} $ & $31/51840$ & $\Phi_{1}^{2} \Phi_{2}^{4} \Phi_{3} $ & $31/51840$ & $\Phi_{2}^{6} \Phi_{3} $ & $1/1088640$ & $\Phi_{1}^{4} \Phi_{3} \Phi_{4} $ & $-1/8640$ \\
\hline
\end{tabular}
}
\end{table}

For even orthogonal groups and when the characteristic polynomial is coprime to $\Phi_1 \Phi_2$, the characteristic polynomial defines two conjugacy classes over $\Qbar$.
They have the same mass.

\begin{table}[H] \centering
\renewcommand*{\arraystretch}{1.2}
\caption{Masses for the group $\mathbf{SO}_4$}
\makebox[\linewidth]{
\begin{tabular}{cc|cc|cc}
\hline
Char.\ pol.\ & mass & Char.\ pol.\ & mass & Char.\ pol.\ & mass \\
\hline
$\Phi_{1}^{4} $ & $1/144$ & $\Phi_{1}^{2} \Phi_{2}^{2} $ & $1/8$ & $\Phi_{2}^{4} $ & $1/144$ \\
$\Phi_{4}^{2} $ & $-1/24$ & $\Phi_{1}^{2} \Phi_{3} $ & $1/9$ & $\Phi_{3}^{2} $ & $-1/36$ \\
$\Phi_{2}^{2} \Phi_{6} $ & $1/9$ & $\Phi_{6}^{2} $ & $-1/36$ & $\Phi_{12} $ & $1/6$ \\
\hline
\end{tabular}
}
\end{table}

\begin{table}[H] \centering
\renewcommand*{\arraystretch}{1.2}
\caption{Masses for the group $\mathbf{SO}_8$}
\makebox[\linewidth]{
\begin{tabular}{cc|cc|cc}
\hline
Char.\ pol.\ & mass & Char.\ pol.\ & mass & Char.\ pol.\ & mass \\
\hline
$\Phi_{1}^{8} $ & $1/58060800$ & $\Phi_{1}^{6} \Phi_{2}^{2} $ & $1/15360$ & $\Phi_{1}^{4} \Phi_{2}^{4} $ & $1357/165888$ \\
$\Phi_{1}^{2} \Phi_{2}^{6} $ & $1/15360$ & $\Phi_{2}^{8} $ & $1/58060800$ & $\Phi_{1}^{4} \Phi_{2}^{2} \Phi_{4} $ & $1/64$ \\
$\Phi_{1}^{2} \Phi_{2}^{4} \Phi_{4} $ & $1/64$ & $\Phi_{1}^{4} \Phi_{4}^{2} $ & $-55/13824$ & $\Phi_{1}^{2} \Phi_{2}^{2} \Phi_{4}^{2} $ & $17/768$ \\
$\Phi_{2}^{4} \Phi_{4}^{2} $ & $-55/13824$ & $\Phi_{4}^{4} $ & $1/7680$ & $\Phi_{1}^{2} \Phi_{2}^{2} \Phi_{8} $ & $3/16$ \\
$\Phi_{1}^{2} \Phi_{4} \Phi_{8} $ & $3/32$ & $\Phi_{2}^{2} \Phi_{4} \Phi_{8} $ & $3/32$ & $\Phi_{8}^{2} $ & $1/32$ \\
$\Phi_{1}^{6} \Phi_{3} $ & $1/25920$ & $\Phi_{1}^{4} \Phi_{2}^{2} \Phi_{3} $ & $1/96$ & $\Phi_{1}^{2} \Phi_{2}^{4} \Phi_{3} $ & $41/5184$ \\
$\Phi_{1}^{2} \Phi_{2}^{2} \Phi_{3} \Phi_{4} $ & $1/8$ & $\Phi_{1}^{2} \Phi_{3} \Phi_{4}^{2} $ & $1/432$ & $\Phi_{1}^{4} \Phi_{3}^{2} $ & $-19/1728$ \\
\hline
\end{tabular}
}
\end{table}

\subsection{Some essentially self-dual, algebraic, level one, automorphic cuspidal representations of $\GL_n$ for $n \leq 13$}
\label{tablestwGL}

The following tables list the \emph{non-zero}
$$\mathrm{card}(S(w_1,\dots, w_n)),\ \mathrm{card}(O_o(w_1,\dots,w_n)) \text{ and } \mathrm{card}(O_e(w_1,\dots,w_{2n}))$$
as defined in the introduction.
These values depend on Assumption \ref{assumAJweak} when $w_i = w_{i+1} + 1$ for some $i$ or
\begin{itemize}
\item $w_n = \frac{1}{2}$ for $\mathrm{card}(S(w_1,\dots, w_n))$,
\item $w_n = 1$ for $\mathrm{card}(O_o(w_1,\dots,w_n))$,
\item $w_n = 0$ for $\mathrm{card}(O_e(w_1,\dots,w_{2n}))$.
\end{itemize}
Much more data is available at \url{http://www.math.ens.fr/~taibi/dimtrace/}.

\begin{table}[H] \centering
\caption{$\mathrm{card}\left(S(w)\right)$}
\begin{tabular}{cc|cc|cc|cc}
\hline
$2w$ & card. & $2w$ & card. & $2w$ & card. & $2w$ & card. \\
\hline
11 & 1 & 23 & 2 & 33 & 2 & 43 & 3 \\
15 & 1 & 25 & 1 & 35 & 3 & 45 & 3 \\
17 & 1 & 27 & 2 & 37 & 2 & 47 & 4 \\
19 & 1 & 29 & 2 & 39 & 3 & 49 & 3 \\
21 & 1 & 31 & 2 & 41 & 3 & 51 & 4 \\
\hline
\end{tabular}
\end{table}

\begin{table}[H] \centering
\caption{$\mathrm{card}\left(S(w_1,w_2)\right)$}
\begin{tabular}{lc|lc|lc|lc}
\hline
$(2w_1,2w_2)$ & card. & $(2w_1,2w_2)$ & card. & $(2w_1,2w_2)$ & card. & $(2w_1,2w_2)$ & card. \\
\hline
(19, 7) & 1 & (31, 5) & 1 & (35, 25) & 5 & (39, 33) & 4 \\
(21, 5) & 1 & (31, 7) & 3 & (35, 27) & 3 & (39, 35) & 1 \\
(21, 9) & 1 & (31, 9) & 2 & (35, 29) & 2 & (39, 37) & 1 \\
(21, 13) & 1 & (31, 11) & 3 & (35, 31) & 1 & (41, 1) & 1 \\
(23, 7) & 1 & (31, 13) & 4 & (37, 1) & 1 & (41, 3) & 1 \\
(23, 9) & 1 & (31, 15) & 4 & (37, 5) & 4 & (41, 5) & 6 \\
(23, 13) & 1 & (31, 17) & 3 & (37, 7) & 3 & (41, 7) & 4 \\
\hline
\end{tabular}
\end{table}

\begin{table}[H] \centering
\caption{$\mathrm{card}\left(S(w_1,w_2,w_3)\right)$}
\begin{tabular}{lc|lc|lc|lc}
\hline
$(2w_i)_i$ & card. & $(2w_i)_i$ & card. & $(2w_i)_i$ & card. & $(2w_i)_i$ & card. \\
\hline
(23, 13, 5) & 1 & (27, 21, 7) & 2 & (29, 19, 13) & 5 & (31, 13, 5) & 3 \\
(23, 15, 3) & 1 & (27, 21, 9) & 4 & (29, 19, 15) & 1 & (31, 13, 7) & 2 \\
(23, 15, 7) & 1 & (27, 21, 11) & 2 & (29, 19, 17) & 1 & (31, 13, 9) & 4 \\
(23, 17, 5) & 1 & (27, 21, 13) & 3 & (29, 21, 3) & 5 & (31, 15, 3) & 3 \\
(23, 17, 9) & 1 & (27, 21, 15) & 1 & (29, 21, 5) & 1 & (31, 15, 5) & 2 \\
(23, 19, 3) & 1 & (27, 21, 17) & 1 & (29, 21, 7) & 10 & (31, 15, 7) & 5 \\
(23, 19, 11) & 1 & (27, 23, 3) & 1 & (29, 21, 9) & 4 & (31, 15, 9) & 3 \\
\hline
\end{tabular}
\end{table}

\begin{table}[H] \centering
\caption{$\mathrm{card}\left(S(w_1,w_2,w_3,w_4)\right)$}
\makebox[\linewidth]{
\begin{tabular}{lc|lc|lc|lc}
\hline
$(2w_i)_i$ & card. & $(2w_i)_i$ & card. & $(2w_i)_i$ & card. & $(2w_i)_i$ & card.  \\
\hline
(25, 17, 9, 5) & 1 & (27, 19, 13, 9) & 3 & (27, 23, 15, 7) & 7 & (27, 25, 19, 7) & 3 \\
(25, 17, 13, 5) & 1 & (27, 19, 15, 3) & 2 & (27, 23, 15, 9) & 4 & (27, 25, 19, 9) & 6 \\
(25, 19, 9, 3) & 1 & (27, 19, 15, 5) & 1 & (27, 23, 15, 11) & 5 & (27, 25, 19, 11) & 3 \\
(25, 19, 11, 5) & 1 & (27, 19, 15, 7) & 1 & (27, 23, 15, 13) & 1 & (27, 25, 19, 13) & 3 \\
(25, 19, 13, 3) & 1 & (27, 19, 15, 9) & 1 & (27, 23, 17, 1) & 5 & (27, 25, 21, 3) & 4 \\
(25, 19, 13, 5) & 1 & (27, 19, 17, 5) & 1 & (27, 23, 17, 3) & 2 & (27, 25, 21, 7) & 4 \\
(25, 19, 13, 7) & 1 & (27, 19, 17, 9) & 1 & (27, 23, 17, 5) & 6 & (27, 25, 21, 9) & 2 \\
\hline
\end{tabular} }
\end{table}

\begin{table}[H]
\caption{$\mathrm{card}\left(S(w_1,w_2,w_3,w_4)\right)$}
\center
\begin{tabular}{lc|lc|lc}
\hline
$(2w_i)_i$ & card. & $(2w_i)_i$ & card. & $(2w_i)_i$ & card. \\
\hline
(23, 21, 17, 11, 3) & 1 & (27, 19, 15, 13, 5) & 1 & (27, 23, 13, 7, 5) & 1 \\
(25, 19, 15, 9, 3) & 1 & (27, 19, 17, 9, 3) & 2 & (27, 23, 13, 9, 3) & 4 \\
(25, 21, 13, 9, 3) & 1 & (27, 19, 17, 11, 5) & 2 & (27, 23, 13, 9, 7) & 1 \\
(25, 21, 15, 7, 3) & 1 & (27, 19, 17, 13, 3) & 1 & (27, 23, 13, 11, 5) & 1 \\
(25, 21, 15, 9, 1) & 1 & (27, 21, 13, 7, 3) & 2 & (27, 23, 15, 7, 3) & 2 \\
(25, 21, 15, 9, 5) & 1 & (27, 21, 13, 9, 1) & 1 & (27, 23, 15, 7, 5) & 1 \\
(25, 21, 15, 11, 3) & 1 & (27, 21, 13, 9, 5) & 1 & (27, 23, 15, 9, 1) & 2 \\
\hline
\end{tabular}
\end{table}

\begin{table}[H] \centering
\caption{$\mathrm{card}\left(S(w_1,w_2,w_3,w_4,w_5,w_6)\right)$}
\makebox[\linewidth]{
\begin{tabular}{lc|lc|lc}
\hline
$(2w_i)_i$ & card. & $(2w_i)_i$ & card. & $(2w_i)_i$ & card. \\
\hline
(25, 21, 17, 13, 7, 3) & 1 & (25, 23, 21, 19, 11, 7) & 1 & (27, 23, 17, 15, 7, 1) & 4 \\
(25, 23, 17, 11, 7, 3) & 1 & (25, 23, 21, 19, 13, 5) & 1 & (27, 23, 17, 15, 7, 5) & 3 \\
(25, 23, 17, 13, 7, 1) & 1 & (27, 21, 15, 13, 7, 3) & 1 & (27, 23, 17, 15, 9, 3) & 8 \\
(25, 23, 17, 13, 9, 3) & 2 & (27, 21, 17, 11, 7, 3) & 1 & (27, 23, 17, 15, 9, 5) & 1 \\
(25, 23, 19, 13, 7, 3) & 2 & (27, 21, 17, 13, 5, 3) & 1 & (27, 23, 17, 15, 9, 7) & 1 \\
(25, 23, 19, 13, 9, 1) & 2 & (27, 21, 17, 13, 7, 1) & 1 & (27, 23, 17, 15, 11, 1) & 2 \\
(25, 23, 19, 13, 9, 5) & 3 & (27, 21, 17, 13, 7, 5) & 2 & (27, 23, 17, 15, 11, 5) & 3 \\
\hline
\end{tabular} }
\end{table}

\begin{table}[H]
\caption{$\mathrm{card}\left(O_o(w)\right)$}
\center
\begin{tabular}{lc|lc|lc|lc}
\hline
$w$ & card. & $w$ & card. & $w$ & card. & $w$ & card. \\
\hline
11 & 1 & 21 & 1 & 29 & 2 & 37 & 2 \\
15 & 1 & 23 & 2 & 31 & 2 & 39 & 3 \\
17 & 1 & 25 & 1 & 33 & 2 & 41 & 3 \\
19 & 1 & 27 & 2 & 35 & 3 & 43 & 3 \\
\hline
\end{tabular}
\end{table}

\begin{table}[H]
\caption{$\mathrm{card}\left(O_o(w_1,w_2)\right)$}
\center
\begin{tabular}{lc|lc|lc|lc}
\hline
$(w_1,w_2)$ & card. & $(w_1,w_2)$ & card. & $(w_1,w_2)$ & card. & $(w_1,w_2)$ & card. \\
\hline
(13, 6) & 1 & (21, 12) & 5 & (25, 6) & 4 & (28, 7) & 6 \\
(13, 8) & 1 & (21, 14) & 5 & (25, 8) & 6 & (28, 9) & 8 \\
(15, 6) & 1 & (21, 16) & 4 & (25, 10) & 6 & (28, 11) & 8 \\
(15, 8) & 1 & (21, 18) & 3 & (25, 12) & 9 & (28, 13) & 10 \\
(15, 10) & 1 & (21, 20) & 1 & (25, 14) & 8 & (28, 15) & 11 \\
(15, 12) & 1 & (22, 3) & 1 & (25, 16) & 9 & (28, 17) & 9 \\
(16, 7) & 1 & (22, 5) & 1 & (25, 18) & 9 & (28, 19) & 10 \\
\hline
\end{tabular}
\end{table}

\begin{table}[H]
\caption{$\mathrm{card}\left(O_o(w_1,w_2,w_3)\right)$}
\center
\begin{tabular}{lc|lc|lc}
\hline
$(w_i)_i$ & card. & $(w_i)_i$ & card. & $(w_i)_i$ & card. \\
\hline
(12, 8, 4) & 1 & (16, 9, 5) & 1 & (17, 12, 7) & 5 \\
(13, 8, 5) & 1 & (16, 9, 7) & 1 & (17, 12, 9) & 6 \\
(13, 10, 3) & 1 & (16, 10, 2) & 1 & (17, 12, 11) & 2 \\
(13, 10, 5) & 1 & (16, 10, 4) & 2 & (17, 13, 2) & 1 \\
(13, 10, 7) & 1 & (16, 10, 6) & 2 & (17, 13, 4) & 4 \\
(13, 12, 5) & 1 & (16, 10, 8) & 1 & (17, 13, 6) & 5 \\
(13, 12, 7) & 1 & (16, 11, 3) & 1 & (17, 13, 8) & 3 \\
\hline
\end{tabular}
\end{table}

\begin{table}[H]
\caption{$\mathrm{card}\left(O_o(w_1,w_2,w_3,w_4)\right)$}
\center
\begin{tabular}{lc|lc|lc}
\hline
$(w_i)_i$ & card. & $(w_i)_i$ & card. & $(w_i)_i$ & card. \\
\hline
(13, 10, 9, 4) & 1 & (15, 12, 11, 6) & 4 & (15, 14, 13, 10) & 3 \\
(13, 12, 7, 4) & 1 & (15, 12, 11, 8) & 2 & (16, 8, 7, 3) & 1 \\
(13, 12, 9, 4) & 1 & (15, 12, 11, 10) & 1 & (16, 10, 7, 3) & 1 \\
(13, 12, 9, 6) & 1 & (15, 13, 7, 3) & 2 & (16, 10, 7, 5) & 1 \\
(13, 12, 11, 4) & 1 & (15, 13, 7, 5) & 1 & (16, 10, 8, 4) & 1 \\
(14, 10, 9, 5) & 1 & (15, 13, 8, 2) & 1 & (16, 10, 9, 3) & 2 \\
(14, 12, 7, 3) & 1 & (15, 13, 8, 4) & 1 & (16, 10, 9, 5) & 2 \\
\hline
\end{tabular}
\end{table}

\begin{table}[H]
\caption{$\mathrm{card}\left(O_o(w_1,w_2,w_3,w_4,w_5)\right)$}
\center
\begin{tabular}{lc|lc|lc}
\hline
$(w_i)_i$ & card. & $(w_i)_i$ & card. & $(w_i)_i$ & card. \\
\hline
(13, 10, 9, 6, 3) & 1 & (14, 12, 11, 6, 4) & 1 & (14, 13, 11, 9, 6) & 2 \\
(13, 11, 9, 8, 4) & 1 & (14, 12, 11, 7, 1) & 1 & (14, 13, 12, 7, 1) & 1 \\
(13, 12, 7, 6, 3) & 1 & (14, 12, 11, 7, 3) & 1 & (14, 13, 12, 7, 3) & 1 \\
(13, 12, 9, 4, 3) & 1 & (14, 12, 11, 7, 5) & 1 & (14, 13, 12, 8, 4) & 1 \\
(13, 12, 9, 6, 1) & 1 & (14, 12, 11, 8, 2) & 2 & (14, 13, 12, 9, 5) & 1 \\
(13, 12, 9, 6, 3) & 1 & (14, 12, 11, 8, 4) & 2 & (15, 10, 7, 6, 3) & 1 \\
(13, 12, 9, 6, 5) & 1 & (14, 12, 11, 8, 6) & 1 & (15, 10, 9, 4, 3) & 1 \\
\hline
\end{tabular}
\end{table}

\begin{table}[H]
\caption{$\mathrm{card}\left(O_o(w_1,w_2,w_3,w_4,w_5,w_6)\right)$}
\center
\begin{tabular}{lc|lc|lc}
\hline
$(w_i)_i$ & card. & $(w_i)_i$ & card. & $(w_i)_i$ & card. \\
\hline
(13, 11, 10, 8, 4, 3) & 1 & (14, 12, 11, 8, 4, 2) & 1 & (14, 13, 10, 8, 7, 1) & 1 \\
(13, 12, 9, 8, 5, 4) & 1 & (14, 12, 11, 8, 5, 1) & 2 & (14, 13, 10, 8, 7, 3) & 2 \\
(13, 12, 9, 8, 7, 2) & 1 & (14, 12, 11, 8, 5, 3) & 4 & (14, 13, 10, 8, 7, 5) & 1 \\
(13, 12, 10, 8, 5, 3) & 1 & (14, 12, 11, 8, 6, 2) & 3 & (14, 13, 10, 9, 5, 2) & 1 \\
(13, 12, 11, 6, 5, 4) & 1 & (14, 12, 11, 8, 6, 4) & 3 & (14, 13, 10, 9, 5, 4) & 1 \\
(13, 12, 11, 8, 3, 2) & 1 & (14, 12, 11, 8, 7, 1) & 2 & (14, 13, 10, 9, 7, 2) & 2 \\
(13, 12, 11, 8, 5, 2) & 2 & (14, 12, 11, 8, 7, 3) & 3 & (14, 13, 10, 9, 7, 4) & 1 \\
\hline
\end{tabular}
\end{table}

\begin{table}[H]
\caption{$\mathrm{card}\left(O_e(w_1,w_2)\right)$}
\center
\begin{tabular}{lc|lc|lc|lc}
\hline
$(w_1,w_2)$ & card. & $(w_1,w_2)$ & card. & $(w_1,w_2)$ & card. & $(w_1,w_2)$ & card. \\
\hline
(13, 2) & 1 & (24, 7) & 2 & (29, 6) & 6 & (33, 0) & 1 \\
(14, 3) & 1 & (24, 9) & 2 & (29, 8) & 2 & (33, 2) & 6 \\
(15, 4) & 1 & (24, 13) & 2 & (29, 10) & 3 & (33, 4) & 4 \\
(16, 1) & 1 & (25, 2) & 4 & (29, 12) & 3 & (33, 6) & 6 \\
(16, 5) & 1 & (25, 4) & 2 & (29, 14) & 3 & (33, 8) & 3 \\
(17, 2) & 1 & (25, 6) & 2 & (29, 18) & 4 & (33, 10) & 6 \\
(17, 6) & 2 & (25, 8) & 2 & (30, 1) & 4 & (33, 12) & 3 \\
\hline
\end{tabular}
\end{table}

\begin{table}[H]
\caption{$\mathrm{card}\left(O_e(w_1,w_2,w_3,w_4)\right)$}
\center
\begin{tabular}{lc|lc|lc}
\hline
$(w_i)_i$ & card. & $(w_i)_i$ & card. & $(w_i)_i$ & card. \\
\hline
(12, 9, 5, 2) & 1 & (14, 12, 10, 2) & 1 & (15, 12, 7, 2) & 2 \\
(12, 10, 7, 1) & 1 & (14, 12, 10, 4) & 1 & (15, 12, 7, 4) & 3 \\
(13, 9, 5, 1) & 1 & (14, 13, 6, 1) & 1 & (15, 12, 8, 1) & 3 \\
(13, 9, 7, 3) & 1 & (14, 13, 7, 2) & 1 & (15, 12, 8, 3) & 2 \\
(13, 10, 5, 2) & 1 & (14, 13, 8, 1) & 2 & (15, 12, 8, 5) & 2 \\
(13, 10, 7, 4) & 1 & (14, 13, 9, 0) & 1 & (15, 12, 9, 2) & 4 \\
(13, 11, 5, 3) & 1 & (14, 13, 9, 4) & 1 & (15, 12, 9, 4) & 1 \\
\hline
\end{tabular}
\end{table}

\begin{table}[H]
\caption{$\mathrm{card}\left(O_e(w_1,w_2,w_3,w_4,w_5,w_6)\right)$}
\center
\begin{tabular}{lc|lc|lc}
\hline
$(w_i)_i$ & card. & $(w_i)_i$ & card. & $(w_i)_i$ & card. \\
\hline
(13, 11, 10, 6, 4, 3) & 1 & (14, 12, 10, 6, 3, 2) & 1 & (14, 12, 11, 9, 5, 0) & 1 \\
(13, 12, 8, 7, 4, 1) & 1 & (14, 12, 10, 6, 4, 1) & 2 & (14, 12, 11, 9, 5, 2) & 2 \\
(13, 12, 9, 6, 3, 2) & 1 & (14, 12, 10, 6, 4, 3) & 2 & (14, 12, 11, 9, 6, 1) & 1 \\
(13, 12, 9, 7, 5, 1) & 1 & (14, 12, 10, 6, 5, 2) & 1 & (14, 12, 11, 9, 6, 3) & 1 \\
(13, 12, 9, 8, 5, 2) & 1 & (14, 12, 10, 7, 3, 1) & 1 & (14, 12, 11, 9, 7, 2) & 2 \\
(13, 12, 10, 6, 4, 2) & 1 & (14, 12, 10, 7, 4, 0) & 1 & (14, 12, 11, 9, 7, 4) & 1 \\
(13, 12, 10, 7, 4, 1) & 1 & (14, 12, 10, 7, 4, 2) & 2 & (14, 12, 11, 9, 8, 3) & 1 \\
\hline
\end{tabular}
\end{table}

\newpage

\bibliographystyle{amsalpha}
\bibliography{dimtrace}

\end{document}